\documentclass[12pt]{amsart}
\usepackage[hmargin=1 in, vmargin = 1 in]{geometry}
\usepackage[obeyspaces,hyphens,spaces]{url}
\usepackage[hyperfootnotes=true]{hyperref}
\usepackage{microtype}
\usepackage[utf8]{inputenc}
\usepackage[initials]{amsrefs}
\usepackage{bm,amsmath,amssymb,amsthm,mathtools}
\usepackage{enumerate}
\usepackage{amsfonts}
\usepackage{mathrsfs} \usepackage[all]{xy} \usepackage{stmaryrd} \usepackage{pb-diagram}
\usepackage{dsfont}  \usepackage{upgreek,bm}
\usepackage{tikz}

\def\id{{\rm id}}
\def\ppi{p}
\def\bi{{\bf i}}
\def\a{{\bf a}}
\def\aD{D} \def\oPi{{\mathring{G/P}}}
\def\Hecke{{\rm Hecke}}
\def\oHecke{{\overline{\Hecke}}}
\def\oHk{{\overline{\Hk}}}
\def\Hk{{\rm Hk}}
\def\Aut{{\rm Aut}}
\def\AA{{\mathsf{E}}} \def\Exp{{\mathsf{E}}} \def\Mult{{\mathsf{M}}} 
\def\g{{\mathfrak g}}
\def\u{{\mathfrak u}}
\def\AS{{\rm AS}}
\def\af{{\rm af}}
\def\Lie{{\rm Lie}}
\def\Inv{{\rm Inv}}
\def\G{{\mathbb G}}
\def\Bun{{\rm Bun}}

\def\pr{{\rm pr}}
\def\ev{{\rm ev}}
\def\Gr{{\rm Gr}}
\def\oGr{{\overline{\Gr}}}

\def\Kl{{\rm Kl}} \def\TKl{{\rm TKl}} 

\def\tR{{\tilde R}}

\def\tiota{{\tilde \iota}}

\def\E{{\mathcal{E}}}
\def\O{{\mathcal{O}}}
\def\C{{\mathbb C}}
\def\Z{{\mathbb Z}}
\def\A{{\mathbb A}}

\def\P{{\mathbb P}}

\def\Gm{{\P^1\backslash \{0,\infty\}}}
\def\Trace{{\rm Tr}}
\newcommand\ip[1]{{\langle #1 \rangle}}
\def\bQ{{\overline{\mathbb Q}}}
\def\F{{\mathbb F}}

\def\oGP{\mathring{G/P}}
\def\t{{\mathfrak {t}}}
\def\u{{\mathfrak{u}}}
\def\opp{{\rm opp}}
\def\Op{{\rm Op}}

\def\cl{{\rm cl}}
\def\Hitch{{\rm Hitch}}
\def\c{{\mathfrak{c}}}
\def\p{{\mathfrak{p}}}
\def\b{{\mathfrak{b}}}
\def\Frac{{\rm Frac}}
\def\RS{{\rm RS}}
\def\Fun{{\rm Fun}}
\def\sl{{\mathfrak{sl}}}
\def\gr{{\rm gr}}
\def\cBun{\mathring{\rm Bun}}
\def\Aut{{\rm Aut}}
\def\pt{{\rm pt}}

\def\DR{{\rm DR}}
\def\i{{\mathfrak{i}}}

\def\MF{{\rm MF}}
\def\rM{{\rm M}}
\def\MG{{\rm MG}}
\def\deg{{\rm deg}}

\def\Tor{{\rm Tor}}

\def\qroot{\kappa}
\def\comp{\kappa}

\providecommand{\cR}{\mathcal{R}}
\providecommand{\ki}{\mathfrak{i}}
\providecommand{\cO}{\mathcal{O}}
\providecommand{\R}{\mathbb{R}}
\providecommand{\cQ}{\mathcal{Q}}
\providecommand{\kc}{\mathfrak{c}}
\providecommand{\cY}{\mathcal{Y}}
\providecommand{\cG}{\mathcal{G}}
\providecommand{\cM}{\mathcal{M}}
\providecommand{\cL}{\mathcal{L}}
\DeclareMathOperator {\Sym} {Sym}
\providecommand{\cE}{\mathcal{E}}
\providecommand{\cH}{\mathcal{H}}
\def\isom{\stackrel{\sim}{\ra}}
\DeclareMathOperator {\Sp} {Sp}

\def\hra{\hookrightarrow}
\DeclareMathOperator {\SO} {SO}
\DeclareMathOperator {\SL} {SL}
\DeclareMathOperator {\PGL} {PGL}
\providecommand{\kt}{\mathfrak{t}}
\providecommand{\kg}{\mathfrak{g}}
\providecommand{\Q}{\mathbb{Q}}
\def\ra{\rightarrow}
\DeclareMathOperator {\Frob} {Frob}

\newcommand{\BK}{\mathrm{Cr}} \newcommand{\WBK}{\mathrm{WCr}}  
\newcommand{\Jac}{\mathrm{Jac}}
       \newcommand{\kv}{\mathfrak{v}}
\DeclareMathOperator{\Spec}{Spec}

\newcommand{\GM}{\mathrm{GM}}
\newcommand{\PD}{\mathrm{PD}}
\newcommand\defn[1]{{\it #1}}

\usepackage{url}
\usepackage{docmute}
\usepackage{color,xcolor}

\usepackage{comment}

\newtheorem{theorem}[subsection]{Theorem}
\newtheorem{definition}[subsection]{Definition}
\newtheorem{lemma}[subsection]{Lemma}
\newtheorem{proposition}[subsection]{Proposition}
\newtheorem{corollary}[subsection]{Corollary}
\newtheorem{conjecture}[subsection]{Conjecture}
\theoremstyle{definition}
\newtheorem{remark}[subsection]{Remark}
\newtheorem{example}[subsection]{Example}
\newtheorem{question}[subsection]{Question}
\numberwithin{equation}{subsection}

\begin{document}
\title[Mirror conjecture for minuscule flag varieties]{The mirror conjecture for minuscule flag varieties}
\author[Thomas Lam]{Thomas Lam}
\address{Department of Mathematics, University of Michigan, 530 Church St., Ann Arbor MI 48109.}
\email{tfylam@umich.edu}
\author[Nicolas Templier]{Nicolas Templier}
\address{Department of Mathematics, Cornell University, Malott Hall, Garden Av., Ithaca NY 14853-4201.}
\email{npt27@cornell.edu}
\begin{abstract}
We prove Rietsch's mirror conjecture that the Dubrovin quantum connection for minuscule flag varieties is isomorphic to the character $\aD$-module of the
Berenstein--Kazhdan geometric crystal. The idea is to recognize the quantum connection 
as Galois and the geometric crystal as automorphic. We reveal surprising
relations with the works of Frenkel--Gross, Heinloth--Ng\^o--Yun, and Zhu on Kloosterman sheaves. The isomorphism comes from global rigidity results where Hecke eigensheaves are determined by their local ramification. As corollaries we obtain combinatorial identities for counts of rational curves and the Peterson variety presentation of the small quantum cohomology ring.
\end{abstract}

\maketitle

\tableofcontents

\section{Introduction}
Let $G$ be a complex almost simple algebraic group, $B \subset G$ a Borel subgroup, and 
$P\subset G$ a parabolic subgroup containing $B$.  Let $B^\vee \subset P^\vee \subset G^\vee$ denote 
the Langlands duals.  In the case that $P^\vee$ is a minuscule maximal parabolic subgroup, 
we prove the mirror theorem that the \defn{quantum connection} of the partial flag 
variety $G^\vee/P^\vee$ is isomorphic to the \defn{character $\aD$-module} of the geometric 
crystal associated to $(G,P)$.  This isomorphism is the top row of the diagram 
of $\aD$-modules of Figure~\ref{fig:four}, where the bottom row is an instance of the geometric Langlands 
program.
\begin{center}
\begin{figure}
\begin{tikzpicture}
\node at (0,1) {$B$-model};
\node at (10,1) {$A$-model};
\node at (10,-4) {Galois};
\node at (0,-4) {Automorphic};
\node[draw, align = center] (A) at (0,0) {character $\aD$-module of \\ geometric crystal for $(G,P)$};
\node[draw,align=center] (B) at (10,0) {quantum $\aD$-module \\ for $G^\vee/P^\vee$};
\node[draw,align=center] (C) at (0,-3) {Kloosterman $\aD$-module for \\ minuscule representation of $G^\vee$};
\node[draw,align=center] (D) at (10,-3) {Frenkel--Gross connection for \\ minuscule 
representation of $G^\vee$};
\draw (A) -- 
node [above, align=center] {Rietsch's mirror conjecture} (B) 
-- node[right,align=center] {\S\ref{s:FG}-\S\ref{s:qh} } (D) -- node 
[below, align=center] {Zhu's Theorem}(C) 
-- node[left,align=center] {\S\ref{s:crystal}-\S\ref{s:HNY}}
(A);
\end{tikzpicture}
\caption{The four $\aD$-modules in this work.}
\label{fig:four}
\end{figure}
\end{center}

In fact, our main result is stronger. It concerns the \emph{equivariant} quantum 
cohomology of $G^\vee/P^\vee$ and moreover adds a \emph{parameter} $\hbar\in \A^1$. This is 
Theorem~\ref{thm:Shbar}.
Specializing $\hbar=1$ is the equivariant version of the above mirror theorem which was 
conjectured by Rietsch~\cite{Rietsch:mirror-construction-QH-GmodP}, and taking the semiclassical limit
($\hbar\to 0$) yields the equivariant Peterson isomorphism which was stated in the as yet unpublished lectures of Peterson~\cite{Peterson:lectures}.

We now discuss the diagram of Figure~\ref{fig:four} in detail.

\subsection{Quantum cohomology and mirror symmetry for flag varieties}
The study of the topology of flag varieties $G^\vee/B^\vee$ has a storied history.  Borel 
found the cohomology rings $H^*(G^\vee/B^\vee,\C)$ to be isomorphic to the coinvariant 
algebras of the Weyl group $W$ acting on the natural reflection representations.  This 
result is continued by the works of Chevalley, Bernstein--Gelfand--Gelfand, Demazure, 
Lascoux--Sch\"utzenberger, and many others on the Schubert calculus of flag varieties.

Much progress was made on the quantum cohomology of flag varieties in the last two 
decades.  Givental and Kim \cite{Givental-Kim} and Ciocan-Fontanine 
\cite{Ciocan-Fontanine:QHflags} (for $G^\vee$ of type $A$), and Kim \cite{Kim:ann-QH-GmodB} 
(for general $G^\vee$) identified the quantum cohomology rings $QH^*(G^\vee/B^\vee,\C)$ with 
the ring of regular functions on the nilpotent leaf of the Toda lattice of $G$.  
Subsequently, Givental \cite{Givental:mirror} formulated a mirror conjecture that oscillatory integrals over a middle-dimensional cycle inside the mirror manifold should be solutions to the quantum $\aD$-module, and established this result for $G^\vee$ of type $A$ (see also \cite{Eguchi-Hori-Xiong}).  This mirror theorem was extended to general $G^\vee/B^\vee$ by Rietsch \cite{Rietsch:mirror-construction-QH-GmodP, Rietsch:mirror-solution-Toda}.  These oscillatory integrals gave new integral formulae for Whittaker functions.

By contrast, our understanding of mirror symmetry for partial flag varieties $G^\vee/P^\vee$ is much more limited.  Peterson \cite{Peterson:lectures} discovered a uniform geometric description of the quantum cohomology rings $QH^*(G^\vee/P^\vee,\C)$, but this work remains unpublished (see however \cite{Cheong:QH-LG-OG,Rietsch:mirror-construction-QH-GmodP, Rietsch:JAMS, Lam-Shimozono:GmodP-affine}).  The quantum $\aD$-modules of $G^\vee/P^\vee$ have remained largely unstudied in full generality.  Batyrev--Ciocan-Fontanine--Kim--van Straten~\cite{BCKS:acta} proposed a mirror conjecture for $\SL(n)/P^\vee$, and Rietsch formulated a mirror conjecture for arbitrary $G^\vee/P^\vee$, in the style of Givental.

One of the main aims of this work is to establish Rietsch's mirror conjecture in the case that $P^\vee$ is minuscule (see~\S\ref{sub:minuscule}).  This class of spaces includes projective
spaces, Grassmannians, and orthogonal Grassmannians (see Figure \ref{fig:minuscule} for the full list).  Even for the case of Grassmannians, whose quantum
cohomology rings are well studied \cite{Siebert-Tian:QH,Witten:Verlinde,Bertram:QH} and a large part of the mirror conjecture established in \cite{Marsh-Rietsch:B-model-Grassmannians}, our results are new.

\subsection{Small quantum $\aD$-module} 
We now let $P^\vee$ be a minuscule (maximal) parabolic subgroup.  The small quantum cohomology ring $QH^*(G^\vee/P^\vee)$\footnote{In this paper, cohomologies and quantum cohomologies are all taken with $\C$ coefficients.} 
is isomorphic to $\C[q,q^{-1}] \otimes H^*(G^\vee/P^\vee)$ as a vector space, with quantum multiplication denoted by $*_q$.  

Let $\C_q^\times = \Spec(\C[q,q^{-1}])$ be the one-dimensional torus with coordinate $q$.  The small quantum $\aD$-module (at $\hbar = 1$) \cite{Dubrovin:geometry-TFT} is the connection on the trivial $H^*(G^\vee/P^\vee)$--bundle over $\C_q^\times$ given by
\begin{equation}\label{eq:smallquantum}
\cQ^{G^\vee/P^\vee} := d + (\sigma*_q)\frac{dq}{q},
\end{equation}
where $\sigma \in H^2(G^\vee/P^\vee,\Z)$ is the effective divisor class, and we consider
\begin{equation}\label{i:chevalley} 
\sigma *_q \in \mathrm{End}\left( H^*(G^\vee/P^\vee) \right) \otimes \C[q].
\end{equation}

In~\cite{Chevalley:decompositions-cellulaires}, Chevalley gave a combinatorial formula for 
the cup-product in $H^*(G^\vee/P^\vee)$ with the divisor class $\sigma$, that is 
for~\eqref{i:chevalley} at $q=0$.  A quantum Chevalley formula (see Theorem \ref{thm:FW}) 
evaluating $\eqref{i:chevalley}$ in terms of Schubert classes for general flag varieties was 
stated by Peterson \cite{Peterson:lectures} and proved by Fulton and Woodward~\cite{FW}.  
This formula has been extended to the equivariant case by 
Mihalcea~\cite{Mihalcea:EQ-Chevalley} and to the cotangent bundle of partial flag 
varieties by Su~\cite{Su:equivariant-QH}. For recent developments in the minuscule case, 
see~\cite{BCMP:Chevalley-K-cominuscule}.

In the sequel, we also let $\cQ^{G^\vee/P^\vee}$ denote the corresponding algebraic $\aD$-module, where $\aD=\aD_{\C^\times_q} = \C[q,q^{-1}]\langle \partial_q \rangle$ is the ring of differential operators on $\C^\times_q$, and $\partial_q := \frac{d}{dq}$.

\subsection{The character $\aD$-module of a geometric crystal}\label{ssec:BKintro}
Berenstein--Kazhdan~\cite{BK}, based on previous works by Lusztig and of 
Berenstein--Zelevinsky~\cite{Berenstein-Zelevinsky:tensor-product-multiplicities}, have 
constructed geometric crystals which are certain complex algebraic varieties equipped 
with rational maps. The motivation of the construction was the birational lifting of the 
combinatorics of Lusztig's canonical bases and Kashiwara's 
crystal bases.

Fix opposite Borel subgroups $B$ and $B_-$ of $G$ with unipotent subgroups $U$ and $U_-$, and let $T = B \cap B_-$.  Let $R$ denote the root system, and $R^\pm$ denote the subsets of positive and negative roots. Let $\psi:U\to \G_a$ be a {\it nondegenerate} character in the sense that  \(\psi\) is nontrivial on every simple root space when composed with the exponential isomorphism  \(\bigoplus_{i\in I}\A^1 \cong U/[U,U]\).

For a parabolic subgroup $P\subset G$, let $W_P \subset W$ be the Weyl group of the Levi subgroup $L_P$, and let $I_P \subset I$ be the corresponding subset of the Dynkin diagram.  There is a unique set $W^P\subset W$ of minimal length coset representatives for the quotient $W/W_P$.  Define $w_P^{-1} \in W$ to be the longest element in $W^P$.  The \defn{(parabolic) geometric crystal} $X = X_{(G,P)}$ is the smooth subvariety 
$$
X = U Z(L_P) \dot w_P U \cap B_- \subset G,
$$
where $Z(L_P)$ denotes the center of the Levi subgroup $L_P$, and $\dot w_P \in G$ is a representative of $w_P \in W$, equipped~\cite[\S2.2]{BK} with geometric crystal actions $\G_m \times X \to X$ (which are rational maps, defined on a dense open subset) and three regular maps of importance to us:
\begin{align*}
f: X &\to \A^1, \qquad u_1 t \dot w_P u_2 \mapsto \psi(u_1) + \psi(u_2)& \mbox{called the decoration function,} \\
\gamma: X & \to T, \qquad x \mapsto x \mod U_- \in B_-/U_- \cong T & \mbox{called the weight function,}\\
\pi: X &\to Z(L_P), \qquad u_1 t \dot w_P u_2 \mapsto t & \mbox{called the highest weight function.}
\end{align*}
The fiber $X_t := \pi^{-1}(t)$ for $t \in Z(L_P)$ is called the \defn{geometric crystal with 
highest weight $t$}.  For any $t \in Z(L_P)$, $X_t$ is a log Calabi--Yau variety isomorphic 
to 
the open projected Richardson variety $\oGP \subset G/P$ \cite{KLS:projections-Richardson}, the 
complement in $G/P$ of a particular anticanonical divisor $\partial_{G/P}$.  (The boundary 
divisor $\partial_{G/P}$ is not normal crossing in general.  There is a Bott--Samelson resolution that 
provides an explicit compactification of $\oGP$ with normal crossing divisors; see 
\cite[Appendix]{KLS:projections-Richardson}, \cite[\S4.2]{Kumar-Schwede}.)
The affine variety $\oGP$ has a distinguished holomorphic volume form $\omega$ (\S\ref{s:Richardson}), with simple poles along the boundary divisor $\partial_{G/P}$.

On $\A^1$ we consider the cyclic $\aD$-module $\Exp:=\aD_{\A^1}/\aD_{\A^1}(\partial_x-1)$ with generator the exponential function, where $\aD_{\A^1}=\C[x]\langle \partial_x \rangle$. The pullback $f^* \Exp$ is a $\aD$-module on $X$. Note that one can identify the $\aD$-module \(\Exp\) with the connection  \(d-dx\) on the trivial line bundle on  \(\A^1\). The pullback  \(f^*\Exp\) can be identified with the connection  \(d-df\) on the trivial line bundle on  \(X\). We define the \defn{character $\aD$-module} of the geometric crystal $X$ by 
\begin{equation} \label{eq:BK}
\BK_{(G,P)}:= R\pi_* f^* \Exp,
\end{equation}
which is a $\aD$-module on $Z(L_P)$.  
For  \(P=B\), the geometric crystal $X$ tropicalizes to Kashiwara's combinatorial crystals, see~\cite[\S6]{BK}. As 
explained in Lam~\cite{Lam:Whittaker} and Chhaibi~\cite{Chhaibi:Whittaker-processes}, the tropicalization of \eqref{eq:BK} is an irreducible character of $G^\vee$.

A priori $\BK_{(G,P)}$ is a complex of $\aD$-modules, but we show in Theorem~\ref{tintro:HNY} (= Theorem~\ref{t:HNYisBK}) that it is just a $\aD$-module.  
Our proof is via the left-hand side of Figure~\ref{fig:four}, which enables us to recognize this statement as the Ramanujan property, in the context of the geometric Langlands program, for a certain cuspidal automorphic $D$-module $A_\cG$; see \S\ref{intro:Kloosterman} below and~\cite[Thm.1]{HNY}.

This article seems to be the first time the properties of the character $\aD$-module $\BK_{(G,P)}$ are studied.   There are other geometric crystals, and as we shall see below, other families of Landau--Ginzburg models that one could apply this construction to.
We also note that automorphic $D$-modules with wild ramification, and geometric analogues of Arthur conjectures, which both play an important role in our study, are themes which have been largely unexplored at the present time.

\subsection{Rietsch's Landau--Ginzburg model}
In~\cite{Rietsch:mirror-construction-QH-GmodP}, Rietsch constructed conjectural Landau--Ginzburg mirror partners of all partial flag varieties $G/P$.  Her
construction was motivated by earlier works of Givental \cite{Givental:mirror}, Joe--Kim~\cite{Joe-Kim:equivariant}, and Batyrev--Ciocan-Fontanine--Kim--van Straten~\cite{BCKS:acta} for type $A$ flag varieties, and also by the Peterson presentation of $QH^*(G^\vee/P^\vee)$.

Rietsch's mirror construction are families of varieties fibered over $q \in \Spec(\C[q_i^{\pm 1} \mid i \notin I_P])$, equipped with holomorphic superpotentials $f_q$, and holomorphic volume forms $\omega_q$.

It was observed by Lam~\cite{Lam:Whittaker}, and Chhaibi~\cite{Chhaibi:Whittaker-processes} that Rietsch's mirror construction could be obtained from the group geometry of geometric crystals.  Thus, after identifying $\Spec(\C[q_i^{\pm 1} \mid i \notin I_P])$ with $Z(L_P)$, Rietsch's mirror family can be identified with the highest weight function $\pi: X \to Z(L_P)$ of \S\ref{ssec:BKintro}, and Rietsch's superpotential becomes the decoration function $f_t := f|_{X_t}: X_t \to \A^1$; henceforth we will use $f_q$ or $f_t$ interchangeably ($q$ being a point in $\Spec(\C[q_i^{\pm 1} \mid i \notin I_P])$ and $t$ a point in $Z(L_P)$).

Earlier mirror Landau--Ginzburg models for various partial flag varieties~\cite{Givental:mirror,BCKS:acta,Eguchi-Hori-Xiong} were Laurent polynomial superpotentials defined on an algebraic torus.  These Landau--Ginzburg models arose from toric degenerations of $G^\vee/P^\vee$.  In contrast, Rietsch's candidate mirror Landau--Ginzburg model is defined on a partial compactification of a torus, and is intrinsically related to the group geometry of $G/P$ (and not to any toric degeneration); see~\cite{Pech-Rietsch-Williams:LG-quadrics,Pech-Rietsch:LG-Lagrangian-Grassmannians}.  In the literature this distinction also appears in the form of ``strong mirror'' versus ``weak mirror''.  

Stated informally our main goal in this paper is to show:
\begin{center} 
\it
If $P^\vee$ is minuscule then $G^\vee/P^\vee$ and $(\stackrel{\circ}{G/P},f_q)$ form a Fano type mirror pair.
\end{center}
On the $A$-model side $G^\vee/P^\vee$ is a projective Fano variety, and on the $B$-model 
side $\stackrel{\circ}{G/P}$ is a log Calabi--Yau variety; 
see~\cite{Katzarkov-Kontsevich-Pantev:LG-models} for general expectations for Fano type 
mirror pairs.  We show that some of the mirror symmetry expectations hold.

One expectation is a relationship between the Gromov--Witten invariants of 
$G^\vee/P^\vee$ and the coefficients of the Laurent series expansion of the potential $f_q$ 
restricted to a torus of $\oGP$.
If $P^\vee$ is minuscule, we shall establish such a relationship in the form of an integral 
representation 
of 
the quantum period of 
$G^\vee/P^\vee$ which was previously known for 
Grassmannians~\cite{Marsh-Rietsch:B-model-Grassmannians} and for 
quadrics~\cite{Pech-Rietsch-Williams:LG-quadrics}; see Section~\ref{s:parabolic-bessel} for 
details.

\subsection{The mirror isomorphism}
The following is a simple version of the main result of this paper and establishes the top 
row of Figure \ref{fig:four}.

\begin{theorem}[Theorem~\ref{t:Dmodule-mirror}]\label{tintro:Dmodule-mirror}
Suppose $P^\vee$ is minuscule.
The geometric crystal $\aD$-module $\BK_{(G,P)}$  and the quantum cohomology $\aD$-module $\cQ^{G^\vee/P^\vee}$ for $G^\vee/P^\vee$ are isomorphic.
\end{theorem}
For $G^\vee/P^\vee$ a projective space $\P^n$ the result is well-known~\cite{Givental:equivariant-GW,Katz:exp-sums-diff-eq}.
The homological mirror symmetry version is established in~\cite{Abouzaid:hms-toric,FLTZ:T-duality-mirror-toric}.
Our approach gives an original perspective in terms of hyper-Kloosterman sheaves studied in SGA4$\frac12$~\cite{Deligne:SGA4h}.

For $G^\vee/P^\vee$ a Grassmanian $\mathrm{Gr}(k,n)$, the result is already new. 
Partial results are obtained by 
Marsh--Rietsch~\cite{Marsh-Rietsch:B-model-Grassmannians}, notably a canonical 
injection of $\cQ^{G^\vee/P^\vee}$ into
$\BK_{(G,P)}$, who establish as a consequence a conjecture of
Batyrev--Ciocan-Fontanine--Kim--Van Straten~\cite[Conj.5.2.3]{BCKS:conifold}. Our 
Theorem~\ref{tintro:Dmodule-mirror} is stronger.  Indeed, it establishes the
conjecture of~\cite[\S3]{Marsh-Rietsch:B-model-Grassmannians} that the canonical injection is bijective, and thereby also another conjecture of
Batyrev--Ciocan-Fontanine--Kim--Van Straten~\cite[Conj.5.1.1]{BCKS:acta}. 

For $G^\vee/P^\vee$ an even-dimensional quadric, the injection of $\cQ^{G^\vee/P^\vee}$ into 
$\BK_{(G,P)}$ is obtained by 
Pech--Rietsch--Williams~\cite{Pech-Rietsch-Williams:LG-quadrics}, and
our Theorem~\ref{tintro:Dmodule-mirror} establishes a conjecture of~\cite[\S4]{Pech-Rietsch-Williams:LG-quadrics}. 

Although both sides of Theorem~\ref{tintro:Dmodule-mirror} are described explicitly, this 
does not lead to a way of establishing the isomorphism. Indeed our proof will follow a 
lengthy path, where the isomorphism will eventually arise from Langlands reciprocity for 
the automorphic  \(D\)-module $A_\cG$ of \cite[\S2.5]{HNY} over the rational function field $\C(t)$. 

While the Langlands program has integrated for a long time adjacent areas of 
mathematics into the 
solution of some of its conjectures, in the other direction, there are yet rather few 
applications of the Langlands program to problems outside of number theory. 
Interestingly, in this paper we shall apply recent advances in the geometric Langlands 
program to establish results on the 
geometry of flag varieties.

\subsection{Kloosterman sums, Kloosterman sheaves, and Kloosterman 
$\aD$-modules}\label{intro:Kloosterman}
For a prime $p$ and a finite field $\F_q$, $q = p^m$, define the two maps
\begin{align*}
f: (\F^\times_q)^n &\to \F_q \qquad (x_1,x_2,\ldots,x_n) \mapsto  x_1 + x_2 + \cdots + x_n, \\
\pi: (\F^\times_q)^n &\to \F^\times_q \qquad (x_1,x_2,\ldots,x_n) \mapsto  x_1  x_2 \cdots x_n.
\end{align*}
The \defn{(hyper)Kloosterman sum} in $(n-1)$-variables is
\begin{equation}\label{eq:kloostermansum}
\Kl_n(a;q) := (-1)^{n-1} \sum_{x \in \pi^{-1}(a)} \exp\Bigl(\frac{2\pi i}{p} \Trace_{\F_q/\F_p}f(x)\Bigr),
\end{equation}
where $a \in \F_q^\times$.
Deligne \cite{Deligne:SGA4h} defines the \defn{(hyper)Kloosterman sheaf} to be the $\ell$-adic sheaf on $\F_p^\times$ given by
\begin{equation}\label{eq:kloostermansheaf}
\Kl^{\bQ_\ell}_n := R \pi_! f^* \AS_\psi[n-1],
\end{equation}
where $\AS_\psi$ is the Artin--Schreier sheaf on $\A^1$ corresponding to a nontrivial 
character $\psi: \F_p \to \bQ_\ell$.  For an appropriate
embedding $\iota: \bQ_\ell \to \C$, the Kloosterman sum \eqref{eq:kloostermansum} is identified as the Frobenius trace of the Kloosterman sheaf ;q
\eqref{eq:kloostermansheaf}: $\Kl_n(a;q) = \iota( \Trace(\Frob_a,\Kl^{\bQ_\ell}_n))$.  The \defn{Kloosterman 
$\aD$-module} is defined~\cite{Katz:exp-sums-diff-eq} by replacing the Artin--Schreier 
sheaf with the exponential $\aD$-module:
\begin{equation}\label{eq:kloostermanD}
\Kl_n:= R\pi_! f^* \Exp.
\end{equation}
The pair $(\pi:(\F^\times_q)^n \to \F^\times_q,f)$ and \eqref{eq:kloostermansheaf} should be compared with the geometric crystal mirror family $(\pi: X \to Z(L_P),f)$ and \eqref{eq:BK}.

Heinloth--Ng\^o--Yun~\cite{HNY} generalize Kloosterman sheaves and $\aD$-modules.  More 
precisely, for a representation $V$ of $G^\vee$, they define a generalized Kloosterman 
$\aD$-module $\Kl_{(G^\vee,V)}$ on $\C^\times$.  Their construction uses the moduli stack $\Bun_\cG$ 
of $\cG$-bundles on $\P^1$, where $\cG$ is a particular nonconstant group scheme over $\P^1$ 
(see \S \ref{s:HNY}); equivalently, $\Bun_\cG$ classifies $G$-bundles with specified ramification 
behavior.  Heinloth--Ng\^o--Yun construct an automorphic Hecke eigen-$\aD$-module 
$A_\cG$ on the Hecke stack over $\Bun_\cG$.  The generalized Kloosterman $\aD$-module 
$\Kl_{(G^\vee,V)}$ is defined to be the Hecke eigenvalue of $A_\cG$.  The projection and 
superpotential maps $\pi$ and $f$ are replaced in this setting by the projection maps of the 
Hecke moduli stack.  

A remarkable feature of the automorphic $\aD$-module $A_\cG$ is that it is \emph{rigid}: it 
can be characterised uniquely by its local components.  Indeed, the
existence of the rigid local systems constructed by Heinloth--Ng\^o--Yun was predicted 
by Gross, who constructed the trace function of $A_{\cG}$ over finite fields via the 
stable trace 
formula. 
We refer to~\cite{Yun:rigid} for a comprehensive survey of rigid automorphic forms.

The following result gives an automorphic interpretation of geometric crystals.

\begin{theorem}[Theorem~\ref{t:HNYisBK}]\label{tintro:HNY}
Let $P\subset G$ be a cominuscule parabolic and $V$ be the corresponding minuscule 
representation of $G^\vee$. The character
$\aD$-module $\BK_{(G,P)}$ is isomorphic to the Kloosterman $\aD$-module $\Kl_{(G^\vee,V)}$ 
defined as the $V$-Hecke eigenvalue of the automorphic $\aD$-module $A_\cG$.
\end{theorem}

The proof of Theorem \ref{tintro:HNY} is by a comparison of the geometry of the Hecke 
moduli stack and that of parabolic geometric crystals.

The above suggests a new parallel between exponential
sums over finite fields and Landau--Ginzburg models.  
Recall that \emph{arithmetic mirror symmetry} has been studied for mirror Calabi--Yau 
varieties~\cite{Wan:variation-padic-Newton}*{\S3}. 
The present work leads us to suggest that 
arithmetic mirror symmetry could 
conjecturally extend to Fano varieties and their mirror Landau--Ginzburg models.
Although we do not pursue this direction in the present paper, we observe for example the precise compatibility between the recent conjecture of
Katzarkov--Kontsevich--Pantev~\cite{Katzarkov-Kontsevich-Pantev:LG-models}*{(3.1.5)}, specialized to $G^\vee/P^\vee=\P^n$, and the classical theorem of Sperber~\cite{Sperber:hyperKloosterman-congruences} on the Newton polygon of the
Kloosterman sums $\Kl_n(a;q)$.  We remark that the Hodge theory of Kloosterman connections was studied in \cite{FSY}.

We believe that the purity property of an exponential sum (viewed from a 
mirror Landau--Ginzburg model perspective) should mirror the Hodge--Tate property of 
the cohomology of a Fano variety. 
In the context of Theorem \ref{tintro:HNY}, the Kloosterman sum $\Kl_{(G^\vee,V)}$ is 
pointwise pure~\cite{HNY}, and the partial flag variety $G/P$ has cohomology of 
Hodge--Tate type. We speculate that the slope multiplicities of the Newton polygon of $\Kl_{(G^\vee,V)}$ should 
coincide with the Betti numbers $\dim H^i(G/P)$ (which are well-known). The same 
construction 
applied to 
other mirror families, see for example \cite{CCGGK}, produces new $\ell$-adic sheaves and overconvergent  \(F\)-isocrystals.

\subsection{Frenkel--Gross rigid connection}
In \cite{Frenkel-Gross}, Frenkel--Gross study a rigid irregular connection on the trivial $G^\vee$-bundle on $\P^1$ given by the formula
\begin{equation}\label{eq:FG}
\nabla^{G^\vee} :=  d + y_p \frac{dq}{q} + x_\theta dq,
\end{equation}
where $y_p \in \g^\vee = \Lie(G^\vee)$ is a principal nilpotent, and $x_\theta \in \g^\vee_\theta$ lives in 
the highest root space.  For any $G^\vee$-representation $V$, we have an associated 
connection $\nabla^{(G^\vee,V)}$.

When $V$ is the minuscule representation of $G^\vee$ corresponding to parabolic $P^\vee$, we have a natural isomorphism $L: H^*(G^\vee/P^\vee) \cong V$.  

\begin{theorem}[Theorem~\ref{t:FGisquantum}]\label{tintro:FG}
Under the isomorphism $L:H^*(G^\vee/P^\vee) \cong V$, the quantum
connection $\cQ^{G^\vee/P^\vee}$ is isomorphic to the connection $\nabla^{(G^\vee,V)}$.
\end{theorem}

The isomorphism $L$ sends the Schubert basis of $H^*(G^\vee/P^\vee)$ to the canonical basis 
of $V$.  The proof of Theorem \ref{tintro:FG} is via a direct comparison of the 
Frenkel--Gross connection in the canonical basis with the quantum Chevalley 
formula.

\subsection{Zhu's theorem}
Beilinson--Drinfeld  have introduced a class of connections called \defn{opers}, extending 
earlier work of Drinfeld and Sokolov.  They use opers to construct (part of) the Galois to 
automorphic direction of the geometric Langlands correspondence.  Frenkel--Gross 
\cite{Frenkel-Gross} have observed that \eqref{eq:FG} can be put into oper form after a 
gauge transformation.

Zhu \cite{Zhu:FG-HNY} has extended Beilinson--Drinfeld's construction to allow certain 
nonconstant group schemes, or equivalently to allow specified ramifications.  He thereby 
confirms the conjecture of \cite{HNY} that the Kloosterman $\aD$-module 
$\Kl_{G^\vee}$ is isomorphic to the Frenkel--Gross connection 
$\nabla^{G^\vee}$.

Theorem \ref{tintro:Dmodule-mirror} is obtained by composing the isomorphisms of 
Theorems \ref{tintro:HNY}, \ref{tintro:FG}, and Zhu's theorem.

This concludes our discussion of Figure~\ref{fig:four}. We continue this introduction by 
explaining the stronger Theorem~\ref{intro:hbar} and 
how it relates to Rietsch's equivariant mirror conjecture and to the equivariant Peterson  
isomorphism.

\subsection{Equivariant quantum cohomology and weighted geometric crystals}
We may replace the quantum cohomology ring $QH^*(G^\vee/P^\vee)$ by the 
$T^\vee$-equivariant quantum cohomology ring $QH^*_{T^\vee}(G^\vee/P^\vee)$.
We briefly discuss the new features.

  The corresponding equivariant quantum connection $\cQ^{G^\vee/P^\vee}_{T^\vee}$ is a 
  connection over 
$\C^\times_q\times \t^*$ relative to $\t^*$, where $\t = \Lie(T)$ and instead of $\sigma *_q$ 
in \eqref{eq:smallquantum}, we have the operator $c_1^T(O(1)) \; *^{T^\vee}_{q}$ of 
equivariant quantum multiplication in $QH^*_{T^\vee}(G^\vee/P^\vee)$ (see \S 
\ref{S:equivariant}).  
Identifying $S = H^*_{T^\vee}(\pt)$ with the coordinate  \(\C\)-algebra $\Sym(\t) \cong \C[\t^*]$, we may 
equivalently consider the family of connections $\cQ^{G^\vee/P^\vee}_{T^\vee} \otimes_h \C$ 
indexed by 
$h \in \t^*$ viewed as algebra morphism $h:S\to \C$.  These specialize again to connections on 
the 
trivial $H^*(G^\vee/P^\vee)$--bundle over $\C^\times_q$.

Let us now discuss the {\it weighted character $\aD$-module} of the geometric crystal $X$, 
equipped with the weight function $\gamma:X \to T$.  
The character $\aD$-module \eqref{eq:BK} can be weighted with a parameter $h\in \t^*$, to give
\begin{equation}\label{eq:equivariantBK}
\WBK_{(G,P)} :=
R\pi_* (\gamma^*\Mult_T \otimes f^* \Exp),
\end{equation}
where $\Mult_T$ denotes the $D\otimes S$-module on $T$ defined in~\S\ref{s:Gross-form}.

The $\aD$-module version of Rietsch's equivariant mirror 
conjecture~\cite{Rietsch:mirror-construction-QH-GmodP} states that there is an 
isomorphism
\[
\WBK_{(G,P)} \simeq \cQ^{G^\vee/P^\vee}_{T^\vee}.
\]

\subsection{A deformation of Zhu's theorem}
The {\it twisted Kloosterman $D$-module} $\TKl_n$ arises in 
\eqref{eq:kloostermansheaf} by replacing the Artin--Schreier sheaf by the tensor 
product of an Artin--Schreier sheaf and a Kummer sheaf.
The automorphic $\aD$-module $A_\cG$ can similarly be deformed to an automorphic 
$\aD$-module 
which further depends on the choice of a character of $T$~\cite{HNY}.
We thus obtain the twisted Kloosterman $\aD$-module, denoted $\TKl_{G^\vee}$.  

The corresponding deformation of the Frenkel--Gross connection \eqref{eq:FG} has not appeared in the 
literature as far as we know.  We define the \emph{deformed Frenkel--Gross connection} 
to be
\begin{equation}\label{eq:equivariantFG}
\widetilde \nabla^{G^\vee} :=  d + (y_p + h) \frac{dq}{q} + x_\theta dq,
\end{equation}
for $h \in \t^\vee$ an element of the Cartan subalgebra.

With these modifications, Figure \ref{fig:four} and Theorems \ref{tintro:Dmodule-mirror}, 
\ref{tintro:HNY}, and \ref{tintro:FG} all hold with their equivariant and deformed 
counterparts.  In particular, we deduce Rietsch's equivariant mirror conjecture for minuscule flag varieties.  With 
some mild variation, we also generalize Zhu's theorem to the deformed setting: the 
twisted Kloosterman $D$-module $\TKl_{G^\vee}$ and the deformed Frenkel--Gross connection 
$\widetilde \nabla^{G^\vee}$ are isomorphic.

\subsection{$\aD_{\hbar}$--module generalization}  The passage from the quantum connection $\cQ^{G^\vee/P^\vee}_{T^\vee}$ to the quantum cohomology ring $QH^*_{T^\vee}(G^\vee/P^\vee)$ itself is obtained by taking a semiclassical limit.  A framework to rigorously formalize the semiclassical limit is to extend the 
mirror theorem to an 
isomorphism of $\aD_\hbar$-modules, where
$\aD_\hbar:=\C[q,q^{-1},\hbar]\langle \hbar \partial_q \rangle$ and $\hbar$ is an additional parameter.  In~\eqref{eq:equivariantBK}, the $\aD$-module $\Exp$ is replaced by the $\aD_\hbar$-module 
$\Exp^{1/\hbar}$ defined in \S\ref{ssec:twisted-Dhbar}, the  \(D\otimes S\)-module  \(\Mult_T\) is replaced by the  \(\aD_{\hbar}\otimes S\)-module  \(\Mult^{1/\hbar}\) defined in \S\ref{s:S-structures}, and we obtain the 
$D_{\hbar}\otimes S$-module
\begin{equation}\label{eq:hequivariantBK}
\WBK^{1/\hbar}_{(G,P)} :=
R\pi_* (\gamma^* \Mult_T^{1/\hbar}\otimes f^* \Exp^{1/\hbar}).
\end{equation}
  Similarly, one obtains an $\hbar$-deformation $\TKl^{1/\hbar}_{G^\vee}$ 
of the twisted Kloosterman $D$-module.  The equivariant quantum connection and 
deformed Frenkel--Gross connection \eqref{eq:equivariantFG} also possess a further 
$\hbar$-deformation
\begin{align*}
\cQ^{G^\vee/P^\vee}_{\hbar,T^\vee}&:= \hbar d + c_1^T(O(1)) *_q \frac{dq}{q}, \\
\widetilde \nabla^G_\hbar &:=  \hbar d + (y_p + h) \frac{dq}{q} + x_\theta dq.
\end{align*}
These two formulas define \(\hbar\)-connections over $\C_q^\times\times \t^*$ relative to $\t^*$, hence in particular \(D_\hbar\otimes S\)-modules.

We are now able to state the main result of this paper, which includes the equivariant $\hbar$-mirror isomorphism.

\begin{theorem} [Theorems~\ref{t:hEFGisquantum}, \ref{t:eqzhu}, \ref{t:hEHNYisBK}, \ref{thm:Shbar}] \label{intro:hbar}
Suppose $P^\vee$ is minuscule.  The four $\aD_{\hbar}\otimes S$-modules 
$\WBK^{1/\hbar}_{(G,P)}$, $\widetilde \nabla^{G^\vee}_\hbar$, $\cQ^{G^\vee/P^\vee}_{\hbar,T^\vee}$, and
$\TKl^{1/\hbar}_{G^\vee}$ are isomorphic.
\end{theorem}
Specializing both $h=0\in \kt^*$ and $\hbar=1$ yields the earlier mirror  
Theorem~\ref{t:Dmodule-mirror}.
 Specializing $\hbar=0$ establishes the equivariant Peterson 
isomorphism, which we explain in the next subsection.  Thereom~\ref{intro:hbar} is obtained by exploiting the grading of the quantum product on one side, and the
homogeneity of the potential $f_q$ on the other side.

\subsection{Application: proof of the Peterson isomorphism}
Given a regular function on an algebraic variety, one can consider the sheaf of Jacobian ideals generated by all the first-order derivatives. Its quotient
ring defines a subscheme, possibly nonreduced, of critical points of the function.
Since $\oGP$ is affine, applying this construction to the weighted potential, we obtain a 
(relative) Jacobian ring $\mathrm{Jac}(\mathring{G/P},f_q,\gamma)$, which has the structure of a 
$\C[\t^*,q,q^{-1}]$-algebra.

\begin{theorem}[Theorem \ref{t:Jac}]\label{intro:hm}
If $P^\vee$ is minuscule, then we have an isomorphism of $\C[\t^*,q,q^{-1}]$-algebras
$
QH_{T^\vee}^*(G^\vee/P^\vee) \cong \mathrm{Jac}(\mathring{G/P},f_q,\gamma).
$
\end{theorem}
Specializing to non-equivariant cohomology, we obtain the mirror isomorphism of $\C[q,q^{-1}]$-algebras $QH^*(G^\vee/P^\vee) \cong \mathrm{Jac}(\mathring{G/P},f_q)$.
The same isomorphism is expected to hold for every Fano mirror dual pair.
It is established for (possibly orbifold) toric Fano varieties in
\cite{Bat,FOOO:LF-compact-toric-I,Ostrover-Tyomkin:toric-Fano,Coates-Corti-Iritani-Tseng,Gonzalez-Woodward:QH-toric-minimal}.

The equivariant Peterson variety $\cY$ is the closed subvariety of $G/B_- \times \t^*$ defined by
$$
\cY:= \{ (gB_-, h) \in (G/B_-) \times \t^* \mid g^{-1} \cdot (F - h) \text{ vanishes on 
$[\u_-,\u_-]$ } \},
$$
where $F \in \g^*$ is a principal nilpotent defined  as in \eqref{eq:FG}, see~\cite[\S3.2]{Rietsch:mirror-construction-QH-GmodP},
and
$\u_- := \Lie(U_-)$, and $g^{-1} 
\cdot (-)$ denotes the co-adjoint action.  It contains an open subscheme 
$$
\cY^*:= \cY \cap B_- w_0 B_- /B_-,$$
obtained by intersecting with the open Schubert cell $B_- w_0 B_- /B_-$, where $w_0$ denotes the longest element of $W$.
The intersection of $\cY^*$ with the opposite Schubert
stratification $\{B w B_-/B_-\}$ gives the $2^{\mathrm{rk}(\kg)}$ strata
\begin{equation}\label{eq:Peterson}
\cY^*_P:= \cY^* \cap B w_0^P B_- /B_-,
\end{equation}
where $w_0^P$ is the longest element of $W_P \subset W$ and the intersections are to be taken scheme-theoretically.  In \cite{Peterson:lectures}, Peterson announced the isomorphism $\cY^*_P \cong \Spec(QH^*_{T^\vee}(G^\vee/P^\vee))$.

Rietsch~\cite{Rietsch:mirror-construction-QH-GmodP} has proved that $\mathrm{Jac}(\mathring{G/P},f_q,\gamma)$ is isomorphic to $\C[\cY^*_P]$ as 
$\C[\t^*,q,q^{-1}]$-algebras.  We thus obtain:

\begin{corollary} [Equivariant Peterson isomorphism -- Corollary \ref{cor:Peterson}]\label{intro:Peterson}
If $P^\vee$ is minuscule, then we have an isomorphism of $\C[\t^*,q,q^{-1}]$-algebras
$
QH^*_{T^\vee}(G^\vee/P^\vee) \cong \C[\cY^*_P].
$
\end{corollary}
The Peterson isomorphism has been established directly for Grassmannians by Rietsch~\cite{Rietsch:QH-grassmannian-positivity}, for quadrics by
Pech--Rietsch--Williams~\cite{Pech-Rietsch-Williams:LG-quadrics}, and for Lagrangian and Orthogonal Grassmannians by Cheong~\cite{Cheong:QH-LG-OG}, all in the non-equivariant
case (that is, specializing $h \in \t^*$ to 0). In the equivariant case, the results of~\cite{Rietsch:QH-grassmannian-positivity}
and~\cite{Marsh-Rietsch:B-model-Grassmannians} can be combined to also obtain Corollary \ref{intro:Peterson} for Grassmannians,
see~\cite{Marsh-Rietsch:B-model-Grassmannians}*{\S5}.
For some other works on the spectrum of classical equivariant cohomology rings, which correspond to the specialization $q=0$, see~\cite{Gorbounov-Petrov,Goresky-MacPherson:spectrum}.

Theorem~\ref{intro:hm} is proven by passing to the semiclassical limit 
($\hbar \to 0$) in the equivariant $\hbar$-mirror isomorphism of Theorem~\ref{intro:hbar}. 

\subsection{Mirror pairs of Fano type and towards mirror symmetry for Richardson varieties}
In our mirror theorem, the A-model $G^\vee/P^\vee$ and the B-model $(X_t,f_t)$ plays distinctly different roles.  On the other hand, the geometry of $G/P$ features prominently in the construction of $X_t$.  This suggests a more symmetric mirror conjecture should exist.

One such setting could be the mirror pairs of compactified Landau--Ginzburg models studied in \cite{Katzarkov-Kontsevich-Pantev:LG-models}, and one might speculate on the mirror symmetry of the pair of compactified Landau--Ginzburg models
 \[(G/P,g,\omega_{G/P},f_{G/P}) \text{ and }
 (G^\vee/P^\vee,g^\vee ,\omega_{G^\vee/P^\vee},f_{G^\vee/P^\vee}). 
\]
where $g$ is a K\"ahler form, $\omega_{G/P}$ denotes the volume form of 
\cite{KLS:projections-Richardson}, and $f_{G/P}$ denotes the potential function on $\oGP$ 
discussed above.  If such a mirror theorem holds, we would expect a matching of the 
cohomologies of the log Calabi--Yau manifolds $\oGP$ and $\mathring{G^\vee/P^\vee}$.  Indeed the 
equality $H^*(\oGP) \cong H^*(\mathring{G^\vee/P^\vee})$ holds more generally for open Richardson 
varieties.  

Namely, we identify the Weyl group of $G$ and of $G^\vee$, and denote it by $W$. For $v,w\in W$ with $v \le w$, the open Richardson variety $\cR^w_v\subset G/B$ is the intersection of the Schubert cell $B_- \dot{v} B/B$ with the opposite Schubert cell $B \dot{w}B/B$.  We denote by $\check{\cR}^{w}_v \subset G^\vee/B^\vee$ the open Richardson variety attached to $G^\vee$.  Then we have the equality $H^*(\cR^w_v) \cong H^*(\check{\cR}^{w}_v)$ (Proposition \ref{prop:Richardson}).  We are thus led to the question: can our mirror theorems be generalized to Richardson varieties?

Let us also comment that the open Richardson varieties $\cR^w_v$ are expected to be cluster varieties \cite{Leclerc:cluster-Richardson}.
We refer to~\cites{Goncharov-Shen:canonical-bases,GHK:mirror-log-CY-I,GHKK:canonical-bases-cluster} for recent results on
canonical bases on log Calabi--Yau varieties assuming the existence of a cluster structure.  For a discussion of the cluster structure of $\oGP$, see~\cite{BFZ:cluster-III}*{\S2}.

\subsection{Other related works} 
Witten~\cite{Witten:gauge-theory-wild-ramification,Donagi-Pantev:Langlands-duality}, based on previous works by Hausel--Thaddeus~\cite{Hausel-Thaddeus:mirror-Langlands-Hitchin},
Kapustin--Witten and Gukov--Witten, has related Langlands reciprocity for connections with possibly irregular singularities and mirror symmetry of
Hitchin moduli spaces of Higgs bundles. The present work may perhaps be seen as an instance of this relation in the case of rigid connections, although we are
considering rather the moduli spaces of \emph{holomorphic} bundles. Another important difference is
that automorphic $D$-modules appear in~\cite{Witten:gauge-theory-wild-ramification} as the $A$-side, as opposed to the $B$-side in the present work. 

Recently, the rigid connections of~\cite{Frenkel-Gross,HNY} have been generalized by Yun 
and Chen to parahoric
structures and Yun considered rigid automorphic forms ramified at three points. See
also~\cite{Boalch:wild-character-varieties,Bremer-Sage:moduli-irregular,STWZ:Cluster-Legendrian}
 for recent advances on wild character varieties.

Quantum multiplication by divisor classes in the equivariant quantum cohomology ring $QH^*_{T^\vee\times \C^\times}(T^*G^\vee/P^\vee)$ of the cotangent
bundle has been recently computed for any partial flag variety by Su~\cite{Su:equivariant-QH}, extending work of Braverman--Maulik--Okounkov~\cite{BMO}
for the cotangent bundle $T^*G^\vee/B^\vee$ of the full flag variety. 
Specializing the $\C^\times$-equivariant parameter to zero, it recovers the 
$T^\vee$-equivariant quantum Chevalley formula of 
Mihalcea~\cite{Mihalcea:EQ-Chevalley}. It would be interesting to investigate 
generalizations of Rietsch's mirror conjecture to this setting.
See~\cite{Tarasov-Varchenko:hypergeometric} for work in this direction.

A different approach to mirror phenomena for partial flag varieties is the study of period integrals of hypersurfaces in $G/P$ by Lian--Song--Yau \cite{Lian-Song-Yau}.  Their ``tautological system" is further studied in \cite{Huang-Lian-Zhu}, where geometry such as the open projected Richardson $\oGP$ also makes an appearance.

Since $QH^*(G^\vee/P^\vee)$ is known~\cite{Chaput-Manivel-Perrin:QH-minuscule-III} to be semisimple for minuscule $P^\vee$, the Dubrovin conjecture
concerning full
exceptional collections of vector bundles on $G^\vee/P^\vee$ and the Stokes matrix of $\cQ^{G^\vee/P^\vee}$ at $q=\infty$ is expected to hold. It has
been established for projective spaces by Dubrovin and Guzzetti, and more generally for Grassmannians by Ueda~\cite{Ueda:Stokes-QHgr}. Recent works on exceptional
collections for projective homogeneous spaces include~\cite{Kuznetsov-Polishchuk} for $G^\vee$ classical and~\cite{Faenzi-Manivel} for $G^\vee$ of type $E_6$.

Related to the Dubrovin conjecture are the Gamma conjectures~\cite{Galkin-Golyshev-Iritani:gamma-conj}. The relation with mirror symmetry is discussed
in~\cite{Katzarkov-Kontsevich-Pantev:LG-models,Galkin-Iritani:Gamma-conj}.  The conjectures are known for
Grassmannians~\cite{Galkin-Golyshev-Iritani:gamma-conj}, for certain toric varieties~\cite{Galkin-Iritani:Gamma-conj}, and is compatible with
taking hyperplane sections~\cite{Galkin-Iritani:Gamma-conj}. Also, \cite{Golyshev-Zagier} establishes the Gamma conjecture I for Fano 3-folds with Picard rank
one, exploiting notably the modularity of the quantum differential equation which holds for 15 of the 17 families from the Iskovskikh classification.  See \cite{HKLY,Fang-Zhou} for further recent progress.

\section{Preliminaries}\label{s:pre}
Notation in \S\ref{sub:root} to \S\ref{sub:minuscule} will be used frequently in the article. The results of \S\ref{ssec:gamma} will be used in \S\ref{s:FG}-\S\ref{s:qh}.
\subsection{Root systems and Weyl groups}\label{sub:root}
Let $G$ denote a complex almost simple algebraic group, $T \subset G$ a maximal torus, and $B, B_-$ opposed Borel subgroups.  The Lie algebras are denoted $\g, \t, \b, \b_-$ respectively.  Let $R$ denote the root system of $G$ and $I$ denote the vertex set of the Dynkin diagram.  The simple roots are denoted $\alpha_i \in \t^*$, and the simple coroots are denoted $\alpha_i^\vee \in \t$.    The pairing between $\t$ and $\t^*$ is denoted by $\ip{.,.}$.  Thus $a_{ij} = \ip{\alpha_i,\alpha^\vee_j}$ are the entries of the Cartan matrix.  Let $R^+, R^- \subset R$ denote the subsets of positive and negative roots.  Let $\theta \in R^+$ denote the highest root and $\rho := \frac12 \sum_{\alpha \in R^+} \alpha$ be the half sum of positive roots.

We let $W$ denote the Weyl group, and $s_i$, $i \in I$ denote the simple generators.  For a root $\alpha \in R$, we let $s_\alpha \in W$ denote the corresponding reflection.  The length of $w$ is denoted $\ell(w)$.  For $w \in W$, we let $\Inv(w) := \{\alpha\in R^+ \mid w \alpha \in R^-\}$ denote the inversion set of $w$.  Thus $|\Inv(w)| = \ell(w)$.  Let $\leq$ denote the Bruhat order on $W$ and $\lessdot$ denote a cover relation (that is, $w \lessdot v$ if $w < v$ and $\ell(w) = \ell(v) -1$).

Let $P \subset G$ denote the standard parabolic subgroup associated to a subset $I_P \subset I$.  Let $W_P \subset W$ be the subgroup generated by $s_i, i \in I_P$.  Let $W^P$ be the set of minimal length coset representatives for $W/W_P$.  Let $\pi_P: W \to W^P$ denote the composition of the natural map $W \to W/W_P$ with the bijection $W/W_P \cong W^P$.  Let $R_P \subset R$ denote the root system of the Levi subgroup of $P$.  Let $\rho_P:=\frac12 \sum_{\alpha\in R^+_P} \alpha$.

For a weight $\lambda$ of $\g$, we let $V_\lambda$ denote the irreducible highest-weight representation of $\g$ with highest weight $\lambda$.

\subsection{Root vectors and Weyl group representatives}\label{sub:rootvectors}
We pick a generator $x_\alpha$ for the weight space $\g_\alpha$ for each root $\alpha$.  Write $y_j$ for $x_{-\alpha_j}$, and $x_j$ for $x_{\alpha_j}$, and $y_\alpha$ for $x_{-\alpha}$.  Define
$$\dot s_j := \exp(-x_j)\exp(y_j)\exp(-x_j) \in G.$$  Then if $w = s_{i_1} \cdots s_{i_\ell}$ is a reduced expression, the group element $\dot w = \dot s_{i_1} \cdots \dot s_{i_\ell}$ does not depend on the choice of reduced expression.

We assume that the root vectors $x_\alpha$ have been chosen to satisfy:
\begin{enumerate}
\item
$\dot w \cdot x_\alpha = \pm x_{w\alpha}$,
\item
$[x_\alpha,y_\alpha] = \alpha^\vee$,
\end{enumerate}
where $\dot w \cdot x_\alpha$ denotes the adjoint action.  See \cite[Chap.3]{book:Steinberg}.  In 
\eqref{eq:thetasign}, we will make a choice of sign for $x_\theta$.

\subsection{Quantum roots}\label{ssec:tR}
If $G$ is simply-laced, we consider all roots to be long roots.  Otherwise, we have both long and short roots.  Let $\tR^+ \subseteq R^+$ be the subset of positive roots defined by
$$
\tR^+ = \{\beta \in R^+ \mid \ell(s_\beta) = \ip{2\rho,\beta^\vee} -1\}.
$$
A root $\beta \in \tR^+$ is called a {\it quantum root} (terminology to be explained in \S \ref{ssec:Chevalley}).
If $G$ is simply-laced, then $\tR^+ = R^+$.  Otherwise, it is a proper subset.  A root $\alpha \in R^+$ belongs to $\tR^+$ if one of the following is satisfied (see \cite{BMO}):
\begin{enumerate}
\item $\alpha$ is a long root, or
\item no long simple roots $\alpha_i$ appear in the expansion of $\alpha$ in terms of simple roots.
\end{enumerate}
If $G$ is of type $B_n$, then $\tR^+$ is the union of the long positive roots with the short simple root $\alpha_n$.
If $G$ is of type $C_n$, then $\tR^+$ is the union of the long positive roots, with the short roots of the form $\alpha_i + \alpha_{i+1} + \cdots + \alpha_j$ where $1 \leq i \leq j \leq n-1$. 

\subsection{Minuscule weights}\label{sub:minuscule}
Let $\ki \in I$ and $\varpi_\ki$ denote the corresponding fundamental weight.  We call $\ki$, or $\varpi_\ki$, {\it minuscule}, if the weights of $V = V_{\varpi_\ki}$ are exactly the set $W \cdot \lambda$.  Equivalently, $\ki \in I$ is minuscule if the coefficient of $\alpha_\ki^\vee$ in every coroot $\alpha^\vee$ is $\leq 1$. 

Let $P=P_{\ki} \subset G$ be the parabolic subgroup associated to $I_P = I \setminus \{\ki\}$.    Then $W_P$ is the stabilizer of $\varpi_\ki$.  We have natural bijections between $W^P$, $W/W_P$, and $W \cdot \varpi_\ki$.  We have that $\alpha \in R_P$ if the simple root $\alpha_\ki$ does not occur in $\alpha$. We say that  \(P=P_{\ki}\) is a {\it minuscule parabolic} if \(\ki\) is a minuscule weight.

The minuscule nodes for each irreducible root system are listed in Figure 
\ref{fig:minuscule}.  Our conventions follow the Bourbaki numbering, 
see~\cite{bourbaki:lie7-9}*{Chap.VIII, \S7.4, Prop.8}.

If $G$ is simply-laced then a minuscule node is also cominuscule.  Thus the coefficient of $\alpha_\i$ in every root $\alpha \in R^+$ is $\leq 1$.
This means that the nilradical of $P$ is abelian, hence by Borel--de Siebenthal theory, 
$G/P$ is a compact Hermitian symmetric space; see e.g.,~\cite{Gross:minuscule-principalsl2}.

\begin{figure}
\begin{tabular}{||c|c|c|c|c|c||}\hline
\text{$R$} & \text{Dynkin Diagram} &$V_{\varpi_\ki}$ & $\dim V$ &\text{Flag variety} & $\dim G/P$ \\\hline \hline $A_n$ &
\setlength{\unitlength}{3mm}
\begin{picture}(11,3)
\multiput(0,1.5)(2,0){6}{$\circ$}
\multiput(0.55,1.85)(2,0){5}{\line(1,0){1.55}}
\put(6,1.5){$\bullet$} \put(0,0){$1$} \put(2,0){$2$}
\put(3.5,0){$\cdots$} \put(6,0){$k$} \put(7.5,0){$\cdots$}
\put(10,0){$n$}
\end{picture}
& $\Lambda^k \C^{n+1}$
& $\binom{n+1}{k}$
&
Grassmannian $\Gr_{k,n+1}$
&
$k(n+1-k)$
\\
\hline $B_n \; (n\geq 2)$ & \setlength{\unitlength}{3mm}
\begin{picture}(11,3) \multiput(2,1.5)(2,0){5}{$\circ$}
\multiput(0.55,1.85)(2,0){4}{\line(1,0){1.55}}
\multiput(8.55,1.75)(0,.2){2}{\line(1,0){1.55}} \put(8.85,1.53){$>$}
\put(0,1.5){$\circ$}
\put(10,1.5){$\bullet$} \put(0,0){$1$} \put(2,0){$2$}
\put(4,0){$\cdots$} \put(7,0){$\cdots$} \put(10,0){$n$}
\end{picture} &
spinor
&
$2^n$
&
odd orthogonal Grassmannian &$n(n+1)/2$
\\
\hline $C_n \; (n\geq 2)$ & \setlength{\unitlength}{3mm}
\begin{picture}(11,3) \multiput(2,1.5)(2,0){5}{$\circ$}
\multiput(0.55,1.85)(2,0){4}{\line(1,0){1.55}}
\multiput(8.55,1.75)(0,.2){2}{\line(1,0){1.55}} \put(8.85,1.53){$<$}
\put(0,1.5){$\bullet$}
\put(10,1.5){$\circ$} \put(0,0){$1$} \put(2,0){$2$}
\put(4,0){$\cdots$} \put(7,0){$\cdots$} \put(10,0){$n$}
\end{picture} &
$\C^{2n}$
&$2n$
&
projective space &$2n-1$
\\
 \hline $D_n \; (n\geq 4)$ & $\begin{array}{c}
\setlength{\unitlength}{2.9mm}
\setlength{\unitlength}{2.9mm} \begin{picture}(11,3.5)
\multiput(0,1.6)(2,0){5}{$\circ$}
\multiput(0.55,2)(2,0){4}{\line(1,0){1.55}}
\put(8.5,1.95){\line(2,-1){1.55}} \put(8.5,1.95){\line(2,1){1.55}}
\put(10,2.5){$\circ$} \put(10,0.7){$\circ$} \put(0,0){$1$}
\put(2,0){$2$} \put(4,0){$\cdots$} \put(7,0){$\cdots$}
\put(9.1,0){$n\!-\!1$} \put(11, 2.3){$n$} \put(10,2.45){$\circ$}
\put(0,1.6){$\bullet$}
\put(10,0.75){$\circ$}
\end{picture}
\end{array}$ &
$\C^{2n}$
&
$2n$
&
even dimensional quadric &
$2n-2$
\\
\hline $D_n \; (n \geq 4)$ & $\begin{array}{c}
\setlength{\unitlength}{2.9mm}
\setlength{\unitlength}{2.9mm} \begin{picture}(11,3.5)
\multiput(0,1.6)(2,0){5}{$\circ$}
\multiput(0.55,2)(2,0){4}{\line(1,0){1.55}}
\put(8.5,1.95){\line(2,-1){1.55}} \put(8.5,1.95){\line(2,1){1.55}}
\put(10,2.5){$\circ$} \put(10,0.7){$\circ$} \put(0,0){$1$}
\put(2,0){$2$} \put(4,0){$\cdots$} \put(7,0){$\cdots$}
\put(9.1,0){$n\!-\!1$} \put(11, 2.3){$n$} \put(10,2.45){$\circ$}
\put(10,0.75){$\bullet$}
\end{picture}
\\
\setlength{\unitlength}{2.9mm}
\setlength{\unitlength}{2.9mm} 
\begin{picture}(11,3.5)
\multiput(0,1.6)(2,0){5}{$\circ$}
\multiput(0.55,2)(2,0){4}{\line(1,0){1.55}}
\put(8.5,1.95){\line(2,-1){1.55}} \put(8.5,1.95){\line(2,1){1.55}}
\put(10,2.5){$\circ$} \put(10,0.7){$\circ$} \put(0,0){$1$}
\put(2,0){$2$} \put(4,0){$\cdots$} \put(7,0){$\cdots$}
\put(9.1,0){$n\!-\!1$} \put(11, 2.3){$n$} \put(10,2.45){$\bullet$}
\put(10,0.75){$\circ$}
\end{picture}
\end{array}$  &
spinor
&
$2^{n-1}$
&
even orthogonal Grassmannian &
$n(n-1)/2$
\\ 
\hline
$E_6$ & $\begin{array}{c}
\setlength{\unitlength}{3mm}
\begin{picture}(9,3.6)
\multiput(0,0.5)(2,0){5}{$\circ$}
\multiput(0.55,0.95)(2,0){4}{\line(1,0){1.6}} \put(0,0.5){$\bullet$} \put(8,0.5){$\circ$}
\put(8,0.5){$\circ$} \put(4,2.6){$\circ$}
\put(4.35,1.2){\line(0,1){1.5}} \put(0,-.6){$1$} \put(2,-0.6){$3$}
\put(4,-.6){$4$} \put(6,-.6){$5$} \put(5,2.5){$2$} \put(8,-.6){$6$}
\end{picture}
\\
\setlength{\unitlength}{3mm}
\begin{picture}(9,3.6)
\multiput(0,0.5)(2,0){5}{$\circ$}
\multiput(0.55,0.95)(2,0){4}{\line(1,0){1.6}} \put(0,0.5){$\circ$} \put(8,0.5){$\bullet$}
\put(8,0.5){$\circ$} \put(4,2.6){$\circ$}
\put(4.35,1.2){\line(0,1){1.5}} \put(0,-.6){$1$} \put(2,-0.6){$3$}
\put(4,-.6){$4$} \put(6,-.6){$5$} \put(5,2.5){$2$} \put(8,-.6){$6$}
\end{picture}
\end{array}$ &
&
27
&
Cayley plane &
16
\\[12mm] \hline $E_7$ & \setlength{\unitlength}{3mm}
\begin{picture}(11,4)
\put(0,0.9){$\circ$} \multiput(2,0.9)(2,0){4}{$\circ$}
\put(10,0.9){$\bullet$}
\multiput(0.55,1.35)(2,0){5}{\line(1,0){1.6}} \put(10,0.9){$\circ$}
\put(4,3){$\circ$} \put(4.35,1.6){\line(0,1){1.5}} \put(0,-.2){$1$}
\put(2,-0.2){$3$} \put(4,-.2){$4$} \put(6,-.2){$5$} \put(5,2.9){$2$}
\put(8,-.2){$6$} \put(10,-.2){$7$}
\end{picture} 
&
& 56
& Freudenthal variety & 27
\\[1mm] \hline
\end{tabular}
\caption{The minuscule parabolic quotients. In the second column, the possible minuscule nodes are indicated with a black vertex.}
\label{fig:minuscule}
\end{figure}

\subsection{A remarkable quantum root}\label{ssec:gamma}
Fix a minuscule node $\ki$ and corresponding parabolic $P = P_\ki$.  Define the long root $\qroot = \qroot(\ki) \in R$ by:
$$
\qroot := \begin{cases} \alpha_\i &\mbox{if $G$ is simply-laced}, \\
\alpha_{n-1}+ 2\alpha_n & \mbox{if $G$ is of type $B_n$ (and thus $\ki = n$)}, \\
2\alpha_1 + 2 \alpha_2 + \cdots + 2\alpha_{n-1}+ \alpha_n=\theta & \mbox{if $G$ is of type $C_n$ (and thus $\ki = 1$).}
\end{cases}
$$
Since $\qroot$ is a long root, it is also a quantum root.  The coroot $\qroot^\vee$ is the (unique) 
``Peterson--Woodward lift" of $\alpha_\i^\vee + Q^\vee_P \in Q^\vee/Q^\vee_P$ where $Q^\vee, 
Q^\vee_P$ denote coroot lattices; see \S\ref{sec:proofs}.

Let $I_Q = \{j\in I_P \mid \ip{\alpha_j,\qroot^\vee} = 0\} = \{j \in I_P \mid \ip{\qroot,\alpha_j^\vee} = 0\}$.  Then $\alpha\in R_Q$ if no simple root $\alpha_j$ with $j\notin I_Q$ occurs in $\alpha$.  If $G$ is simply-laced, then $I_Q$ is the set of nodes in $I$ not adjacent to $\i$.  If $G$ is of type $B_n$, then $I_Q = \{1,2,\ldots,n-3,n-1\}$.  If $G$ is of type $C_n$, then $I_Q = \{2,3,\ldots,n\} = I_P$.

\begin{lemma}\label{lem:gamma}
The root $\qroot$ has the following properties:
\begin{enumerate}
\item
$\ip{\varpi_\i,\qroot^\vee} = 1$, and
\item
 $\ip{\alpha,\qroot^\vee} = -1$ for $\alpha \in R^+_P \setminus R^+_Q$.
\end{enumerate}
\end{lemma}
\begin{proof}
Direct check.
\end{proof}

It turns out that the root $\qroot$ can be characterized in a number of ways. 

\begin{proposition}\label{prop:gamma}
Let $\beta \in R^+\setminus R^+_P$.  Then the following are equivalent: 
\begin{enumerate}
\item
$\beta = \qroot$;
\item
we have $\ip{\alpha,\beta^\vee} \in \{-1,0\}$ for all $\alpha \in R^+_P$;
\item
there exists $w \in W^P$ such that $\beta = -w^{-1}(\theta)$. 
\end{enumerate}
\end{proposition}

Define
$
W(\qroot) := \{w \in W^P \mid w\qroot = -\theta\}.
$
Let $w_{P/Q} \in W_P$ be the longest element that is a minimal length coset representative in $w_{P/Q} W_Q $.  Note that $\Inv(w_{P/Q}) = R^+_P\setminus R^+_Q$ (see Lemma \ref{lem:z}).  Denote $s'_\qroot:= s_\qroot w_{P/Q}^{-1}$.  

\begin{proposition}\label{prop:Wgamma}  Suppose $w \in W(\qroot)$.  Then:
\begin{enumerate}
\item
$\ell(ws_\qroot) = \ell(w) - \ell(s_\qroot)$,
\item
$\ell(ws'_\qroot) = \ell(w) - \ell(s_\qroot) - \ell(w_{P/Q}^{-1}) = \ell(w) - \ell(s'_\qroot)$,
\item
$ws'_\qroot = \pi_P(ws_\qroot)$,
\item
there is a unique length-additive factorization $w =u w'$, where $u \in W_J$ and $w' \in W(\qroot)$ is the minimal length element in the double coset $W_J w W_P$.  Here, $W_J$ is a standard parabolic subgroup all of whose generators stabilize $\theta$.
\end{enumerate}
Conversely, suppose $w \in W^P$ satisfies (1) and (2).  Then $w \in W(\qroot)$.
\end{proposition}

Proofs of Propositions \ref{prop:gamma} and \ref{prop:Wgamma} are given in \S \ref{sec:proofs}.

\section{Frenkel--Gross connection}\label{s:FG}
We caution the reader that the roles of $G$ and $G^\vee$ are reversed in \S \ref{s:FG} -- \S \ref{s:examples} compared to the rest of the paper.

\subsection{Principal $\mathfrak{sl}_2$}\label{s:principal-sl2} 
Let $y_p:=\sum_{i \in I}y_i$ which is a principal nilpotent in $\mathfrak{b}_{-}$. Let 
$2\rho^\vee=\sum_{\alpha\in R^+} \alpha^\vee$ viewed as an element of $\mathfrak{t}$.
We have 
\[
2\rho^\vee = 2 \sum_{i\in I} \varpi^\vee_i =  \sum_{i\in I} c_i \alpha_i^\vee, 
\]
where the $c_i$ are positive integers. Let $x_p:=\sum_{i\in I} c_i x_i\in \mathfrak{b}$. Then 
$(x_p,2\rho^\vee,y_p)$ is a principal $\mathfrak{sl}_2$-triple; see~\cite{Gross:motive-principalsl2} 
and~\cite[Chap.VIII, \S11, n$^\circ$4]{bourbaki:lie7-9}.

Let $\mathfrak{z}(y_p)$ be the centralizer of $y_p$ which is an abelian subalgebra of 
dimension equal to the rank of $\g$. The adjoint action of $2\rho^\vee$ preserves 
$\mathfrak{z}(y_p)$ and the eigenvalues are nonnegative even integers. We denote the 
eigenspaces by $\mathfrak{z}(y_p)_{2m}$ with $m\ge 0$. Thus $\mathfrak{z}(y_p)_0=\mathfrak{z}(\g)$.
The integers $m\ge 1$ counted with multiplicity $\dim \mathfrak{z}(y_p)_{2m}$ coincide with the 
exponents $m_1\le \cdots \le m_r$ of the root system $R$. 
Kostant has shown that $\g= \oplus_{m\ge 0} \mathrm{Sym}^{2m}(\C^2) \otimes \mathfrak{z}(y_p)_{2m}$ 
as a representation of the principal $\mathfrak{sl}_2$.
It implies that twice the sum of exponents is equal to the number of roots $|R|$.

The first exponent is $m_1=1$ since $\mathfrak{z}(y_p)_2$ contains $y_p$. The last exponent is 
$m_r=c-1$ which is the height of the highest root $\theta$ because $x_{-\theta}\in 
\mathfrak{z}(y_p)$. In fact, $m_i+m_{r+1-i}=c$ for any $i$.

\subsection{Rigid irregular connection} Frenkel and Gross~\cite{Frenkel-Gross} 
construct a meromorphic connection $\nabla^G$ on the trivial $G$-bundle on $\Gm$ by the 
formula: 
\begin{equation}\label{FG-def} 
\nabla^{G} :=  d + y_p \frac{dq}{q} + x_\theta dq.
\end{equation}
Here $d$ is the trivial connection and $ y_p \frac{dq}{q} + x_\theta dq $ is the $\kg$-valued 
connection $1$-form attached to the
trivialization $G\times \Gm$. For any finite-dimensional $G$-module $V$, it induces a 
meromorphic flat connection
$\nabla^{(G,V)}$ on the trivial vector bundle $V\times \Gm$.  If $V_\lambda$ is an irreducible 
highest module, we also write $\nabla^{(G,\lambda)}$ for $\nabla^{(G,V_\lambda)}$.

The formula~\eqref{FG-def}  is in oper form~\cite{Frenkel-Gross} because $\nabla^G$ is 
everywhere transversal to the trivial $B$-bundle $B\times \Gm$ inside $G\times \Gm$.
The connection $\nabla^G$ has a regular singularity at the point $0$ with monodromy generated by the principal unipotent
$\exp(2i\pi y_p)$. It has an irregular singularity at the point $\infty$, and it is shown 
in~\cite{Frenkel-Gross} that  the slope is
$1/c$ where $c$ is the Coxeter number of $G$. One of the main results of~\cite{Frenkel-Gross} is that the connection is
rigid in the sense of the vanishing of the cohomology of the intermediate extension to 
$\P^1$ of $\nabla^{(G,\mathrm{Ad})}$, viewed as a holonomic $\aD$-module on $\Spec \C[q,q^{-1}] = 
\Gm$. 
Here, $\mathrm{Ad}$ is the adjoint representation of $G$ on $\g$. 

\subsection{Outer automorphisms}\label{s:reduction}
In certain cases, the connection $\nabla^G$ admits a reduction of the structure group. This is related to outer automorphisms of $G$, and thus to automorphisms of the Dynkin diagram. If $G$ is of type $A_{2n-1}$ then $\nabla^G$ can be reduced to type $C_n$. 
If $G$ is of type $E_6$ then $\nabla^G$ can be reduced to type $F_4$.
If $G$ is of type $D_{n+1}$ with $n\ge 4$ then $\nabla^G$ can be reduced to type $B_n$.
In particular, there is a reduction from type $D_4$ to type $B_3$.
In fact, by using the full group $S_3$ of automorphisms of the Dynkin diagram, if $G$ is of type $D_4$, then $\nabla^G$ can be
reduced to type $G_2$.  As a consequence, there is also a reduction of $\nabla^G$ from type $B_3$ to type $G_2$ even though $B_3$
has no outer automorphism.  It follows from~\cite[\S6 and \S13]{Frenkel-Gross} who determine the differential Galois group of $\nabla^G$ for
every $G$, that the above is a complete list of possible reductions.

\subsection{Homogeneity}\label{s:FGgrade} We make the observation that the connection $\nabla^G$ is compatible with the natural grading on $\g$
induced by the adjoint action of the cocharacter subgroup $\rho^\vee: \G_m \to G$.  Precisely we have a $\G_m$-action on $\kg$ induced by
$\zeta \mapsto \mathrm{Ad}(\rho^\vee(\zeta))$ where $\zeta \in \G_m$. Consider also the $\G_m$-action on 
$\Gm$ given by $\zeta\cdot q=\zeta^{c}q$, where we recall that $c$ is the Coxeter number of $\g$. 
It induces a natural $\G_m$-equivariant linear action on $T^*\Gm$ and also on the bundle 
$\kg \otimes T^* \Gm$. 

\begin{lemma}\label{l:connection-homogeneous}
The connection $1$-form $y_p \frac{dq}{q} + x_\theta dq$ in $\Omega^1(\Gm,\kg)$ is 
homogeneous of degree one under the above $\G_m$-equivariant action.
\end{lemma}
\begin{proof}
We have seen in \S\ref{s:principal-sl2} that $y_p$ has degree $m_1=1$ and $x_\theta$ has degree 
$-m_r=1-c$, which implies the assertion.
\end{proof}

\subsection{Frenkel--Gross operator acting on the minuscule representation}\label{s:FG-minuscule}
Let $\ki$ be a minuscule node and $V_{\varpi_\i}$ denote the minuscule representation.  In this section, we explicitly compute $\nabla^{(G,\varpi_\ki)}$.  We shall use the
canonical basis of $V_{\varpi_\ki}$, constructed in \cite{Geck:minuscule}.

There is a basis $\{v_w \mid w \in W^P\}$ of $V_{\varpi_{\ki}}$ characterized by the properties:
\begin{align*}
x_j(v_w) &= \begin{cases} v_{s_jw} & \mbox{if $\ip{w \varpi_\i,\alpha_j^\vee} = -1$}, \\
0 & \mbox{otherwise.}
\end{cases}
\qquad
y_j(v_w) = \begin{cases} v_{s_jw} & \mbox{if $\ip{w \varpi_\i,\alpha_j^\vee} = 1$}, \\
0 & \mbox{otherwise.}
\end{cases}
\end{align*}
and the condition that $v_w$ has weight $w \varpi_\i$.  Note that in the formulae above, $s_jw$ always lies in $W^P$.  (For example, $\ip{w \varpi_\i,\alpha_j^\vee} = 1$ implies that $s_j w > w$ and $s_jw W_P \neq w W_P$.   Together with $w \in W^P$ we have that $s_j w \in W^P$.)

The following result follows from \cite[Lem.3.1]{Geck:minuscule} and the discussion after 
\cite[Lem.3.3]{Geck:minuscule}.  We caution that our $\dot s_j$ is equal to Geck's $n_j(-1)$.
\begin{lemma}\label{lem:Geck} \
\begin{enumerate}
\item
For $w \in W^P$, we have $\dot w v_e = v_w$.  For $u \in W$ and $w \in W^P$, we have $\dot u v_w = \pm v_{\pi_P(uw)}$.
\item
For $\alpha \in R^+$ and $w \in W^P$, we have 
$$
x_\alpha(v_w) = \begin{cases} \pm v_{s_\alpha w} & \mbox{if $\ip{w \varpi_\i,\alpha^\vee} = -1$}, \\
0 & \mbox{otherwise.}
\end{cases} 
\qquad x_{-\alpha}(v_w) = \begin{cases} \pm v_{s_\alpha w} & \mbox{if $\ip{w \varpi_\i,\alpha^\vee} = 1$}, \\
0 & \mbox{otherwise.}
\end{cases}
$$
\end{enumerate}
\end{lemma}

\begin{lemma}\label{lem:simple}
Let $j \in I$ and $w \in W^P$.  Then 
$$
y_j v_w = \begin{cases} v_{ws_\beta} &\mbox{if $\beta = w^{-1}(\alpha_j) \in R^+ \setminus R^+_P$}, \\
0 & \mbox{otherwise.}
\end{cases}
$$
In the first case, we automatically have $ws_\beta \gtrdot w$ and $ws_\beta \in W^P$.
\end{lemma}
\begin{proof}
Let $\beta = w^{-1}(\alpha_j)$.  The condition that $\beta > 0$ is equivalent to $s_j w > w$.  In this case, $\Inv(s_jw) = \Inv(w) \cup \{\beta\}$, so the condition that $s_jw \in W^P$ is equivalent to $\beta \notin R_P$.  The condition $\ip{w \varpi_\i,\alpha_j^\vee} = 1$ is thus equivalent to $\beta \in R^+ \setminus R^+_P$.
\end{proof}

Recall that we have defined a distinguished root $\qroot = \qroot(\i) \in R^+$ and a subset $W(\qroot) \subset W$ in \S \ref{ssec:gamma}.

\begin{lemma}\label{lem:theta}
There is a sign $\varepsilon \in \{+1,-1\}$, not depending on $w \in W^P$, such that 
$$
\varepsilon x_\theta v_w = \begin{cases} v_{\pi_P(ws_\qroot)} &\mbox{if $w \in W(\qroot)$}, \\
0 & \mbox{otherwise.}
\end{cases}
$$
\end{lemma}
\begin{proof}
Let $\beta = -w^{-1}(\theta)$.  By Lemma \ref{lem:Geck}(2), $x_\theta v_w \neq 0$ if and only if $\ip{w \varpi_\i,\theta^\vee} = -1$.  By a similar argument to the proof of Lemma \ref{lem:simple}, this holds if and only if $\beta \in R^+ \setminus R^+_P$.  By Proposition \ref{prop:gamma}, we have $x_\theta v_w \neq 0$ if and only if $\beta = \qroot$, that is, $w \in W(\qroot)$.

Suppose $w \in W(\qroot)$.  By Proposition \ref{prop:Wgamma}(4), we have $w = uw'$ where $u \in W_J$ is an element of a standard parabolic subgroup stabilizing $\theta$, and the product $uw'$ is length-additive.  Then we have
$$
 x_\theta v_w =  x_\theta  \dot u v_{w'} = \varepsilon' \, \dot u \dot w' x_{-\qroot} (\dot w')^{-1} v_{w'} = \varepsilon \, \dot u v_{w's'_\qroot} = \varepsilon \, v_{ws'_\qroot},
$$
where $\varepsilon', \varepsilon$ are signs not depending on $w$.  For the first equality we have used Lemma \ref{lem:Geck}(1).  For the second equality, we used that $\dot u$ is a product of elements $\dot s_j$, where $s_j\theta = \theta$ and thus $\dot s_j$ commutes with $x_\theta$.  For the third equality, we used $x_{-\qroot} v_e = v_{s_\qroot}$ which follows from Lemma \ref{lem:Geck}(2) and Lemma \ref{lem:gamma}(1).  In the last two equalities, we used Proposition \ref{prop:Wgamma}(2) applied to $w, w' \in W(\qroot)$.  In the last equality we also used that $\ell(ws'_\qroot) = \ell(u) + \ell(w's'_\qroot)$.
\end{proof}
From now on we make the assumption that
\begin{equation}\label{eq:thetasign}
\mbox{$x_\theta \in \g_\theta$ is chosen so that $\varepsilon =1$ in Lemma \ref{lem:theta}.}
\end{equation}

\section{Quantum cohomology connection}\label{s:qh}
\subsection{Quantum cohomology of partial flag varieties}
Let $P \subset G$ be an arbitrary standard parabolic subgroup.  Let $QH^*(G/P)$ denote the small quantum cohomology ring of $G/P$.  It is an algebra over $\C[q_i \mid i \notin I_P]$, where we write $q_i$ for $q_{\alpha_i^\vee}$.
 For $w \in W^P$, let $\sigma_w \in QH^*(G/P)$ denote the quantum Schubert class. For each $i\in I$, let $\sigma_i := \sigma_{s_{\alpha_i}}$.  Then we have
$$
QH^*(G/P) \cong \bigoplus_{w \in W^P} \C[q_i \mid i \notin I_P] \cdot \sigma_w.
$$

\subsection{Quantum Chevalley formula}\label{ssec:Chevalley}  

The quantum Chevalley formula for a general $G/P$ is due to Fulton--Woodward~\cite{FW} and Peterson \cite{Peterson:lectures}.  Let $\eta_P: Q^\vee \to Q^\vee/Q^\vee_P$ be the quotient map, where $Q^\vee=\oplus_{i\in I} \Z \alpha^\vee_i$ (resp. $Q^\vee_P=\oplus_{i\in I_P} \Z \alpha_i^\vee$) is the coroot lattice.  Recall that $\rho_P=\frac12 \sum_{\alpha\in R^+_P} \alpha$.
The following version of the quantum Chevalley rule is from \cite[Thm.10.14 and 
Lem.10.18]{Lam-Shimozono:GmodP-affine}.
\begin{theorem}\label{thm:FW}
For $w \in W^P$, we have
$$
\sigma_i *_q \sigma_w = \sum_{\beta} \ip{\varpi_i,\beta^\vee} \sigma_{ws_\beta} + \sum_{\nu} \ip{\varpi_i,\nu^\vee} q_{\eta_P(\nu^\vee)} \sigma_{\pi_P(ws_\nu)},
$$
where the first summation is over $\beta \in R^+ \setminus R^+_P$ such that $ws_\beta \gtrdot w$ and $ws_\beta \in W^P$, and the second summation is over $\nu \in R^+ \setminus R^+_P$ such that 
\begin{gather}
\ell(ws_\nu) = \ell(w) - \ell(s_\nu), \quad \text{and}
\label{FW:1}
\\
\ell(\pi_P(ws_\nu)) = \ell(w) + 1 - \ip{ 2(\rho-\rho_P),\nu^\vee}.
\label{FW:2}
\end{gather}
\end{theorem}

\subsection{Degrees}\label{s:QHgrade}
The quantum cohomology ring $QH^*(G/P)$ is a graded ring.  The degree of $\sigma_w$ is equal to $2\ell(w)$.   The degrees of the quantum parameters $q_i=q_{\alpha_i^\vee}$ for $i\in I\backslash I_P$ are given by
\[
\deg(q_i) = 2 \int_{G/P} c_1(T_{G/P}) \cdot \sigma_{\pi_P(w_0s_i)} = \ip{4(\rho-\rho_P),\alpha_i^\vee}.
\]
The second equality is~\cite[Lem.3.5]{FW}.  Indeed, the first Chern class of $G/P$ satisfies
\begin{equation}\label{c1=}
 c_1(T_{G/P}) = \sum_{i\in I\backslash I_P} \ip{2(\rho-\rho_P),\alpha_i^\vee}  \sigma_i.
\end{equation}

We verify that the quantum multiplication $\sigma_i *_q$ is homogeneous of degree $2$ directly from Theorem~\ref{thm:FW}.  Indeed, $\sigma_{ws_{\beta}}$ has degree $2\ell(w)+2$, and
\begin{equation*}
\begin{aligned}
\deg\, q_{\eta_P(\nu^\vee)} + 2 \ell(\pi_P(ws_\nu)) &= \ip{4(\rho - \rho_P), \eta_P(\nu^\vee)} + 2\ell(w)+2  
- 2 \ip{2(\rho - \rho_P),\nu^\vee}
= 2\ell(w) +2,
\end{aligned}
\end{equation*}
where the second equality follows because $\rho-\rho_P$ is orthogonal to $Q_P^\vee$.

\subsection{Quantum connection and quantum $\aD$-module}\label{s:quantum-connection}
We let   $\C_q$ be the complex points of $\Spec \C[q_i \mid i\notin I_P] $  and $\C^\times_q$ be the 
complex points of $\Spec \C[q_i^{\pm 1} \mid i\notin I_P]$.
We can attach a quantum connection $\cQ^{G/P}$ on the trivial bundle $\C^\times_q\times H^*(G/P)$ over $\C_q^\times$ as follows. 
For each $i\in I\backslash I_P$, the connection $\cQ^{G/P}$ in the direction of $q_i$ is given by
\[
  q_i\frac{\partial}{\partial q_i}  + \sigma_i *_q,
\]
where $*_q$ is quantum multiplication with quantum parameter $q$.
The connection is integrable, which is equivalent to the associativity of the quantum product.
The associated  connection $1$-form is 
\begin{equation} \label{eq:q1form}
\sum_{i\in I\backslash I_P}  (\sigma_i *_q) \frac{dq_i}{q_i} \in \Omega^1(\C_q^\times,\mathrm{End}(H^*(G/P))).
\end{equation}

Define a $\C^\times$-action on $H^*(G^\vee/P^\vee)$ by $\zeta \cdot \sigma = \zeta^{i} \sigma$ for $\zeta \in 
\C^\times$ and $\sigma \in H^{2i}(G^\vee/P^\vee)$. 
Also define a $\C^\times$-action on $\C^\times_q$ by $\zeta \cdot q_i = \zeta^{\deg (q_i)/2} q_i$ for 
$i\notin I_P$.  Then it is clear from the previous \S\ref{s:QHgrade} that the connection 1-form 
\eqref{eq:q1form} is homogeneous of 
degree one for the action of $\C^\times$.

It follows from Theorem~\ref{thm:FW} that quantum multiplication is Laurent polynomial, 
hence $\cQ^{G/P}$ is an \emph{algebraic} connection.

\begin{remark}\label{rem:Cq}
We may identify the universal cover of $\C^\times_q$ with $H^2(G/P)$ and define a flat connection on $H^2(G/P)$ instead, which would correspond to the general framework of Frobenius manifolds (see~\cite{Dubrovin:geometry-TFT,Manin:book-Frobenius,Hertling:book-Frobenius}).
Viewing $\{\sigma_i \mid i\notin I_P\}$ as a basis of $H^2(G/P)$,
the link is the change of parameters given by
$
(q_i \mid i \notin I_P) \mapsto \sum\limits_{i\in I\backslash I_P} \log(q_i)\sigma_i \in H^2(G/P).
$
See e.g.,~\cite[\S2.2]{Iritani:integral-structure}.
Intrinsically, $\C^\times_q$ is identified with the quotient $H^2(G/P)/2i\pi H^2(G/P,\Z)$; see 
also Lemma~\ref{l:2ipiH2} below.
\end{remark}

\subsection{Minuscule case}
For minuscule $G/P$, with $I_P=I \backslash \{\ki\}$, where $\ki$ is a minuscule node, we shall simplify Theorem \ref{thm:FW}. The Schubert divisor class
$\sigma_\ki \in H^2(G/P)$ is a generator of $\mathrm{Pic}(G/P)$. It defines a minimal homogeneous embedding $G/P \subset
\mathbb{P}(V)$, and $G/P$ is realized as the closed orbit of the highest weight vector $v_e\in 
V=V_{\varpi_\ki}$. 
The hyperplane class of $\P(V)$ restricts
to $\sigma$. The following is established in~\cite{Chaput-Manivel-Perrin:QH-minuscule,Snow:nef-value}:
\begin{lemma}\label{l:coxeter-chern}
If $P=P_\ki$ is a minuscule parabolic, then $\langle 2(\rho-\rho_P),\alpha^\vee_\ki \rangle=c$, the Coxeter number of $G$.
\end{lemma}
It then follows from~\eqref{c1=} that the first Chern class $c_1(T_{G/P})$ is equal to $c \sigma$. There is only one quantum parameter $q = q_\ki = q_{\alpha_\ki^\vee}$ which has degree $2c$.  
\begin{proposition}\label{prop:chevalley}
Let $\qroot = \qroot(\i)$ be the long root of \S \ref{ssec:gamma}.  Then for $w \in W^P$, we have
$$
\sigma_\ki *_q \sigma_w = \sum_\beta \sigma_{ws_\beta} + \bm{\chi}(w) q \sigma_{\pi_P(ws_\qroot)},
$$
where the first summation is over $\beta \in R^+ \setminus R^+_P$ such that $ws_\beta \gtrdot w$ and $ws_\beta \in W^P$, and $\bm{\chi}(w)$ equals $1$ or $0$ depending on whether $w \in W(\qroot)$ or not.
\end{proposition}
\begin{proof}
For $\beta^\vee \in R^+$, the coefficient $\ip{\varpi_\ki,\beta^\vee}$ is either 0 or 1, and it is equal to 1 if $\beta^\vee \in R^+ \setminus R^+_P$.  This explains the first summation.

Suppose that $\nu \in R^+ \setminus R^+_P$ and we have \eqref{FW:1} and \eqref{FW:2}. 
Define $I_{Q'} := \{j\in I_P \mid \ip{\alpha_j,\nu^\vee} = 0\} = \{j \in I_P \mid \ip{\nu,\alpha_j^\vee} = 0\}$.  
We have $\Inv(ws_\nu) \cap R^+_{Q'} = \emptyset$ and thus $\Inv(ws_\nu) \cap R^+_P \subseteq (R^+_P \setminus R^+_{Q'})$.  Now, \eqref{FW:1} implies that $s_\nu \in W^P$ and thus $\ip{\alpha_i,\nu^\vee} \leq 0$ for $i \in I_P$.  It follows from our definition of $I_{Q'}$ that $\ip{\alpha_i,\nu^\vee} < 0$ for $i \in I_P$ and $\ip{\alpha_i,\nu^\vee} = 0$ for $i \in I_{Q'}$.  Thus $\ip{\alpha,\nu^\vee} < 0$ for $\alpha \in R^+_P \setminus R^+_{Q'}$, so 
$$
|R^+_P \setminus R^+_{Q'}| \leq -\sum_{\alpha \in R^+_P \setminus R^+_{Q'}}\ip{\alpha,\nu^\vee} =-\sum_{\alpha \in R^+_P }\ip{\alpha,\nu^\vee}= - \ip{2\rho_P,\nu^\vee}.
$$
Condition \eqref{FW:2} guarantees that we have equality and hence that $\ip{\alpha,\nu^\vee} = -1$ for $\alpha \in R^+_P \setminus R^+_{Q'}$.  By Proposition \ref{prop:gamma}, we conclude that $\nu = \qroot$.  It follows from the last sentence of Proposition \ref{prop:Wgamma} that $w \in W(\qroot)$.
\end{proof}

\begin{example}
Suppose that $G/P = \Gr(n-1,n)=\mathbb{CP}^{n-1}$. The minimal representative permutations $w\in W^P$ are determined by the value $w(n)\in
[1,n]$, or equivalently by a Young diagram which is a single column of length $w(n)-1$. Denote the Schubert classes by
$\sigma_{\emptyset}=1$, $\sigma_{1}=\sigma_{s_{\alpha_1}}$, $\sigma_2, \ldots, \sigma_{n-1}$. Then $\sigma_1^{*j} = \sigma_j$ for
$1\le j\le n-1$ and $\sigma_1 *_q \sigma_{n-1} = q$. The quantum cohomology ring has presentation $\C[\sigma_1,q]/(\sigma^n_1 - q)$. 
\end{example}

Chaput--Manivel--Perrin~\cite{Chaput-Manivel-Perrin:QH-minuscule,Chaput-Manivel-Perrin:QH-minuscule-II,Chaput-Manivel-Perrin:QH-minuscule-III} study the quantum
cohomology of minuscule and cominuscule flag varieties. In particular they obtain a
combinatorial description in terms of certain 
quivers~\cite[Prop.24]{Chaput-Manivel-Perrin:QH-minuscule}, which may be compared with 
Proposition~\ref{prop:chevalley} above. 

\subsection{Minuscule representation}
Define the linear isomorphism 
\begin{equation}\label{eq:L}
L: H^*(G/P) \to V \qquad \sigma_w \mapsto v_w \text{ for $w\in W^P$.}
\end{equation} 
Recall the principal $\mathfrak{sl}_2$-triple $(x_p,2\rho^\vee,y_p)$. \begin{proposition}[Gross~\cite{Gross:minuscule-principalsl2}] \label{p:Lefschetz-sl2} 
The isomorphism $L$ intertwines the action of the Lefschetz $\mathfrak{sl}_2$ on $H^*(G/P)$ and the action of the principal $\mathfrak{sl}_2$ on $V$.
\end{proposition}
\begin{proof}  
If the term $\sigma_{w
s_\beta}$ occurs in $\sigma_\ki *_0 \sigma_w$ then $w \beta = \alpha_j$ for some $j$ (see \cite{Stembridge:fully-commutative-Coxeter}).
It
then follows from Lemma \ref{lem:simple} that
$L(\sigma_\ki *_0 \sigma_w) = y_pv_w= y_p\circ L(\sigma_w)$.  
On the other hand, we have $\dim(G/P)=\langle \varpi_\i, 2\rho^\vee \rangle$, and 
 $\ell(w)=\ip{\varpi_\ki,\rho^\vee} - \ip{w\varpi_\ki,\rho^\vee}$; see~\cite[\S6]{Gross:minuscule-principalsl2}.
Since $L(\sigma_w)=v_w$ has weight $w \varpi_\i$ (see \S\ref{s:FG-minuscule}),
for every $d \in [0, 2\dim(G/P)]$, the image $L(H^d(G/P))$ is equal to the $2\rho^\vee$-eigenspace of $V$ of eigenvalue
$\dim(G/P)-d$.  
\end{proof}

Consider quantum multiplication $\sigma_\ki *_q$ as an operator on $H^*(G/P)$ with coefficients in $\C[q]$.
\begin{proposition}\label{prop:FG=QH}
We have $L \circ \sigma_\ki *_q= (y_p + q x_\theta ) \circ L$.
\end{proposition}
\begin{proof}
Let $\sigma_\ki *_q = D_1 + D_2$, where $D_1$ and $D_2$ correspond to the two terms of Proposition \ref{prop:chevalley}.
We have seen in the proof of Proposition~\ref{p:Lefschetz-sl2} that $L \circ D_1 = y_p \circ 
L$.  Assumption \eqref{eq:thetasign}
and Proposition \ref{prop:chevalley} show that $L \circ D_2 = q x_\theta \circ L$.
\end{proof}

Golyshev and Manivel \cite{Golyshev-Manivel} study ``quantum corrections" to the geometric Satake correspondence.  Their main result is closely related to our Proposition \ref{prop:FG=QH} for the simply-laced cases. 

Recall from \S\ref{s:quantum-connection} that the quantum connection on $\C^\times_q$ is given by
\begin{equation}\label{eq:nabla}
\cQ^{G/P} = d +  \sigma_\ki *_q\frac{dq}{q}.
\end{equation}
\begin{theorem}\label{t:FGisquantum}
If $P\subset G$ is minuscule with minuscule representation $V$, then under the isomorphism $L:H^*(G/P)\to V$, the quantum
connection $\cQ^{G/P}$ is isomorphic to the rigid connection $\nabla^{(G,V)}$.
Moreover, the isomorphism is graded with respect to the gradings in \S\ref{s:FGgrade} and \S\ref{s:quantum-connection}.
\end{theorem}

\subsection{Automorphism groups}\label{s:automorphism}
The connected automorphism group of a projective homogeneous space $H/P$ is of the same Dynkin type as $H$ except in the following three exceptional cases; see~\cite[\S3.3]{Akhiezer:Lie-gp-actions}:
\begin{itemize}
\item If $H=\Sp(2n)$ is of type $C_n$, $n\ge 2$, and $\ki=1$ is the unique minuscule node, then $H/P_1$ is isomorphic to
projective space $\mathbb{P}^{2n-1}$. Thus it is homogeneous under the bigger automorphism 
group $G=\PGL(2n)$.

\item If $H=\SO(2n+1)$ is of type $B_n$, $n\ge 2$, and $\ki=n$ is the unique minuscule node, then the odd orthogonal
Grassmannian $\SO(2n+1)/P_n$ is isomorphic to the even orthogonal Grassmannian $\SO(2n+2)/P_{n+1}$.

\item If $H$ is of type $G_2$ and $\ki=1$, then $H/P_1$ is a five-dimensional quadric which is also isomorphic to $\SO(7)/P_1$.
In this case $\ki$ corresponds to the unique short root, which is therefore also the shortest highest root, and thus $H/P_1$ is
quasi-minuscule and coadjoint, but it is neither minuscule nor cominuscule. On the other hand, the five-dimensional quadric is
cominuscule as an homogeneous space under $G=\SO(7)$.
\end{itemize}

In each of the above cases the quantum cohomology rings coincide, hence the quantum connections also coincide.
In the first two cases we can apply Theorem~\ref{t:FGisquantum} to deduce that the corresponding rigid connections associated to a
minuscule representation $V$ coincide. 
In view of \S\ref{s:reduction}, we conclude that if there is a minuscule Grassmannian $H/P$ whose connected automorphic group is $G$, then $\nabla^G$ can be reduced to $\nabla^H$.

\subsection{Quantum period solution}\label{sub:quantum-period}
The connection $\cQ^{G/P}$ has regular singularities at $q=0$. Let $S(q)$ be the horizontal section of the dual connection that is asymptotic to
$\sigma_{w_0w_0^P}$ as $q\to 0$. 
Here, $w_0w_0^P$ (resp.
$w_0$, and $w_0^P$) is the longest element of $W^P$ (resp. $W$, and $W_P$).
The quantum period of $G/P$ is $\langle S(q) , 1\rangle$.  Here, $\langle \cdot, \cdot \rangle$ denotes the intersection pairing on $H^*(G/P)$, so $\langle S(q) , 1\rangle$ is equal to the coefficient of $\sigma_{w_0w_0^P}$ in the Schubert expansion of $S(q)$.  The quantum period $\langle S(q) , 1\rangle$ has a power series expansion in $q$ with nonnegative coefficients, which one can determine using the Frobenius method.
We determine the 
first term in the $q$-expansion in the following.
\begin{lemma}\label{l:q-term} As $q\to 0$,
\[
\langle S(q),1 \rangle = 1 + q \int_{G/P} \sigma_{\ki}^{c-1} \sigma_{\pi_P(w_0w_0^P s_\qroot)} + O(q^2).
\]
The integral above is the number of paths in Bruhat order inside $W^P$ from $\pi_P(w_0w_0^P s_\qroot)$ to $w_0w_0^P$. It is a positive integer.
\end{lemma}

\begin{proof}
We write $S(q)= \sigma_{w_0w_0^P} + qv + O(q^2)$, where $v\in H^*(G/P)$.
Since $S$ is a horizontal section of the connection dual to $\cQ^{G/P}$, we have $\frac{dS}{dq}= \sigma_\ki *_q S(q)$. Using the quantum Chevalley formula in
Proposition~\ref{prop:chevalley} and $\frac{dS}{dq} = v + O(q)$, this implies
\[
v = \sigma_{\ki} *_0 v + \sigma_{\pi_P(w_0w_0^Ps_\qroot)}.
\]
Since $\sigma_\ki *_0$ is nilpotent, this equation uniquely determines $v$.

We have $\ell(w_0w_0^P)=\dim(G/P)$, and
\[
\ell(\pi_P(w_0w_0^Ps_\qroot)) = \dim(G/P) + 1 - c,
\]
where $c=\langle 2(\rho - \rho_P), \qroot^\vee \rangle$ is the Coxeter number of $G$ by Lemma~\ref{l:coxeter-chern}.
Hence we find that $\langle v, 1 \rangle$ is as stated in the lemma.

The interpretation as counting paths in Bruhat order follows from the classical Chevalley formula for the cup product with $\sigma_\ki$.
It is a general fact that the Bruhat order of any $W^P$ is a directed poset with maximal element $w_0 w_0^P$.  In particular, there exists
always a path in Bruhat order from any element to the top, and the count is positive.
\end{proof}

\section{Examples}\label{s:examples}
\subsection{Grassmannians}
Let $G = \PGL_n$.
Then $G/P$ is the Grassmannian $\Gr(k,n)$ for $1\le k \le n-1$.
The Weyl group $W=S_n$ and the simple root $\alpha_\ki=\qroot$ corresponds to the transposition of $k$ and $k+1$. We have $\ip{ 2(\rho-\rho_P),\qroot^\vee}=n$.
The maximal parabolic subgroup is $W_P = S_k \times S_{n-k}$.
The minimal representatives $w\in W^P$ are the permutations such that $w(1)< \cdots < w(k)$ and $w(k+1) < \cdots < w(n)$.
Any such permutation can be identified with a Young diagram that fits inside a $k\times (n-k)$ rectangle, and $\ell(w)$ is the number of boxes in the diagram.
The projection $\pi_P:W\to W^P$ consists in reordering the values $w(1),\ldots,w(k)$ in increasing order and similarly for $w(k+1),\ldots,w(n)$.

In the quantum Chevalley formula of Proposition~\ref{prop:chevalley}, the condition that $ws_\beta \gtrdot w$ means that $\beta\in R^+
\setminus R_P^+$ is the transposition of $l\in [1,k]$ and $m\in (n-k,n]$ with $w(m)=w(l)+1$. 
Equivalently, the Young diagram of
$ws_\beta$ has one additional box on the $(k-l+1)$th row. In the second term of the 
quantum Chevalley formula, the condition
$\ell(\pi_P(ws_\qroot))=\ell(w)+1-n$ is equivalent to $w\qroot=-\theta$, which is in turn equivalent to $w(k)=n$ and $w(k+1)=1$.
This can also be seen from the fact that the element $\pi_P(ws_\qroot)$ has Young diagram 
obtained by deleting the rim of the diagram of $w$, 
see~\cite{Bertram-Ciocan-Fulton:quantum-Schur}. A presentation for the quantum 
cohomology ring of Grassmannians is given in 
~\cite{Siebert-Tian:QH,Buch-Kresch-Tamvakis:GW-grassmannians,Buch:QH-Gr}.

The first term in the $q$-expansion in Lemma~\ref{l:q-term} is $\binom{n-2}{k-1}$.
Indeed, $\pi_P(w_0w_0^Ps_\qroot)$ has Young diagram the $(k-1)\times (n-k-1)$ rectangle.
The number of paths in Bruhat order is equal to the number of ways to sequentially add boxes to form the $k\times n$ rectangle which corresponds to the maximal
element $w_0w_0^P$ of $W^P$. This is consistent with the $q$-expansion of the quantum 
period in terms of binomial coefficients in~\cite{BCKS:acta}*{Thm.5.1.6}
and~\cite{Marsh-Rietsch:B-model-Grassmannians}*{Cor.4.7}.

The fundamental representation $V=V_{\varpi_\ki}$ is the exterior product $\Lambda^k \C^n$. The 
highest weight vector is $v_e=e_1\wedge \cdots \wedge e_k$. For every $w\in W^P$, the basis 
vector is $v_w = \dot{w} \cdot v_e =e_{w(1)}\wedge \cdots \wedge e_{w(k)}$. The Schubert class 
$\sigma_w$ is the $B$-orbit closure of $\mathrm{Span}(  e_{w(1)}, \ldots, e_{w(k)} )$ inside $\Gr(k,n)$.

\begin{example}
Assume that $k=2$ and $n=4$. Denote the Schubert classes by $\sigma_{\emptyset}=1$,
$\sigma_1=\sigma_{s_{\alpha_\ki}}$, $\sigma_{11}$, $\sigma_{2}$, $\sigma_{21}$ and $\sigma_{22}$. The quantum Chevalley formula gives the identities $\sigma_1 *_q \sigma_1 = \sigma_{11} + \sigma_2$, $\sigma_1 *_q \sigma_{11} = \sigma_1 *_q \sigma_2 = \sigma_{21} $, $\sigma_1 *_q \sigma_{21} = \sigma_{22} + q $ and $\sigma_1 *_q \sigma_{22} = q \sigma_{21}$.
\end{example}

\subsection{Type $D$}
If $G=\SO(2n)$ is of type $D_n$, $n\ge 4$, and $\ki=1$, then $G/P_1$ is a quadric of dimension $2n-2$ in $\mathbb{P}^{2n}$. The quantum cohomology ring is described in~\cite{Chaput-Manivel-Perrin:QH-minuscule-II,Pech-Rietsch-Williams:LG-quadrics}.

The two minuscule nodes $\ki=n$ and $\ki=n-1$ are equivalent, and then $G/P_n$ is isomorphic to one connected component of the orthogonal Grassmannian of maximal isotropic subspaces in $\C^{2n}$. A presentation for the quantum cohomology ring is given in~\cite{Kresch-Tamvakis:QH-orthogonal}.

\subsection{Exceptional cases}
A presentation of the quantum cohomology ring of the Cayley plane $E_6/P_6$ (resp. the 
Freudenthal variety $E_7/P_7$) is given 
in~\cite[Thm.31]{Chaput-Manivel-Perrin:QH-minuscule} 
(resp.~\cite[Thm.34]{Chaput-Manivel-Perrin:QH-minuscule}).
The quantum corrections in the quantum Chevalley formula are also described in terms of 
the respective Hasse diagram. There are $6$ (resp. $12$) correction terms for $E_6/P_6$ 
(resp.~$E_7/P_7$).

\subsection{Six-dimensional quadric, triality of $D_4$} \label{s:six-dim-quadric}
 A case of special interest is $G=\SO(8)$ of type $D_4$ where all minuscule nodes $1,3,4$ are equivalent. The homogeneous space $G/P_1$ is a six-dimensional
quadric. It also coincides with the Grassmannian $\SO(7)/P_3$ of isotropic spaces of dimension $3$ inside $\C^7$.

The quadric is minuscule both as a $SO(8)$-homogeneous space and as a $SO(7)$-homogeneous space. Theorem~\ref{t:FGisquantum}
applies in both cases so that $\cQ^{SO(8)/P_1}\simeq \cQ^{\SO(7)/P_3}$ is isomorphic to the Frenkel--Gross connection $\nabla^{(G,V)}$ for both $G=\SO(8)$ and $G=\SO(7)$. Here, the
representation $V$ is either the standard representation of $\SO(8)$, or its restriction to 
$\SO(7)$ which is the direct sum of the trivial representation plus the standard 
representation.

\begin{proposition}\label{p:six-quad}
 \

(i) The quantum connection $\cQ^{\SO(8)/P_1}$ of the six-dimensional quadric is the direct sum of two irreducible constituents of dimensions one and seven respectively.

(ii) The differential Galois group is $G_2$.
\end{proposition}
\begin{proof}
We have seen in \S\ref{s:reduction} that $\nabla^G$ for $G$ of type $D_4$ reduces to $\nabla^G$ 
for $G$ of type $G_2$. Thus it suffices to observe that the standard representation $V$ of 
$\SO(8)$ when restricted to $G_2$ decomposes into the trivial representation plus the 
irreducible representation of dimension seven. This holds because the restriction of the 
standard representation of $\SO(7)$ is the seven-dimensional representation of $G_2$.
\end{proof}

A presentation of the quantum cohomology ring of the homogeneous space $G/P_1$ is given in~\cite{Chaput-Manivel-Perrin:QH-minuscule-II}.
It is also given in~\cite{Kresch-Tamvakis:QH-orthogonal} as a particular case of Grassmannian of isotropic spaces and
in~\cite{Pech-Rietsch-Williams:LG-quadrics} as a particular case of even-dimensional quadrics. From either of these presentations or from the quantum Chevalley formula, we find the quantum multiplication by $\sigma = \sigma_\i$ in the Schubert basis, thus
\[
\cQ^{G/P_1} = q\frac{d}{dq} + 
\left(\begin{smallmatrix}
0&0&0&0&0&0&q&0 \\ 
1&0&0&0&0&0&0&q \\ 
0&1&0&0&0&0&0&0 \\ 
0&0&1&0&0&0&0&0 \\ 
0&0&1&0&0&0&0&0 \\ 
0&0&0&1&1&0&0&0 \\ 
0&0&0&0&0&1&0&0 \\ 
0&0&0&0&0&0&1&0 \\ 
\end{smallmatrix}\right).
\]

The middle cohomology $H^6(G/P_1)$ is two-dimensional, spanned by $\{\sigma_3^+,\sigma_3^-\}$. Since $\sigma *_q \sigma_3^+ = \sigma *_q \sigma_3^-$, the
subspace $\C(\sigma_3^+ - \sigma_3^-)$ is in the kernel of $\sigma$ and in particular is a stable one-dimensional subspace of the connection. The
other stable subspace, denoted $H^\#(G/P_1)$ following~\cite{Guest:book:QH}, has dimension seven and is spanned by $\sigma_3^+ + \sigma_3^-$ and all the cohomology in the remaining degrees. This is consistent with Proposition~\ref{p:six-quad}.(i).

The rank seven subspace $H^\#(G/P_1)$ is generated as an algebra by $H^2(G/P_1)$, and moreover the vector $1$ is cyclic for the multiplication by $\sigma$. The quantum $\aD$-module $\cQ^{G/P_1}$ is then given in scalar form as $\aD/\aD L$ where 
\[
L:= \left(q \frac{d}{dq}\right)^7 + 4q^2 \frac{d}{dq} + 2q.
\]
The differential Galois group of $L$ on $\Gm \simeq \C^\times_q$ is equal to $G_2$ according to Proposition~\ref{p:six-quad}.(ii). Recall from~\cite{Frenkel-Gross} that ultimately the reason for the differential Galois group to be $G_2$ is the triality of $D_4$ and the invariance of the Frenkel--Gross connection $\nabla^{SO(8)}$ under outer automorphisms which reduces it to $\nabla^{G_2}$.

After rescaling $L$ by $q\mapsto -q/4$, the $\aD$-module $\aD/\aD L$ becomes isomorphic to the hypergeometric
$\aD$-module ${}_1F_6\left( \begin{smallmatrix}
& & 1/2 & & \\ 1 & 1 & 1 \ 1 & 1 & 1 
\end{smallmatrix} \right)$  studied in~\cite{Katz:exp-sums-diff-eq} with the notation $\mathcal{H}(0,0,0,0,0,0,0; 1/2)$. Katz proved in
\cite[Thm.4.1.5]{Katz:exp-sums-diff-eq} that the differential Galois group is $G_2$ which is 
consistent with Proposition~\ref{p:six-quad}.(ii).

Our work gives a new interpretation of $\aD/\aD L$ studied by Katz and Frenkel--Gross as the quantum connection $\cQ^{G/P_1}$. Hence the following
\begin{question}
Is it possible to see \emph{a priori} that the differential Galois group of the quantum 
connection of the six-dimensional quadric is $G_2$?
\end{question}
The question seems subtle because for example the quantum connection of the $5$-dimensional quadric, which is homogeneous under $G_2$, has rank $6$ (see \S\ref{s:odd-dim} below), and thus its differential Galois group is unrelated to the group $G_2$.

\subsection{Odd-dimensional quadrics}\label{s:odd-dim}
More generally, let $G=\SO(2n+1)$ be of type $B_n$ with $n\ge 3$. Then $G/P_1$ is a 
$(2n-1)$-dimensional quadric and is cominuscule. The cohomology has total dimension $2n$. 
There is one Schubert class $\sigma_k$ in each even degree $2k\le 4m-2$. The quantum 
product is determined in ~\cite[\S4.1.2]{Chaput-Manivel-Perrin:QH-minuscule-II}. In 
particular $\sigma_1 *_q \sigma_{k-1} = \sigma_{k}$ for $1\le k \le n-1$ and $n+1\le k \le 2n-2$, 
$\sigma_1 *_q \sigma_{n-1} = 2 \sigma_n$, $\sigma_1 *_q \sigma_{2n-2} = \sigma_{2n-1} + q$ and $\sigma_1 
*_q \sigma_{2n-1} = q \sigma_1$, which also follows from the quantum Chevalley formula. The 
relation between the quantum connection and hypergeometric $\aD$-modules is studied in 
detail by Pech--Rietsch--Williams~\cite{Pech-Rietsch-Williams:LG-quadrics}. 

The space $G_2/P_1$ is a five-dimensional quadric. Its connected automorphism group is $\SO(7)$ by \S\ref{s:automorphism}.  It is coadjoint as a $G_2$-homogeneous space and cominuscule as a $\SO(7)$-homogeneous space. The cohomology has total dimension $6$.  A presentation of the quantum cohomology ring is $\C[\sigma,q]/(\sigma^6-4\sigma q)$; see~\cite[\S5.1]{Chaput-Perrin:QH-adjoint}.

\section{Character $\aD$-module of a geometric crystal }\label{s:crystal}
In this section, we introduce the character $\aD$-module for the geometric crystal of Berenstein and Kazhdan \cite{BK}.
The roles of $G$ and $G^\vee$ are interchanged relative to \S \ref{s:FG}--\ref{s:examples}.

\subsection{Double Bruhat cells}\label{sec:factorizations}
Let $U\subset B$ and $U_-\subset B_-$ be opposite maximal unipotent subgroups.
For each $w\in W$, define
\begin{align*}
B_-^w &:=  U \dot w U \cap B_-,\\
U^w &:= U \cap B_- \dot w B_-.
\end{align*}

\begin{lemma} Let $ U(w) := U \cap \dot w U_- \dot w^{-1}$.
For $u \in U^{w^{-1}}$, there is a unique $\eta(u) \in B_-^w$ and a unique $\tau(u) \in U(w)$ such that 
\[
\eta(u) = \tau(u) \dot w u.
\]
The twist map $\eta: U^w \to B_-^w$ is a biregular isomorphism and $\tau: U^w \hra U(w)$ is an injection.
\end{lemma}
\begin{proof} This is \cite[Prop.5.1 and 5.2]{Berenstein-Zelevinsky:total-positivity}; see also
\cite[Thm.4.7]{Berenstein-Zelevinsky:tensor-product-multiplicities} and \cite[Claim.3.25]{BK}. 
Since our conventions differ
from those in \cite{Berenstein-Zelevinsky:tensor-product-multiplicities,BK} slightly, we provide a proof. 

If we define the subgroup $ U'(w) := U_- \cap \dot w U_- \dot w^{-1}$, then the multiplication maps $U(w) \times U'(w) \to \dot w
U_- \dot w^{-1}$ and $U'(w) \times
U(w) \to \dot w U_- \dot w^{-1}$ are bijective.  
In particular,
$B_- \dot w B_- \dot w^{-1} = B_- \dot w U_- \dot w^{-1} = B_- U'(w) U(w)=B_-U(w)$.

We have $u^{-1} \in U^w$. Thus $u^{-1} \dot w^{-1} \in B_- U(w) \subset B_-U$. Since $B_- \cap U = 1$, the
factorization $u^{-1} \dot w^{-1} = \eta(u)^{-1} \tau(u)$ with $\eta(u) \in B_-$ and $\tau(u) \in U(w)$ is unique. Moreover $\eta(u)\in B_-^w$. Since $\tau(u)w \in B_- u^{-1}$, it follows similarly that $u\mapsto \tau(u)$ is injective.

Conversely, $\dot w^{-1}U \dot w U = (\dot w^{-1} U \dot w \cap U_- ) U= \dot w^{-1}U(w)\dot w U$. Hence, given $x\in
B_-^w$ we have $\dot w^{-1}x\in \dot w^{-1}U(w)\dot w U$, which provides by factorization an inverse element $\eta^{-1}(x)\in U^w$.    
\end{proof}

\begin{lemma}\label{l:T-action}
For $t\in T$ let $s:=\dot w t \dot w^{-1}$. 
Each of $U^w$ and $U(w)$ is $\mathrm{Ad}(T)$-stable, and for $u\in U^w$, 
\[
\tau(tut^{-1}) = s\tau(u)s^{-1}, \quad  \eta(tut^{-1}) = s \eta(u) t^{-1}.
\]
\end{lemma}
\begin{proof}
Since $s\dot w = \dot w t$, we have
$
s \eta(u) t^{-1}  =  s\tau(u)s^{-1} \dot w  t u t^{-1},
$
hence the assertion follows.
\end{proof}

\subsection{Geometric crystals} \label{s:crystal-potential}
Fix an arbitrary standard parabolic subgroup $P \subset G$.  Let $w_0 \in W$ be the longest element of $W$ and $w_0^P \in W_P$ be the longest element of $W_P$.  Define $w_P := w_0^P w_0$ so that $w_P^{-1}$ is the longest element in $W^P$.  In this case, the subgroup $U_P := U(w_P)$ is the unipotent radical of $P$.  The \defn{parabolic geometric crystal} associated to $(G,P)$ is
$$
X := U Z(L_P) \dot w_P   U \cap B_- = Z(L_P)B_-^{w_P}.
$$ 

We now define three maps $\pi, \gamma, f$ on $X$, called the highest weight map, the weight map, and the decoration or superpotential.

The \defn{highest weight map} is given by
$$
\pi:X \to Z(L_P) \qquad x= u_1 t  \dot w_P  u_2 \mapsto t.
$$  Let $X_t = \pi^{-1}(t) = \{u_1 t  \dot w_P  u_2 \in B_-\}$ be the fiber of $X$ over $t$.   We call $X_t$ the {\it geometric crystal with highest weight $t$}.  Since the product map $Z(L_P) \times B_-^{w_P} \to X$ is an isomorphism, we have a natural isomorphism $X_t \cong B_-^{w_P}$.
Geometrically we think of $X$ as a family of open Calabi--Yau manifolds fibered over $Z(L_P)$.

The \defn{weight map} is given by
\begin{equation}\label{eq:weight}
\gamma: X \to T \qquad x \mapsto x \mod U_- \in B_-/U_- \cong T.
\end{equation}

For $i \in I$, let $\chi_i: U \to \A^1$ be the additive character uniquely determined by
$$
\chi_i(\exp(tx_j)) = \delta_{ij} t,
$$
where the elements \(x_j\) for  \(j\in I\) are given in \S\ref{sub:rootvectors}.
Let the \emph{standard additive character} be $\psi := \sum_{i \in I} \chi_i$. 
The \defn{decoration}, or \defn{superpotential} is given by 
\[
f: X \to \A^1 \qquad x= u_1 t  \dot w_P  u_2  \mapsto \psi(u_1)+\psi(u_2).
\]
It follows from \cite[Lem.5.2]{Rietsch:mirror-construction-QH-GmodP} that $f$ agrees with 
Rietsch's superpotential.  (This reference has conventions Langlands dual to ours.  Our 
$Z(L_P)$ and $ w_0^P$ are denoted $(T^\vee)^{W_P}$ and $w_P$ in  
\cite{Rietsch:mirror-construction-QH-GmodP}, while our $\dot s_i$ is inverse to the 
corresponding notation there.)

Set $\psi_t(u):= \psi(tut^{-1})$ for $t\in T$ and $u\in U$. For $t\in Z(L_P)$, the potential can be expressed as a function of $u\in
U^{w_P^{-1}}$ as follows:
\begin{equation}\label{f_t} 
f_t(u) := f(t\eta(u)) =  \psi_t(\tau(u)) + \psi(u).
\end{equation}
Equivalently, the potential is expressed on $B^{w_P}_{-}= U\dot w_P U \cap B_-$ by
\[
f_t(u_1\dot w_P u_2) = \psi_t(u_1) + \psi(u_2).
\]

\begin{example}\label{ex:p1}
Let $G = \SL(2)$ and $P = B$.  With the parametrizations
$$
u_1 = \left( \begin{array}{cc} 1 & a \\ 0 & 1 \end{array} \right), \qquad t = \left( \begin{array}{cc} t & 0 \\ 0 & 1/t \end{array} \right), \qquad \dot w_P = \dot s_1 = \left(\begin{array}{cc} 0 & -1 \\ 1 & 0 \end{array}\right), \qquad u_2 = \left( \begin{array}{cc} 1 & t^2/a \\ 0 & 1\end{array} \right),
$$
the geometric crystal $X$ is the set of matrices
$$
X = \left\{  \left(
\begin{array}{cc}
 a/t & 0 \\
1/t & t/a \\
\end{array}
\right) \mid a,t \in \C^\times \right\} \subset \SL(2),
$$
equipped with the functions 
$$
f(x) = a + t^2/a, \qquad \pi(x) = \left( \begin{array}{cc} t & 0 \\ 0 & 1/t \end{array} \right), \qquad \text{and } \gamma(x) = \left(
\begin{array}{cc}
 a/t & 0 \\
0 & t/a \\
\end{array}
\right).
$$
\end{example}

\subsection{Open projected Richardson varieties} \label{s:Richardson}

For $v,w\in W$ with $v \le w$, the open Richardson variety $\cR^w_v\subset G/B$ is the intersection of the Schubert cell $B_- \dot{v} B
/B$ with the opposite Schubert cell $B \dot{w}B/B$.  The map $u\mapsto u\dot w_0 \pmod{B}$ induces an isomorphism $U^w \isom \cR^{w_0}_{ww_0}$.  For every $t\in Z(L_P)$, we have a sequence of isomorphisms
\begin{equation}\label{eq:Xisom}
X_t \cong B^{w_P}_- \cong U^{w_P} \cong \cR^{w_0}_{w^P_0},
\end{equation}
given by $x = t u_1 \dot w_P u_2 \mapsto u_1 \dot w_P u_2 \mapsto u_2^{-1} \mapsto u_2^{-1}\dot w_0B$, where in the
factorization we assume $u_1 \in U(w_P)$. We describe directly the composition of these isomorphisms as follows.
\begin{lemma}\label{lem:xu2}
For every $x=u_1 t \dot w_P u_2\in X_t$ with $u_1 \in U_P = U(w_P)$, we have 
$
x^{-1} \dot w_0^P B =  u_2^{-1}\dot w_0B. 
$
\end{lemma}
\begin{proof}
We have $ x^{-1} \dot w_0^P B =  u_2^{-1}\dot w_P^{-1} u_1 \dot w_P \dot w_0 B$. It suffices to observe that $\dot w_0\dot w_P^{-1}
u_1 \dot w_P \dot w_0\in U$ since $u_1\in U(w_P)$.
\end{proof}

The projection $p: G/B \to G/P$ induces an isomorphism of $ \cR^{w_0}_{w^P_0}$ onto its image $\oPi$, the {\it open projected Richardson variety} of $G/P$.  The complement of $ \oPi$ in $G/P$ is the divisor $\partial_{G/P}$ in $G/P$, the multiplicity-free union of the irreducible codimension one subvarieties $D^i$, $i \in I$ and $D_i$, $i \notin I_P$, where 
$$
D^i := \overline{p(\cR^{w_0s_i}_{w_0^P})} \qquad \text{and} \qquad D_i:= \overline{p(\cR^{w_0}_{s_i w_0^P})}.
$$
By \cite[Lem.5.4]{KLS:projections-Richardson}, $\partial_{G/P}$ is anticanonical in $G/P$, or in 
other words, the canonical bundle is given by $\omega_{G/P} = \cO_{G/P}(\partial_{G/P})$.   
There is thus, up to scalar, a unique meromorphic form $\omega$  on $G/P$ that has no 
zeroes, and simple poles along $\partial_{G/P}$.  We let $1/\omega$ denote the section of the 
anticanonical bundle inverse to $\omega$.
\subsection{Explicit formula for superpotential}\label{s:explicit}
We now give an explicit formula for the superpotential $f_t$ as a function on 
$U^{w_P^{-1}}$.  Given $g = b_- v$ where $b_- \in B_-$ and $v \in U$, we set $\pi_+(g) = v$.  
Also let $g \mapsto g^T$ denote the transpose antiautomorphism of $G$ (see for example 
\cite{FZ:double-bruhat}).  Let $g \mapsto g^{-T}$ denote the composition of the inverse and 
transpose antiautomorphisms (which commute).
There is an involution $\star: I \to I$ determined by $w_0 \cdot \alpha_i = -\alpha_{i^\star}$.  We let $P^\star$ be the standard parabolic subgroup determined by $I_{P^\star} = (I_P)^\star$.
\begin{lemma}\label{lem:pi+}
For $u \in U^{w_P^{-1}}$, we have 
$$
 \pi_+((\dot w_0)^{-1} u^T  \dot  w_0^{P^\star}) = (\dot w_0)^{-1}\tau(u)^{-T} \dot w_0. 
$$
\end{lemma}
\begin{proof}
Let $v = \tau(u)$.  Then $x = v \dot w_P u \in B_-$ and $ u = (\dot w_P)^{-1} v^{-1} x$.  Noting that $(\dot w)^T = (\dot w)^{-1}$, we have
\begin{align*}
(\dot w_0)^{-1} u^T =(\dot w_0)^{-1} x^T v^{-T} \dot w_P 
= [(\dot w_0)^{-1} x^T \dot w_0] [(\dot w_0)^{-1} v^{-T}  \dot w_0][(\dot w_0)^{-1} \dot w_P].
\end{align*}
We note that $[(\dot w_0)^{-1} x^T \dot w_0] \in B_-$ and $ [(\dot w_0)^{-1} v^{-T}  \dot w_0] \in U$, so the claim follows from the equality 
$$
(\dot w_0)^{-1} \dot w_P  = (\dot  w_0^{P^\star})^{-1}.
$$
We first argue that $(\dot w_0)^{-1} \dot s_i \dot w_0 = \dot s_{i^\star}$.  Write $\alpha^\vee(t)$ for the cocharacter $\G_m \to T$.  Then $\alpha^\vee_i(-1) = (\dot s_i)^2 \in T$ and $\alpha^\vee_i(-1)^2 = 1$.  Let $w' = s_i w_0 = w_0 s_{i^\star}$ and compute
$$
(\dot w_0)^{-1} \dot s_i \dot w_0 = (\dot w_0)^{-1} \alpha^\vee_i(-1) \dot w' = (\dot w_0)^{-1} \alpha^\vee_i(-1) \dot w_0  \alpha^\vee_{i^\star}(-1) \dot s_{i^\star} = \dot s_{i^\star}
$$
where we have used $(\dot w_0)^{-1} \alpha^\vee_i(t) \dot w_0  = w_0 \cdot \alpha_i(t) = \alpha_{i^\star}(t^{-1})$.  It follows that 
$$
(\dot w_0)^{-1} \dot w_P = (\dot w_0)^{-1} \dot w_P \dot w_0 (\dot w_0)^{-1}= \dot w_{P^\star} (\dot w_0)^{-1}= (\dot  w_0^{P^\star})^{-1},
$$
as required.  
\end{proof}

 In the following, we shall assume that $G$ is simply-connected.  Since the partial flag variety $G/P$ only depends on the type of $G$, we lose no generality.

For a fundamental weight $\varpi_i$ and elements $u,w \in W$, there is a generalized minor 
$\Delta_{u\varpi_i,w\varpi_i}:G \to \A^1$, defined in \cite{FZ:double-bruhat}.
This function is equal to the matrix coefficient $g \mapsto \ip{g \cdot
v_{w\varpi_i},v_{u\varpi_i}}$ of $G$ acting on the irreducible representation $V_{\varpi_i}$, with respect to extremal weight
vectors $v_{w\varpi_i} : = \dot w \cdot v_{\varpi_i}$ and $v_{u\varpi_i} :=  \dot u \cdot v_{\varpi_i}$ with weights $w\varpi_i$ and $u\varpi_i$ respectively.  Here $v_{\varpi_i}$ denotes a fixed highest weight vector with weight $\varpi_i$.
\begin{lemma}\label{lem:superu}
For $u \in U^{w_P^{-1}}$, we have 
$$
\psi(\tau(u)) = \sum_{i \in I \setminus I_P^\star}  \frac{\Delta_{w_0^{P^\star} s_i \varpi_{i}, w_0 \varpi_i}(u)}{\Delta_{\varpi_{i}, w_0 \varpi_i}(u)}.
$$
Thus
\begin{equation}\label{eq:fexplicit}
f_t(u) = \psi(u) + \sum_{i \in I \setminus I_P^\star}  \alpha_{i^\star}(t) \frac{\Delta_{w_0^{P^\star} s_i \varpi_{i}, w_0 \varpi_i}(u)}{\Delta_{\varpi_{i}, w_0 \varpi_i}(u)}.
\end{equation}
\end{lemma}
\begin{proof}
First note that $\chi_i(\pi_+(g)) = \dfrac{\Delta_{\varpi_i,s_i \varpi_i}(g)}{\Delta_{\varpi_i,\varpi_i}(g)}$.
We have $\psi(\tau(u)) = \sum_{i \notin I_P} \chi_i(\tau(u))$.  Since $(\dot w_0)^{-1} \exp(ty_i) \dot w_0 = \exp(-tx_{i^\star})$, we have $\chi_i(\tau(u)) = \chi_{i^\star}((\dot w_0)^{-1}\tau(u)^{-T} \dot w_0)$.  By Lemma \ref{lem:pi+}, we have for $i^\star \notin I_P$, the equalities $\chi_{i^\star}(\tau(u)) = $
$$
\chi_{i}(\pi_+((\dot w_0)^{-1} u^T  \dot  w_0^{P^\star})) 
= \frac{\Delta_{\varpi_i,s_i \varpi_i}((\dot w_0)^{-1} u^T  \dot w_0^{P^\star})}{\Delta_{\varpi_i,\varpi_i}((\dot w_0)^{-1} u^T \dot  w_0^{P^\star})} 
=\frac{\Delta_{w_0 \varpi_i,w_0^{P^\star} s_i \varpi_i}( u^T  )}{\Delta_{ w_0 \varpi_i,w_0^{P^\star} \varpi_i}(u^T)} = \frac{\Delta_{w_0^{P^\star} s_i \varpi_{i}, w_0 \varpi_i}(u)}{\Delta_{w_0^{P^\star} \varpi_{i}, w_0 \varpi_i}(u)}.
$$
Finally, observe that we have $w_0^{P^\star} \varpi_{i} = \varpi_i$ whenever $i \in I \setminus I_P^\star$.
For the last formula, we note that for $t \in Z(L_P)$, we have $\chi_i(t\tau(u)t^{-1}) = \alpha_i(t) \chi_i(\tau(u))$.
\end{proof}

Fix a reduced word $\bi = i_1 i_2 \cdots i_\ell$ of $w_P^{-1}$.  We have the Lusztig rational 
parametrization $\G_m^{\ell} \to U^{w_P^{-1}}$ given by 
$$
\a = (a_1,a_2,\ldots,a_\ell) \longmapsto x_\bi(\a) = x_{i_1}(a_1) x_{i_2}(a_2) \cdots x_{i_\ell}(a_\ell),
$$
where $x_{i}(t) := \exp(tx_i)$ denotes a one-parameter subgroup of $G$.

\begin{corollary}\label{cor:superpositive}
In the Lusztig parametrization, the superpotential  $f_t|_{\G_m^\ell}:\G_m^{\ell}  \to \A^1$  is given by 
the function
$$
f_t(a_1,a_2,\ldots,a_\ell) = a_1 + a_2 + \cdots + a_{\ell} + \sum_{i \in I \setminus I_P} \alpha_i(t),
P_i
$$
where $P_i$ is a Laurent polynomial in $a_1,a_2,\ldots,a_{\ell}$ with positive coefficients.
\end{corollary}
\begin{proof}
We may assume that $G$ is simply-connected and apply Lemma \ref{lem:superu}.  We have 
$\psi(x_\bi(\a)) = a_1 + a_2 + \cdots + a_\ell$.  Now, for any $i\in I\setminus I_P^\star$, the 
generalized minor 
$\Delta_{w_0^{P^\star} s_i \varpi_{i}, w_0 \varpi_i}(x_\i(\a))$ is a polynomial in $a_1,a_2,\ldots,a_\ell$ 
with positive coefficients by 
\cite[Thm.5.8]{Berenstein-Zelevinsky:tensor-product-multiplicities} 
and $\Delta_{\varpi_{i}, w_0 \varpi_i}(x_\bi(\a))$ is a monomial in $a_1,a_2,\ldots,a_\ell$ by 
\cite[Cor.9.5]{Berenstein-Zelevinsky:tensor-product-multiplicities}.
\end{proof}

Corollary \ref{cor:superpositive} generalizes \cite[Thm.5.6]{Chhaibi:Whittaker-processes} to 
the parabolic setting.

\begin{example}\label{ex:squarefree}
If $P$ is minuscule, then $I\setminus I_P = \{\ki\}$ consists of a single minuscule node. It 
follows from the proofs of
\cite[Thm.5.8 and Cor.9.5]{Berenstein-Zelevinsky:tensor-product-multiplicities} 
that $\Delta_{\varpi_{\ki^\star}, w_0 \varpi_{\ki^\star}}(x_\bi(\a)) = a_1a_2\cdots a_\ell$ and that 
$\Delta_{w_0^{P^\star} s_{\ki^\star} \varpi_{\ki^\star}, w_0 \varpi_{\ki^\star}}(x_\i(\a))$ is a square-free 
polynomial in 
$a_1,a_2,\ldots,a_\ell$ with positive integer coefficients.
\end{example}

\begin{example}
Let us pick $G = \SL(4)$ and $\ki = 2$.  A reduced word for $w_P = w_P^{-1}$ is $2312$.  We obtain the parameterization 
$$
(a_1,a_2,a_3,a_4) \longmapsto u = x_\bi(\a) = \left( \begin{array}{cccc}
 1 & a_3 &a_3a_4 &0 \\
 0 & 1 & a_1+a_4 & a_1 a_2 \\
 0 & 0 &1 & a_2 \\
 0 & 0 &0 & 1
\end{array}
\right),
$$
and
$$
\tau(u) = \left( \begin{array}{cccc}
 1 & 0& \frac{1}{a_1a_3} & -\frac{1}{a_1}  \\
 0 & 1 &\frac{a_1+a_4}{a_1a_2a_3a_4} &-\frac{1}{a_1a_2}  \\
 0 & 0 & 1 & 0 \\
 0 & 0 & 0 & 1
\end{array}
\right).
$$
Thus $\psi(\tau(u)) = (a_1+a_4)/a_1a_2a_3a_4$.  This is equal to the ratio $\Delta_{24,34}(u)/\Delta_{12,34}(u)$, agreeing with Lemma \ref{lem:superu}.  Here, $\Delta_{I,J}$ denotes the minor using rows $I$ and columns $J$.  Thus the superpotential is given by
$$
f_t|_{\G_m^4} = a_1 + a_2 + a_3 + a_4 + q \frac{a_1+a_4}{a_1a_2a_3a_4},
$$
where $q = \alpha_\ki(t)$.  For generic $q$, the function $f_t$ has 4 critical points in the chart $\{(a_1,a_2,a_3,a_4) \in \G_m^{4}\}$, which for $q =1$ are given by 
\begin{align*}
&(-1/\sqrt{2},-\sqrt{2},-\sqrt{2},-1/\sqrt{2}), (-i/\sqrt{2},-i\sqrt{2},-i\sqrt{2},-i/\sqrt{2}), \\
&(i/\sqrt{2},i\sqrt{2},i\sqrt{2},i/\sqrt{2}),(1/\sqrt{2},\sqrt{2},\sqrt{2},1/\sqrt{2}).
\end{align*}
We have $4 < 6 = \dim(H^*(\Gr(2,4))$.  So the Laurent polynomial $f_t|_{\G_m^4}$ is a ``weak mirror":  
The missing critical points lie outside of this toric chart inside $X_t$, consistent with the 
discussion in \cite[\S9]{Rietsch:Peterson}.
\end{example}

\begin{example}
Let us pick $G = \SL(5)$ and $\ki = 2$.  A reduced word for $w_P$ is $234123$.  Using the reversed reduced word for $w_P^{-1}$, we obtain the parametrization $$
(a_1,a_2,a_3,a_4,a_5,a_6) \longmapsto u = x_\bi(\a) = \left( \begin{array}{ccccc}
 1 & a_3 & a_3a_6 & 0 & 0 \\
 0 & 1 &{a_2}+{a_6} & {a_2} {a_5} & 0 \\
 0 & 0 & 1 & {a_1}+{a_5} & {a_1} {a_4} \\
 0 & 0 & 0 & 1 & {a_4} \\
 0 & 0 & 0 & 0 & 1 
\end{array}
\right),$$
and 
$$
\tau(u) = \left( \begin{array}{ccccc}
 1 & 0& \frac{1}{a_1a_2a_3} & -\frac{1}{a_1a_2} & \frac{1}{a_1} \\
 0 & 1 &\frac{a_1a_2+a_1a_6+a_5a_6}{a_1a_2a_3a_4a_5a_6} &-\frac{a_1+a_5}{a_1a_2a_4a_5} & \frac{1}{a_1a_4} \\
 0 & 0 & 1 & 0 & 0 \\
 0 & 0 & 0 & 1 & 0 \\
 0 & 0 & 0 & 0 & 1 
\end{array}
\right).
$$
Thus $\psi(\tau(u)) = (a_1a_2+a_1a_6+a_5a_6)/a_1a_2a_3a_4a_5a_6$.  This is equal to the ratio $\Delta_{235,345}(u)/\Delta_{123,345}(u)$, agreeing with Lemma \ref{lem:superu}.   Thus the superpotential is given by
$$
f_t|_{\G_m^6} = a_1 + a_2 + a_3 + a_4 + a_5 + a_6 + q \frac{a_1a_2+a_1a_6+a_5a_6}{a_1a_2a_3a_4a_5a_6},
$$
where $q = \alpha_\ki(t)$.  In this case $f_t|_{\G_m^6}$ has $10=\dim H^*(\Gr(2,5))$ critical points in the toric chart, so there are no ``missing critical points".
\end{example}

\subsection{Polar divisor of the superpotential}

By construction, the potential $f_t$ can be identified with a rational function on $G/P$. We now show that the polar divisor of $f_t$ is equal to the anticanonical divisor $\partial_{G/P} \subset G/P$.

\begin{proposition}\label{p:potential-minors} The potential $f_t$, viewed as a rational function on $G/P$ has polar divisor
$\partial_{G/P}$.  It can thus be written as the ratio $(1/\eta_t)/(1/\omega)$ of two holomorphic anticanonical sections on $G/P$, where $1/\omega$ is the holomorphic anticanonical section of \S\ref{s:Richardson}.

\end{proposition}

\begin{proof}

Let $x= t u_1  \dot
w_P u_2 \in X_t$, where $u_1 \in U(w_P)$.  Under \eqref{eq:Xisom}, we have that $x$ is sent to $x^{-1} \dot w_0^P P = u_2^{-1} \dot w_0 P$.

Given $y \in B_-$, and $i \in I$, we have
\begin{equation}\label{eq:zB}
\Delta_{w_0  \varpi_i,\varpi_i}(y) = 0 \iff y \in \overline{B\dot w_0 \dot s_i B}.
\end{equation}
Indeed, writing $y = b_1 v b_2$, we have
$$
\langle y \cdot v_{\varpi_i}, v_{w_0 \varpi_i} \rangle = 0 \iff \langle b_1 \cdot v_{v \varpi_i}, v_{w_0 \varpi_i} \rangle = 0 \iff v \leq w_0s_i \iff y \in \overline{B\dot w_0 \dot s_i B}.
$$
Working in the open affine chart (the big cell) $(B_- \cap Uw_P^{-1}U)P/P \subset G/P$, the divisor $D^i$ is thus cut out by the single equation $\Delta_{w_0  \varpi_i,\varpi_i}(y) = 0$.  

Now, we take $y = x^{-1}$.  By Lemma~\ref{lem:xu2}, the vanishing of $\Delta_{w_0  
\varpi_i,\varpi_i}(x^{-1})$ is equivalent to the vanishing of $\Delta_{w_0  \varpi_i,\varpi_i}(u_2^{-1} 
\dot w_0)$ (noting that $w_0^P \varpi_i =\varpi_i$).  On the other hand, by 
\cite[Prop.2.6]{FZ:double-bruhat} and \cite[(1.8)]{BK}
\[
\chi_i(u_2) = - \chi_i(u_2^{-1}) = -\langle u_2^{-1} v_{w_0\varpi_i}, v_{w_0s_i\varpi_i} \rangle 
= - \frac{\Delta_{w_0s_i\varpi_i, w_0\varpi_i}(u_2^{-1})}{\Delta_{w_0\varpi_i, w_0\varpi_i}(u_2^{-1})},
\]
which has a pole along the zero locus of $\Delta_{w_0\varpi_i, w_0\varpi_i}(u_2^{-1}) = \Delta_{w_0\varpi_i, \varpi_i}(u_2^{-1} w_0)$.  Thus the function $\chi_i(u_2)$ has a simple pole along $D^i$ and the function $\psi(u_2)$ has simple poles along all the $D^i, i \in I$.  In a similar manner using Lemma~\ref{lem:superu}, we find that $\psi(u_1)$ has simple poles along the the divisors $D_i, i \notin I_P$.  Since all the divisors $D^i$ and $D_i$ are distinct, the rational function $f_t$  has polar divisor exactly $\partial_{G/P}$.
\end{proof}

\begin{remark}
  Proposition~\ref{p:potential-minors} is one manifestation of mirror symmetry of Fano manifolds.
For example the potentials of mirrors of toric Fano varieties are constructed in~\cite{Givental:mirror-toric} and the same property can be seen to hold.
In general, it is explained in Katzarkov--Kontsevich--Pantev~\cite[Rem.2.5 
(ii)]{Katzarkov-Kontsevich-Pantev:LG-models} by the fact that the cup product by 
$c_1(K_{G^\vee/P^\vee})$ on the cohomology of the mirror manifold $G^\vee/P^\vee$ is a 
nilpotent endomorphism.
\end{remark}

\begin{remark}
The zero divisor of $1/\omega$ and the zero divisor of $1/\eta_t$ may intersect, so $f_t$ has points of indeterminacy. Indeed, this happens in the example of
$\P^2$; see also~\cite[Rem.2.5 (i)]{Katzarkov-Kontsevich-Pantev:LG-models}.
\end{remark}

\subsection{The character $\aD$-module of a geometric crystal}\label{sec:charD}

Let $\Exp:= \aD_{\A^1}/\aD_{\A^1}(\partial_x-1)$ be the exponential $\aD$-module on $\A^1$.  Let $\Exp^f = f^* \Exp$ be the pullback $\aD$-module on $X$.  Finally, define the character $\aD$-module of the geometric crystal $X$ by
\begin{equation}\label{def:Cr}
\BK_{(G,P)} := R\pi_*\Exp^f \quad \text{on $Z(L_P)$.}
\end{equation}
A priori $\BK_{(G,P)}$ lies in the derived category of $\aD$-modules on $Z(L_P)$.  But in 
Theorem \ref{t:HNYisBK} we shall see that $R^i\pi_* \Exp^f = 0$ for $i \neq 0$, and thus 
$\BK_{(G,P)}$ is just a $\aD$-module.

\begin{remark}\label{rem:Dmod}
Our conventions for $\aD$-modules follow those in \cite{HTT:D-modules}.  All the  \(D\)-modules we study in this paper are integrable connections, so in
particular they are holonomic.  We will not need the formalism of derived categories of $D$-modules, since we are always just handling \(D\)-modules.  For the six functors for holonomic $\aD$-modules, we refer the reader to \cite[\S3]{HTT:D-modules}.
Our  \(Rf_*\) and \(Rf_!\) are the same as \(\int_{f}\) and \(\int_{f!}\) there. Our
\(f_*\) and \(f_!\) are the degree  \(0\) parts  \(\int_f^0\) and \(\int_{f!}^0\) there.
Additionally, the reader may consult~\cite{FSY} and the references there for more background on exponential  \(D\)-modules of the form  \(\Exp^f\).
\end{remark}

\subsection{Homogeneity}\label{s:homogeneous}

Recall that $\rho_P=\frac12\sum_{\alpha\in R_P^+} \alpha$ and that $w_P=w^P_0w_0$ is the inverse of the longest element of $W^P$.
\begin{lemma}\label{l:w_Prho} We have $w_P( \rho) = -\rho + 2\rho_P$.
\end{lemma}
\begin{proof}
The element $w_0^P$ sends $R^+_P$ to $R^-_P$ and permutes the elements of $R^+ \setminus R^+_P$.  We compute
$$
w_P(\rho) =w_0^P w_0(\rho) = w_0^P(-\rho) = -w_0^P(\rho-\rho_P) -w_0^P(\rho_P) = -(\rho-\rho_P) +\rho_P = -\rho+2\rho_P. \qedhere
$$
\end{proof}

We view $2\rho^\vee$ as a cocharacter $\G_m\to T$. Similarly, we view $2 \rho^\vee - 2\rho^\vee_P$ as a cocharacter $\G_m\to Z(L_P)$.

\begin{lemma}\label{l:homogeneous}
For any $u\in U^{w_P^{-1}}$, $t\in Z(L_P)$ and $\zeta \in \G_m$,
\[
f_{(\rho^\vee -\rho_P^\vee)(\zeta^2) t}(\mathrm{Ad} \rho^\vee(\zeta)(u)) = \zeta f_t(u).
\]
\end{lemma}
\begin{proof}
We have $\psi_{\rho^\vee(\zeta)}(u) = \zeta \psi(u)$ for any $u\in U$ and $\zeta\in \G_m$.
 Thus in view of~\eqref{f_t}, we have
\[
f_{(\rho^\vee- \rho_P^\vee)(\zeta^2) t}(\mathrm{Ad} \rho^\vee(\zeta)(u))=
\psi_{(\rho^\vee -\rho_P^\vee)(\zeta^2) t}(\tau( \mathrm{Ad} \rho^\vee(\zeta) (u))) - \zeta \psi(u),
\]
and we are now reduced to treating the first term.  It follows from Lemma~\ref{l:T-action} that
\[
\tau( \mathrm{Ad} \rho^\vee(\zeta) (u))
= 
\mathrm{Ad} (w_P ( \rho^\vee)) (\zeta)  (\tau(u))
\]
and therefore the first term above is equal to
\begin{equation*}
\begin{aligned}
\psi_{(w_P ( \rho^\vee))(\zeta) (\rho^\vee - \rho^\vee_P)(\zeta^2) t }(\tau(u)) &= 
\psi_{(w_P ( \rho^\vee) + 2\rho^\vee - 2\rho^\vee_P)(\zeta) t }(\tau(u))  \\
&=\psi_{\rho^\vee(\zeta) t }(\tau(u)) = \zeta  \psi_t(\tau(u)).
\end{aligned}
\end{equation*}
In the second line we have used Lemma~\ref{l:w_Prho}, but with $\rho^\vee$ and $\rho_P^\vee$ instead of $\rho$ and $\rho_P$. 
\end{proof}

We define the following $\G_m$-actions on $X$, $Z(L_P)$, and $\A^1$.  For $\zeta \in \G_m$, we have
\begin{align*}
\zeta \cdot x &= \rho^\vee(\zeta)  x\rho^\vee(\zeta)^{-1} & \mbox{for $x \in X$}, \\
\zeta \cdot t & = (2\rho^\vee -2\rho_P^\vee)(\zeta) t & \mbox{for $t \in Z(L_P)$}, \\
\zeta \cdot a &= \zeta a & \mbox{for $a \in \A^1$}.
\end{align*}
Also equip $T$ with the trivial $\G_m$-action.

\begin{proposition}\label{prop:Gmequiv}
The maps $\pi: X \to Z(L_P)$, $f: X \to \A^1$, and $\gamma: X \to T$ are $\G_m$-equivariant.
\end{proposition}

\begin{proof}
We have 
$
\zeta \cdot x = (2\rho^\vee - 2\rho_P^\vee)(\zeta) w_P(\rho^\vee)(\zeta) x \rho^\vee(\zeta)^{-1}.$
We verify using Lemma~\ref{l:T-action} that
$
x\mapsto w_P(\rho^\vee)(\zeta) x \rho^\vee(\zeta)^{-1}
$
 is an automorphism of
$B_-^{w_P}$. This shows that $\pi(\zeta \cdot x) = \zeta \cdot \pi(x)$.  The second claim follows from Lemma~\ref{l:homogeneous}.  The last claim is immediate from the definitions.
\end{proof}

\begin{corollary}\label{c:hbar-Cr}
For any $\lambda \in \C^\times$, we have 
\[
R\pi_* \Exp^{f/\lambda} \cong [t\mapsto (2\rho^\vee - 2\rho^\vee_P)(\lambda) t]_* \BK_{(G,P)}.
\]
\end{corollary}
\begin{proof}
By definition the left-hand side is equal to $R\pi_* f^* [a\mapsto a/\lambda]^* \Exp$.
In view of Proposition \ref{prop:Gmequiv}, it is isomorphic to
\[
R\pi_* (x \mapsto \hbar\cdot x)_* f^* \Exp =  [t\mapsto (2\rho^\vee - 2\rho^\vee_P)(\hbar) t]_* R\pi_* f^* \Exp 
= 
[t\mapsto (2\rho^\vee - 2\rho^\vee_P)(\hbar) t]_* \BK_{(G,P)},
\]
which concludes the proof.  
\end{proof}

We record the following lemma which will be needed in \S\ref{s:parabolic-bessel} in the context of rapid decay cycles.
\begin{lemma}\label{l:Ad-omega}
The meromorphic form $\omega$ on $G/P$ with simple poles along the anticanonical divisor $\partial_{G/P}$ is preserved under the $T$-action.
\end{lemma}

\begin{proof}
First we observe that each irreducible component of the divisor $\partial_{G/P}$ is $T$-invariant and thus the sections cutting them out are $T$-weight vectors.  Now for $P = B$, the sections cutting out the $2|I|$ divisor components $D_i$ and $D^i$ have $T$ weights $\varpi_1,\ldots,\varpi_n$ and $w_0 \varpi_1,\ldots, w_0 \varpi_n$.  The sum of these weights is 0, so the form $\omega_{G/B}$ must be $T$-invariant.  Each open Richardson variety $\cR^w_v$ has its own canonical form $\omega_{\cR^w_v}$ which is obtained from $\omega_{G/B}$ by taking residues, so again these forms are $T$-invariant.  Finally, for each parabolic $P$, the projection map $p:G/B \to G/P$ induces an isomorphism of $\cR^{w_0}_{w^P_0}$ onto its image $\oPi$.  Since $p$ is $T$-equivariant, the result follows.
\end{proof}

\subsection{Convention for affine Weyl groups}
Let $w\tau^\lambda \in W_\af = W \ltimes P^\vee$ denote an element of the affine Weyl group, and let $\delta$ denote the null root of the affine root system.  Then for $\mu \in P$,
\begin{equation}
\label{eq:Waction}
w \tau^\lambda \cdot (\mu + n \delta) = w\mu +  (n - \ip{\mu,\lambda})\delta.
\end{equation}

\subsection{Cominuscule case}\label{s:crystal-affine}
We now assume that $G$ is simple and of adjoint type.  Fix a cominuscule node $\ki$ of $G$, which is also a
minuscule node of $G^\vee$.  Let $P = P_\i$ be the corresponding (maximal) parabolic, and identify $Z(L_P)$ with $\G_m \cong
\Gm$ via the simple root $\alpha_\ki$. 
\begin{lemma}\label{l:coxeter-cominuscule}
If $P$ is a cominuscule parabolic, then the composition of $(2\rho^\vee-2\rho^\vee_P): \mathbb{G}_m\to Z(L_P)$ with $\alpha_\ki:Z(L_P)\cong \mathbb{G}_m$, is the character
$q\mapsto q^c$, where $c$ is the Coxeter number of $G$.
\end{lemma}
\begin{proof}
We have $\rho^\vee_G - \rho^\vee_P = \rho_{G^\vee} - \rho_{P^\vee}$ for the dual minuscule parabolic group $P^\vee$ of $G^\vee$. Since
$\alpha_\ki$ is a simple coroot of $G^\vee$, it follows from Lemma~\ref{l:coxeter-chern} that $\langle
2(\rho^\vee_G - \rho^\vee_P),\alpha_\ki \rangle = c$ where $c$ is the Coxeter number of $G^\vee$ which is also the Coxeter number
of $G$.
\end{proof}

Let $\Omega$ be the quotient of the coweight lattice of $G$ by the coroot lattice.
Thus $\Omega$ is isomorphic to the center of $G^\vee$. Let $\comp \in \Omega$ be the element 
corresponding to the cominuscule node $\ki$.  Namely, $\comp \equiv -\varpi_\ki^\vee$ under this 
identification (see \cite[\S11.2]{Lam-Shimozono:GmodP-affine}); we have $\comp =  
\tau^{-\varpi^\vee_{\ki}}  w_P$.  For a coweight $\lambda$ of $G$, we abuse notation by letting 
$\tau^\lambda \in G((\tau))$ denote the corresponding element, which is a lift of the translation 
element $\tau^\lambda \in W_\af$.  Our choice is uniquely determined by the following 
property: $\tau^\lambda U_\alpha \tau^{-\lambda} = U_{\tau^\lambda \cdot \alpha}$ where $U_\alpha \subset 
G((\tau))$ denotes the one-parameter subgroup indexed by the affine root $\alpha$.

  Let $\dot \comp =     \tau^{-\varpi^\vee_{\ki}}  \dot w_P=\dot w_P \tau^{\varpi^\vee_{\ki^\star}} $.  Then $\dot \comp \in G((\tau))$ is a lift of $\comp$ to the loop group.  Note that $\dot \comp|_{\tau^{-1} = 1} = \dot w_P$.
Let $G[\tau^{-1}]_1 := \ker(G[\tau^{-1}] \xrightarrow{\ev_\infty} G)$, where $\ev_\infty$ is given by $\tau^{-1} = 0$.
\begin{lemma}\label{lem:alphatheta} \
\begin{enumerate}
\item[(a)]
Let $\alpha \in R$.  Then 
$$
\comp^{-1} \cdot \alpha \in \begin{cases} R - \delta & \mbox{for $\alpha \in R^+ \setminus R^+_P$,} \\
R + \Z_{\geq 0} \delta & \mbox{for $\alpha \notin R^+ \setminus R^+_P$.} 
\end{cases}
$$
\item[(b)]
For $u \in U_P$, we have $\dot \comp^{-1} u \dot \comp \in G[\tau^{-1}]_1$.
\item[(c)]
We have $w_P^{-1} \alpha_\ki = -\theta$ and $\comp^{-1} \alpha_\ki = -\theta - \delta$.
\end{enumerate}
\end{lemma}
\begin{proof}
(a) We first note that $\Inv(w_P^{-1}) = R^+ \setminus R^+_P$ and $w_P = w_{P^\star}^{-1}$ (see \S\ref{s:explicit}).  Thus $w_P^{-1}$ acts as a bijection from $R^+ \setminus R^+_P$ to $-(R^+ \setminus R^+_{P^{\star}})$.  In particular, $w_P^{-1}\cdot R^+ \setminus R^+_P =  -(R^+ \setminus R^+_{P^\star})$.  
For $\alpha \in R$, we compute by \eqref{eq:Waction}
\begin{equation}\label{eq:compact}
\comp^{-1} \cdot \alpha =  \tau^{-\varpi^\vee_{\ki^{\star}}}\cdot w_P^{-1} \cdot \alpha = \tau^{-\varpi^\vee_{\ki^{\star}}} \cdot w_P^{-1}(\alpha) = w_P^{-1}\alpha - \ip{w_P^{-1}\alpha,-\varpi_{\ki^{\star}}^\vee} \delta 
\end{equation}
where $\delta$, the null root of the affine root system, is the weight of $\tau$.  
If $ \alpha \in R^+ \setminus R^+_P$, then \eqref{eq:compact} shows that $\comp^{-1} \cdot \alpha \in R_- - \delta$, using that $\ki$ cominuscule implies $\ki^\star$ cominuscule implies $\ip{\beta,\varpi_{\ki^{\star}}^\vee} = 1$ for all $\beta \in R^+ \setminus R^+_{P^\star}$.  If $\alpha \notin R^+ \setminus R^+_P$, then $\ip{w_P^{-1}\alpha,-\varpi_{\ki^{\star}}^\vee} \geq 0$ and $\comp^{-1} \cdot \alpha \in R + \Z_{\geq 0} \delta$.  This proves (a).  Statement (b) follows immediately from (a).

We prove (c).  Since $w_P^{-1}$ sends $R^+ \setminus R^+_P$ to $-(R^+ \setminus R^+_{P^{\star}})$, we see that $-\theta \in w_P^{-1}(R^+ \setminus R^+_P)$.  Since $\ki$ is cominuscule, every root $\beta \in R^+ \setminus R^+_P$ is of the form $\beta = \alpha_\ki \mod \sum_{j \in I_P} \Z_{\geq 0} \alpha_j$.  But $w_P^{-1} \alpha_j >0$ for all $j \in I_P$, so 
$w_P^{-1}(\alpha_\ki + \sum_{j \in I_P} \Z_{\geq 0} \alpha_j) \subset w_P^{-1}(\alpha_\ki) +\sum_{j \in I} \Z_{\geq 0} \alpha_j$.  We deduce that $w_P^{-1} \alpha_\ki = -\theta$.  The second statement follows from \eqref{eq:compact}.
\end{proof}
Thus we obtain an inclusion
\begin{align}\label{eq:inc}
\iota_t: X_t &\longrightarrow G[\tau^{-1}]_1 \\ 
x = u_1 t \dot w_P u_2 & \longmapsto \dot \comp^{-1} t^{-1} u_1 t \dot \comp \in G[\tau^{-1}]_1,
\end{align}
where $u_1 \in U_P$ and $u_2 \in U^{w^{-1}_P}$.

\subsection{Embedding the geometric crystal into the affine Grassmannian}
We interpret the inclusion $\iota_t$ via the affine Grassmannian.

Let $\Gr = G((\tau))/G[[\tau]]$ denote the affine Grassmannian of $G$.  The connected components $\Gr^\comp$ of $\Gr$ are indexed by $\comp \in \Omega$.  For a dominant weight $\lambda$, let $\Gr_\lambda := G[[\tau]] \tau^{-\lambda} \subset \Gr$ denote the $G[[\tau]]$-orbit.  The closure of $\Gr_\lambda$ is a spherical Schubert variety inside $\Gr$.  (The minus sign in $\tau^{-\lambda}$ is chosen to match the convention \eqref{eq:Waction} and our choice of $\tau^{\lambda}$.)

For $\lambda = \varpi_\ki^\vee$, we have that $\Gr_{\varpi_\ki^\vee} \cong G/P$ is already closed in $\Gr$.  Indeed, the map $G \to \Gr$ given by $g \mapsto g \tau^{\varpi_{\ki^{\star}}^\vee} \mod G[[\tau]]$ has stabilizer $P^{\star}$, giving a closed embedding $G/P^{\star} \cong \Gr_{\varpi_\ki^\vee} \hookrightarrow \Gr^\comp$.  Note that $\varpi_{\ki^{\star}}^\vee \in W \cdot (- \varpi_\ki^\vee)$, so $ \tau^{\varpi_{\ki^{\star}}^\vee} \in G[[\tau]] \cdot \tau^{- \varpi_\ki^\vee}$, and $G/P \cong G/P^\star$.

Since $\dot w_P^{-1} U_P \dot w_P \cap P^{\star}= \{e\}$, we have an inclusion $X_t \hookrightarrow G/P^{\star}$ given by $x = u_1t \dot w_P
u_2 \mapsto t^{-1} u_1^{-1} t \dot w_P\mod P^{\star} $, where $u_1 \in U_P$ and $u_2 \in U^{w_P^{-1}}$.  Composed with the isomorphism $G/P^{\star}  \stackrel{\sim}{\rightarrow} \Gr_{\varpi_\ki^\vee}$, we obtain an inclusion
\begin{align}\label{eq:intoGr}
X_t& \longrightarrow \Gr_{\varpi_\ki^\vee} \\
x = u_1 t \dot w_P  u_2 &\longmapsto t^{-1} u_1 t \dot w_P \tau^{\varpi_{\ki^{\star}}^\vee} = t^{-1} u_1 t \cdot \dot \comp = \dot \comp \; \iota_t(x).
\end{align}

\section{Heinloth--Ng\^o--Yun's Kloosterman $\aD$-module}\label{s:HNY}
While the article \cite{HNY} works over a finite field we will work over the complex numbers $\C$ (cf. 
\cite[\S2.6]{HNY}).  In this section, we assume that $G$ is simple and of adjoint type.

\subsection{A group scheme over $\P^1$}\label{s:groupscheme}
Take $t$ to be the coordinate on $\P^1$, and set $s = t^{-1}$.  
Let 
\begin{align*}
I(0) = I_\infty(0) &:= \{g \in G[[s]] \mid g(0) \in B\},\\
I(1) = I_\infty(1) &:= \{g \in G[[s]] \mid g(0) \in U\}.
\end{align*}
Here \(G[[s]]\) is a shorthand for \(G(\C[[s]])\).
Similarly, define 
\begin{align*}
I^\opp_0(0) &:= \{g \in G[[t]] \mid g(0) \in B_-\},\\
I^\opp_0(1) &:= \{g \in G[[t]] \mid g(0) \in U_-\}.
\end{align*}
Also let $I_\infty(2) = I(2) = [I(1),I(1)]$ be the commutator subgroup of $I(1)$.  One verifies that the Lie algebra of $I(1)/I(2)$ has weights given by the set $I_\af$ of simple affine roots, and in particular our definition of $I(2)$ agrees with that in \cite[\S1.2]{HNY}.  Via the exponential map, we obtain an isomorphism \[I(1)/I(2) \cong \bigoplus_{i \in I_\af} \A^1.\]

Let $\phi: I(1)/I(2) \to \A^1$ be the standard affine character.   Precisely, we fix root 
vectors $x_i = x_{\alpha_i} \in \g_{\alpha_i}$ and define $\phi$ by  
$\phi(\exp(tx_{i})) = -t$ for all $i\in I_\af$.  The choice of $x_i$ for $i \in I$ is already fixed 
in \S\ref{sub:rootvectors}.  Since $\g_{\alpha_0}$ can be identified with $s \g_{-\theta}$, the choice of  
$x_0 \in \g_{\alpha_0}$ is equivalent to a choice of $x_{-\theta} \in \g_{-\theta}$.  The choice of 
$x_{-\theta}$ satisfying the compatibilities (1) and (2) of \S\ref{sub:rootvectors} is equivalent to a choice of a sign, 
which will be fixed in \eqref{eq:thetasign2}. We have that  \(\phi\) is a \emph{generic affine character}, meaning that it takes nonzero values on each $\exp(x_i)$, for $i \in I_\af$.

Denote by $\cG(1,2)$ the group scheme over $\P^1$ in \cite[\S1.2]{HNY}, satisfying
\begin{align*}
\cG(1,2)|_{\Gm} &\cong G \times \Gm, \\ 
\cG(1,2)(\cO_0) &= I^{\opp}_0(1), \\
\cG(1,2)(\cO_\infty) &= I(2).
\end{align*}
Here  \(\cO_0 \cong \C[[s]]\) (resp.  \(\cO_\infty \cong \C[[t]]\)) is the completed local ring at  \(0\) (resp.  \(\infty\)).
The group scheme $\cG(1,2)$ is constructed by \emph{dilatation} \cite[\S3.2]{BLR90}.  First, the group scheme $\cG(1,1)$ is the dilatation of the constant group scheme $G \times \P^1$ along $U_- \times 
\{0\} \subset G \times \{0\}$ and along $U \times \{\infty\} \subset G \times \{\infty\}$.  Then $\cG(1,2)$ is 
the dilatation of $\cG(1,1)$ which is an isomorphism away from $\infty$ and at $\infty$
induces $\cG(1,2)(\cO_\infty) = I(2) \subset I(1) = \cG(1,1)(\cO_\infty)$.

\subsection{Hecke modifications}
Recall that $\Omega$ is the quotient of the coweight lattice of $G$ by the coroot lattice.
Let $\Bun_{\cG}=\Bun_{\cG(1,2)}$ denote the moduli stack of $\cG(1,2)$-bundles on $\P^1$ defined in \cite[\S1.4]{HNY}.  The stack $\Bun_{\cG}$ is the union $\Bun_{\cG} = \bigsqcup_{\comp \in \Omega} \Bun_\cG^\comp$ of connected components indexed by $\comp\in \Omega$.  We let $\star_\comp$ denote the basepoint of $\Bun_\cG^\comp$.  
  Under the isomorphism $\Bun_\cG^0 \cong
\Bun_\cG^\comp$ of \cite[Cor.1.2]{HNY}, the basepoint $\star_\comp$ is the image of the point corresponding to the trivial bundle $\star$.

The stack of Hecke modifications is the stack which on a $\C$-scheme $S$ takes value the groupoid
$$
\Hecke_\cG(S):= \left \{(\E_1,\E_2,x,\phi) \mid \E_i \in \Bun_\cG(S), x: S \to \Gm, \phi: \E_1|_{(\P^1 \backslash \{x\}) \times S} \xrightarrow{\cong} \E_2|_{(\P^1 \backslash \{x\}) \times S} \right\}.
$$
It has two natural forgetful maps
\begin{align}\label{eq:Heckeproj}
\begin{diagram}
\node{} \node{\Hecke_\cG}\arrow {sw,t}{\pr_1} \arrow{se,t}{\pr_2}
\\ \node{\Bun_\cG} \node{} \node{\Bun_\cG \times \Gm,} \end{diagram}
\end{align}
given by  \(\pr_1(\E_1,\E_2,x,\phi)=\E_1\) and  \(\pr_2(\E_1,\E_2,x,\phi)=(\E_2,x)\).
The geometric fibers of $\pr_2$ over $\Bun_\cG \times \Gm$ are isomorphic to the affine 
Grassmannian $\Gr_G=G((\tau))/G[[\tau]]$, where  \(\tau\) is a local coordinate at  \(x\).

Let $\lambda$ be an integral coweight of $G$, and assume that $\Gr_\lambda$ lies in the $\comp$-component of $\Gr$.
The $G[[\tau]]$-orbits $\Gr_\lambda$ 
(and their closures $\oGr_\lambda$) in $\Gr_G$ define substacks $\Hecke_\lambda \subset \Hecke_\cG$ (and 
$\oHecke_\lambda $) \cite[p.259]{HNY}.

\subsection{Parametrization}\label{sec:param}
Assume now that $(G,P)$ are as in \S\ref{s:crystal-affine}, that is $G$ is simple and 
of adjoint type and  \(P=P_\ki\) is the maximal parabolic subgroup corresponding to the cominuscule node  \(\ki\). In particular $Z(L_P)$ is one-dimensional.
We now fix the isomorphism
\begin{equation}\label{eq:nu}
\alpha_\ki: Z(L_P) \cong \Gm \qquad z \mapsto \alpha_\ki(z).
\end{equation}
Via \eqref{eq:nu}, we may use ``$t$'' both as a coordinate on $\P^1$ and as a coordinate on $Z(L_P)$.

We follow \cite[\S5.2]{HNY} in the remainder of this subsection.
Let $\Hk$ be the restriction of $\pr_2:\Hecke_\cG \to \Bun_\cG \times \Gm$ to $\star_\comp \times \Gm \subset \Bun_\cG \times \Gm$ and for $q \in \Gm$, let $\Hk_q$ denote the restriction to $\star_\comp \times \{q\}$.  

By \cite[Cor.1.3]{HNY}, $\Bun^\comp_\cG$ contains an affine open substack isomorphic to $T \times I(1)/I(2)$, called the \emph{big cell}.  Let $\Hk^\circ \subset \Hk$ denote the inverse image of the big cell $T \times I(1)/I(2) \subset \Bun_\cG^0$ under $\pr_1$, and similarly define $\Hk_q^\circ$.  Denote the map 
\[\Hk^\circ \to T \times I(1)/I(2) \simeq T \times U_{-\theta} \times U/[U,U],\] 
by the following triple
$$
(f_T,f_0,f_+): \Hk^\circ \longrightarrow T \times U_{-\theta} \times U/[U,U].
$$
Our aim is to parametrize $\Hk^\circ$ and compute $f_T,f_0,f_+$.

Let $\E_0 = G \times \P^1$ be the trivial bundle and let $\E_\comp$ be the $\cG$-bundle corresponding to the basepoint $\star_\comp \in \Bun_\cG^\comp$.  The bundle $\E_\comp$ is obtained by gluing the trivial bundle on $\P^1 \backslash \{\infty\}$ with the trivial bundle on the formal disk around $\infty$ via the transition function $\comp(t^{-1}) =  \tau^{-\varpi_{\ki}^\vee} \dot w_P$.  We fix once and for all trivializations of $\E_0$ over $\P^1$ and of $\E_\comp$ over $\P^1 \backslash \{\infty\}$.

We use the local parameter $\tau = 1- t/q$ at $q$.  Thus $\tau = 0,1,\infty$ (or $\tau^{-1} = \infty,1,0$) corresponds to $t = q,0,\infty$ respectively.  Let 
\begin{equation}\label{eq:compdef}
 \comp(\tau^{-1}) = \dot \comp = \tau^{-\varpi_{\ki}^\vee} \dot w_P \in G[\tau,\tau^{-1}], \qquad \mbox{so that $\comp(\tau^{-1})|_{\tau^{-1} = 1} = \dot w_P$.}
\end{equation}

We view $\comp(\tau^{-1})$ as an isomorphism
$$
\comp(\tau^{-1}): \E_0|_{\P^1 \backslash\{q,\infty\}} \longrightarrow \E_\comp|_{\P^1 \backslash\{q,\infty\}},
$$
using the trivializations of $\E_0$ and $\E_\comp$ over $\P^1 \backslash \{\infty\}$.  Since 
\begin{equation}
\label{eq:taylor}
\tau^{-1} = -q t^{-1} + O((t^{-1})^2),
\end{equation}
the Laurent expansions of $\comp(\tau^{-1})$ and $\comp(t^{-1})$ in $t^{-1}$ differ by an element of $G[[t^{-1}]]$.  Thus  $\comp(\tau^{-1})$ extends to an isomorphism
\begin{equation}\label{eq:tauiso}
\comp(\tau^{-1}): \E_0|_{\P^1 \backslash \{q\}} \longrightarrow \E_\comp|_{\P^1 \backslash \{q\}}.
\end{equation}
Any point in $\Hk_q^\circ$ can be obtained by precomposing $\comp(\tau^{-1})$ with an element of \[\Aut(\E_0|_{\P^1 \backslash \{q\}}) \cong G[\tau^{-1}].\]  

From the definition of $\cG$, a bundle $\E \in \Bun_\cG$ is equipped with \emph{level structures} at $0$ and at $\infty$.  Let $\comp(\tau^{-1}) g(\tau^{-1})$ represent a point $\E \in \Hk_q^\circ$ under our parametrization.  We define the level structure (at 0 and at $\infty$) associated to $\E$ to be the pair
$$
(\ev_{t = 0}[\comp(\tau^{-1}) g(\tau^{-1}))^{-1}]  ,\ev_{t = \infty}[g(\tau^{-1})^{-1}]) = (g(1)^{-1} \dot w_P^{-1} ,g(0)^{-1} ) \in G \times G
$$
of elements of $G$.  (The isomorphism \eqref{eq:tauiso} preserves the level structure at $\infty$, and $\star_\comp$ has the trivial level structure at $\infty$, hence the formula $\ev_{t = \infty}[g(\tau^{-1})^{-1}]$ for the level structure at $\infty$.)

The big cell $T \times I(1)/I(2) \subset \Bun^0_\cG$ is the orbit of $\E_0$ under the action of the group $I_0^\opp(0) \times I(1)$ (recall that  \(T \cong I_0^\opp(0)/I_0^\opp(1)\), \(I(1) = I_\infty(1)\), and \( I(2) = I_\infty(2)\)).  It follows that $ \comp(\tau^{-1}) g(\tau^{-1})$ projects under $\pr_1$ to the big cell $\Bun^0_\cG$ if and only if
\begin{align*}
g(0)^{-1}   \in U  & \Leftrightarrow g(0) \in U\\
 g(1)^{-1} \dot w_P^{-1}  \in B_- &\Leftrightarrow  \dot w_P g(1)  \in B_-.
\end{align*}

We have a natural evaluation map $\ev_q: \Hk_q^\circ \to \Gr_q$ given by considering $ \comp(\tau^{-1}) g(\tau^{-1})$ as an element of $G((\tau))/G[[\tau]] \cong \Gr_q$.  The image $\ev_q(\Hk_q^\circ)$ is denoted $\Gr_q^\circ$.  We may further rigidify the moduli problem by precomposing with an element of $\Aut(\E_0) = \Aut(G \times \P^1) = G$ to obtain an isomorphism 
$$ \comp(\tau^{-1}) g(\tau^{-1}) g(0)^{-1}:\E_0|_{\P^1 \backslash \{q\}} \longrightarrow \E_\comp|_{\P^1 \backslash \{q\}},
$$ which is the identity at $\infty$.   Set $h(\tau^{-1}) := g(\tau^{-1}) g(0)^{-1} \in G[\tau^{-1}]_1$, where we recall that $G[\tau^{-1}]_1 = \ker(G[\tau^{-1}] \xrightarrow{\tau^{-1} = 0} G)$.  This gives the parametrization
$$
\Hk_q^\circ \cong \{h(\tau^{-1}) \in G[\tau^{-1}]_1 \mid h(1) \in \dot w_P^{-1} B_- U \}.
$$
Varying $q$, this gives the parametrization
\begin{equation}\label{eq:Hk}
\Hk^\circ \cong \{h(\tau^{-1}) \in G[\tau^{-1}]_1 \mid h(1) \in \dot w_P^{-1} B_- U\} \times \Gm. \end{equation}
Under this parametrization, the image of $h(\tau^{-1})$ in $\Gr_q \cong G((\tau))/G[[\tau]]$ is equal to $ \dot \comp h(\tau^{-1})$.

\begin{lemma}\label{lem:HNY}
Under the parametrization \eqref{eq:Hk}, write $h(1) = \dot w_P^{-1}  b_- u$ for $u \in U$ and $b_- \in B_-$.  Then we have
\begin{align*}
f_T(h,q) &= b_-^{-1} \mod U_- \in B_-/U_- \cong T,\\
f_+(h,q) &= u \mod [U,U] \in U/[U,U],\\
f_0(h,q) &= qa_{-\theta}(h) \in U_{-\theta} \cong \g_{-\theta},
\end{align*}
where $a_{-\theta}(h)$ denotes the $\g_{-\theta}$-part of the tangent vector $\left.\dfrac{dh(\tau^{-1})}{d(\tau^{-1})}\right|_{\tau^{-1}=0} \in \g$.
\end{lemma}
\begin{proof}
The formulae for $f_T$ and $f_+$ follow from the parametrization \eqref{eq:Hk}.  The function $f_0(h,q)$ is obtained by expanding $h(\tau^{-1})^{-1}$ at $t = \infty$ using the local parameter $t^{-1}$.  By \eqref{eq:taylor}, we have 
$$
\left.\frac{dh(\tau^{-1})^{-1}}{d(t^{-1})}\right|_{t^{-1}=0} = \left.\frac{d\tau^{-1}}{d(t^{-1})} \frac{dh(\tau^{-1})^{-1}}{d(\tau^{-1})}\right |_{\tau^{-1} = 0} = -q\left.\dfrac{dh(\tau^{-1})^{-1}}{d(\tau^{-1})}\right|_{\tau^{-1}=0} = q\left.\dfrac{dh(\tau^{-1})}{d(\tau^{-1})}\right|_{\tau^{-1}=0}\in \g,
$$
where for the last equality we have used the condition $h(0) = 1 \in G$: if $h(\tau^{-1})^{-1}= 1 + h_1 \tau^{-1} + O(\tau^{-2})$ then $h(\tau^{-1}) = 1 - h_1 \tau^{-1} + O(\tau^{-2})$.
\end{proof}

\subsection{Kloosterman $\aD$-module}\label{sec:KloostermanDmod}
Let $\lambda$ be an integral coweight for $G$. Let $\oHk_\lambda$ be the restriction of $\oHecke_\lambda$ to $\star_\comp \times \Gm$.  Define $\Hk_\lambda^\circ$ and $\Hk_{q,\lambda}^\circ$ by intersecting with $\Hk^\circ$ and $\Hk_{q}^\circ$ the substack $\Hecke_\lambda \subset \Hecke_\cG$. 

Let $\O_{\lambda}$ denote the structure sheaf of $\Gr_\lambda$, considered as a $\aD_{\Gr_\lambda}$-module.
Denote the minimal extension of $\O_{\lambda}$ under the inclusion $j: \Gr_\lambda \hookrightarrow \oGr_\lambda$
by $\aD_\lambda$.  Abusing notation, also denote by $\aD_\lambda$ the holonomic $\aD$-module on $\oHk_\lambda$ obtained via the isomorphism $\oHk_\lambda \cong \oGr_\lambda$.
We consider the following diagram, where we recall that the maps  \(f_+,f_0,\pr_2\) have been defined at the beginning of \S\ref{sec:param}:
\begin{align}\label{eq:def-Kloosterman}
\begin{diagram}
\node{} \node{\Hk^\circ_\lambda \hookrightarrow\oHk_{\lambda}}\arrow {sw,t}{(f_+,f_0)} \arrow{se,t}{\pr_2}
\\ \node{ U/[U,U] \times U_{-\theta}} \node{} \node{\star_\comp \times \Gm.} \end{diagram}
\end{align}

We write 
 \[(\phi_+,\phi_0):U/[U,U]\times U_{-\theta}\cong I(1)/I(2) \stackrel{\phi}{\rightarrow} \A^1.
\]
Recall our $\aD$-module conventions from Remark~\ref{rem:Dmod}, and recall that \(\Exp=D_{\A^1}/D_{\A^1}(\partial_x-1)\) denotes the exponential  \(D\)-module on  \(\A^1\).
We write  \(\Exp^{\phi_+}:=\phi_+^* \Exp\) and  \(\Exp^{\phi_0}:=\phi_0^* \Exp\).
Define \cite[(5.8)]{HNY} the \defn{Kloosterman $\aD$-module} by
\begin{equation}\label{d-Kloosterman}
\Kl_{(G^\vee,\lambda)} := \pr_{2,!}( f_+^* \Exp^{\phi_+}\otimes f_0^* \Exp^{\phi_0} \otimes D_\lambda).
\end{equation}

\begin{remark}
In~\cite[Thm.1]{HNY}, there is another definition of \(\Kl_{(G^\vee,\lambda)}\)
as the  \(\lambda\)-Hecke eigenvalue of an automorphic  \(D\)-module \(A_\cG\). We shall discuss  \(A_\cG\) in Section~\ref{s:eq-zhu} below.
In view of the results of \cite[\S5.2]{HNY}, the two definitions agree.
\end{remark}

\subsection{Comparison}
The inclusion $\iota_t: X_t \to G[\tau^{-1}]_1$ of \eqref{eq:inc} can be extended to an inclusion
$$
\tiota = (\iota,\pi): X \longrightarrow G[\tau^{-1}] \times \Gm,
$$
where for $x =u_1t \dot w_P u_2$ with $u_1 \in U_P$ and $u_2 \in U^{w_P^{-1}}$, we have via \eqref{eq:nu},
\begin{equation}\label{eq:iotax}
\iota(x) =\iota_t(x) = \dot \comp^{-1} t^{-1} u_1 t\dot \comp \in G[\tau^{-1}], \qquad \text{and} \qquad \pi(x) =\alpha_\ki(t) \in \Gm.
\end{equation}

\begin{lemma}\label{lem:comparison}
Under the identification \eqref{eq:Hk}, we have an isomorphism $\tiota: X \cong \Hk^\circ_{\varpi_\ki^\vee}$.
\end{lemma}
\begin{proof}
Fix $t \in Z(L_P)$, and let $q = \alpha_\ki(t) \in \Gm$.  We show that $\iota_t: X_t \cong \Hk^\circ_{q,\varpi_\ki^\vee} = \Hk^\circ_q \cap \Hk_\lambda$.  (In this proof, we explicitly distinguish $t$ and $q$ for clarity.)

Let $x = u_1 t \dot w_P u_2 \in X$.  Then by \eqref{eq:compdef}, we have
\begin{equation}\label{eq:imp}
\dot w_P \left(\dot \comp^{-1} t^{-1} u_1 t \dot \comp \right)|_{\tau^{-1} = 1} =   t^{-1}u_1 t \dot w_P= (t^{-1}x)(u_2^{-1}) \in B_-U,
\end{equation}
so $\iota(x) \in \Hk^\circ_q$.  The map $x \mapsto  \dot \comp \iota_t(x)$ is an inclusion $X_t \hookrightarrow \Gr_{\varpi_\ki^\vee}$; see \eqref{eq:intoGr}.  It follows that $\iota_t(X_t) \subset  \Hk^\circ_{q,\varpi_\ki^\vee}$ and the map $\iota_t$ is injective.

Under our identification \eqref{eq:Hk}, any element of $\Hk_{q,\varpi_\ki^\vee}$ is represented by $h(\tau^{-1}) = \dot \comp^{-1} g \dot \comp$ where $g \in G$ is constant.  It follows from Lemma \ref{lem:alphatheta}(a) that $h(\tau^{-1}) \in G[\tau^{-1}]_1$ is equivalent to $g \in U_P$.  The condition that $h(1)\in  \dot w_P^{-1} B_- U$ is then equivalent to $h(\tau^{-1}) \in \iota_t(X_t)$.  Thus, $\iota_t$ is surjective.  We are done.
\end{proof}

By Lemma \ref{lem:alphatheta}(c), we have $w_P^{-1} \alpha_\ki = -\theta$.  
Henceforth, we assume that 
\begin{equation}\label{eq:thetasign2}
\mbox{$x_{-\theta} \in \g_{-\theta}$ is equal to $ -\dot w_P^{-1} x_\ki \in \g_{-\theta}$.}
\end{equation}
Note that the choice \eqref{eq:thetasign2} is independent of \eqref{eq:thetasign}, which in the notation of this section is an assumption on the root vectors of $\g^\vee$ (rather than $\g$). 

\begin{proposition}\label{prop:comparison}
We have
\begin{align*}
f_T(\tiota(x)) &=t \gamma(x)^{-1}  \in B_-/U_- \cong T,\\
f_+(\tiota(x)) &= u_2^{-1} \mod [U,U] \in U/[U,U],\\
f_0(\tiota(x)) &= -\psi(u_1) x_{-\theta} \in U_{-\theta} \cong \g_{-\theta},
\end{align*}
where $\psi: U \to \A^1$ denotes the standard additive character from \S\ref{s:crystal-potential}.
\end{proposition}

\begin{proof}
Let $\iota(x) = h(\tau^{-1})$.  Then by \eqref{eq:imp}, we have
$$
\dot w_P h(1) = (t^{-1}x)(u_2^{-1}). 
$$
By Lemma \ref{lem:HNY} and the definition \eqref{eq:weight} of the weight map $\gamma$, we have 
$$
f_T(\tiota(x)) = x^{-1} t \mod U_- = \gamma(x)^{-1} t = t \gamma(x)^{-1} \in T$$ 
and $f_+(\tiota(x)) = u_2^{-1} \mod [U,U]$.

It remains to compute $f_0(\tiota(x))$.  Let $\u_P = \Lie(U_P)$.  Then $\u_P = \bigoplus_{\alpha \in R^+\setminus R^+_P} \g_\alpha$.  Since $\ki$ is cominuscule, $\alpha_\ki$ occurs in every $\alpha \in R^+\setminus R^+_P$ with coefficient one.  It follows that $\alpha +\beta$ is never a root for $\alpha,\beta \in R^+\setminus R^+_P$.  In particular, $\u_P$ is an abelian Lie algebra and $U_P$ is an abelian algebraic group, and so is $\dot \comp^{-1} U_P  \dot \comp$.  Let $\exp: \u_P \to U_P$ be the exponential map, which is an isomorphism since $U_P$ is unipotent.

Now, $\dot \comp^{-1} \g_\alpha \comp$ can be identified with the root space $\hat \g_{\comp^{-1} \cdot \alpha}$, where $\hat \g = \g[\tau^{\pm 1}]$ is the loop algebra, and $\comp^{-1} \cdot \alpha$ is now an affine root.  By Lemma \ref{lem:alphatheta}(c), we have $w_P^{-1} \alpha_\ki = -\theta$, and $\comp^{-1} \alpha_\ki = -\theta-\delta$, where $\delta$ denotes the null root.  It follows that 
\begin{equation}\label{eq:atheta}
a_{-\theta}(h) = \dot w_P^{-1} \cdot \exp^{-1}(u)_{\alpha_\ki} =-\psi(u)x_{-\theta},
\end{equation}
where $h(\tau^{-1}) = \dot \comp^{-1} u \dot \comp \in \dot \comp^{-1} U_P \dot \comp$ for $u \in U_P$, and $\exp^{-1}(u)_{\alpha_\ki}$ denotes the component of $\exp^{-1}(u) \in \u_P$ lying in the root space $\g_{\alpha_\ki}$.

We compute
\begin{align*}
f_0(\tiota(x)) &= q a_{-\theta}(h) & \mbox{by Lemma \ref{lem:HNY}} \\
&= -\alpha_\ki(t) \psi(t^{-1} u_1 t) x_{-\theta} &\mbox{by \eqref{eq:nu}, \eqref{eq:iotax} and \eqref{eq:atheta}}\\
&=- \alpha_\ki(t) \alpha_\ki(t^{-1}) \psi(u_1) x_{-\theta}= -\psi(u_1)x_{-\theta} \in \g_{-\theta},
\end{align*}
as claimed.
\end{proof}

We can now prove the main result of this section.

\begin{theorem}\label{t:HNYisBK}
The character
$\aD$-module $\BK_{(G,P)}$ is a flat connection (smooth and concentrated in one degree), and is isomorphic to the Kloosterman $\aD$-module $\Kl_{(G^\vee,\varpi_\ki^\vee)}$.
\end{theorem}
\begin{proof}
Recall~\eqref{d-Kloosterman} that by definition 
\[\Kl_{(G^\vee,\varpi_\ki^\vee)} := \pr_{2,!}( f_+^* \Exp^{\phi_+}\otimes f_0^* \Exp^{\phi_0} \otimes D_{\varpi_\ki^\vee}).\]
  The translation element $\tau^{\varpi_\ki^\vee}$ is minimal in the Bruhat order of $W_\af/W$.  Thus $\Gr_{\varpi_\ki^\vee} = \oGr_{\varpi_\ki^\vee}$, and we have $\aD_{\varpi_\ki^\vee} \cong
\O_{\Hk_{\varpi_\ki^\vee}}$.

Thus we may restrict ourselves 
to considering the diagram
\begin{align}\label{eq:restricted}
\begin{diagram}
\node{} \node{\Hk^\circ_{\varpi_\ki^\vee}}\arrow {sw,t}{(f_+,f_0)} \arrow{se,t}{\pr_2}
\\  \node{U/[U,U] \times U_{-\theta}} \node{} \node{\star_\comp \times \Gm.} \end{diagram}
\end{align}
Recall that $U_{-\theta}\cong \g_{-\theta}$ is identified with $\A^1$ via the root vector $x_{-\theta}$. 
 We then have 
$\Kl_{(G^\vee,\varpi_\ki^\vee)} = \pr_{2,!}( f_+^* \Exp^{\phi_+}\otimes f_0^* \Exp^{\phi_0})$.
By Lemma \ref{lem:comparison} and Proposition \ref{prop:comparison}, diagram \eqref{eq:restricted} gets identified with the diagram
\begin{align}\label{eq:geom}
\begin{diagram}
\node{} \node{X}\arrow {sw,t}{\Theta: \;u_1 t \dot w_P u_2 \mapsto (u_2^{-1},-\psi(u_1))} \arrow{se,t}{\pi}
\\ \node{U/[U,U] \times \A^1} \node{} \node{Z(L_P).} \end{diagram}
\end{align}
Setting $\phi:= (\phi^+,\phi_0)$, it follows that $\phi(u,a) = -\psi(u) - a$ for $(u,a) \in U/[U,U] \times \A^1$.  The definition of the character $\aD$-module can then be written as
$$
\BK_{(G,P)} = R\pi_*\Exp^f =R\pi_*(\Theta^*\Exp^{-\phi}) = R\pr_{2,*}( (f_+,f_0)^*\Exp^\phi) 
\overset{\sim}{\xrightarrow{}} 
\Kl_{(G^\vee,\varpi_\ki^\vee)},
$$
where $\Theta$ is the left arrow in \eqref{eq:geom}.
The last isomorphism is due to \cite[p.269]{HNY}; in the $D$-module setting it also follows from
the main result of \cite{Zhu:FG-HNY}.
It also follows from this calculation that $\BK_{(G,P)}$ is 
a $\aD$-module, rather than a complex of $\aD$-modules.
\end{proof}

\begin{remark} 
Similarly, over a finite field $\F_q$ equipped with a nontrivial additive character 
$\psi:\F_q\to \overline{\Q}_\ell^\times$, we can define the Artin--Schreier $\ell$-adic
sheaf $\cL_{\psi(f)}:=f^* \cL_\psi$ on $X$ and a geometric crystal $\ell$-adic sheaf $R\pi_*\cL_{\psi(f)}$ 
on $Z(L_P)$. The comparison with generalized Kloosterman $\ell$-adic sheaves is the same.
\end{remark}

\subsection{Homogeneity}\label{s:HNYgrade}
In \cite[\S2.6.4]{Yun:moments}, a $\G_m$-action is defined on $\Hk^\circ$. 
Under the parametrization \eqref{eq:Hk}, $\zeta \in \G_m$ acts by conjugation by $\rho^\vee(\zeta)$ on the first factor $G[\tau^{-1}]_1$ and by $q \mapsto \zeta^c q$ on the second factor $\Gm$, where $c$ is the Coxeter number.  The map $(f_+,f_0): \Hk^\circ \to I(1)/I(2)$ is $\G_m$-equivariant where $\G_m$ acts on $I(1)/I(2)$ by scalar multiplication in every affine simple root space.

The $\G_m$-action on $\Hk^\circ$ preserves $\Hk^\circ_{\omega_\ki^\vee}$, and under the isomorphism of the diagrams \eqref{eq:restricted} and \eqref{eq:geom}, this $\G_m$-action is identified with the one in \S\ref{s:homogeneous}.

\section{The mirror isomorphism for minuscule flag varieties}

\subsection{$\aD$-module mirror theorem}
Assume as before that $G$ is of adjoint type and $G^\vee$ is simply-connected.
Let $P\subset G$ be a parabolic subgroup and $P^\vee \subset G^\vee$ be the corresponding parabolic of the dual group.   
\begin{lemma}\label{l:2ipiH2}
There is a canonical exact sequence
\[
2i\pi H^2(G^\vee/P^\vee,\Z) \to  H^2(G^\vee/P^\vee) \to Z(L_P).
\]
\end{lemma}
\begin{proof}
By Borel's theorem there is a canonical isomorphism $H^2(G^\vee/P^\vee) \cong \kt^{W_P}$. We have $Z(L_P)=T^{W_P}$ and thus it only remains to apply the exponential map.
\end{proof}

Recall that the character $\aD$-module $\BK_{(G,P)}$ attached to the Berenstein--Kazhdan parabolic geometric crystal has been constructed
in~\S\ref{s:crystal}, and that the quantum connection $\cQ^{G^\vee/P^\vee}$ for the projective 
homogeneous space $G^\vee/P^\vee$ has been described in \S\ref{s:qh} in terms of the quantum 
Chevalley formula. The base of  $\BK_{(G,P)}$ is $Z(L_P)$, 
and the 
base  $\cQ^{G^\vee/P^\vee}$ is $\C^\times_q$.
Since $\C^\times_q \cong H^2(G^\vee/P^\vee)/2i\pi 
H^2(G^\vee/P^\vee,\Z)$ by Remark~\ref{rem:Cq}, via 
\[
(q_i \mid i \notin I_P) \mapsto \sum\limits_{i\in I\backslash I_P} \log(q_i)\sigma_i,
\]
the above Lemma~\ref{l:2ipiH2} shows that the 
two base tori 
are 
canonically 
isomorphic.

\begin{theorem}\label{t:Dmodule-mirror} Suppose that $P$ is a cominuscule parabolic subgroup of $G$ and let $P^\vee$ be the dual minuscule parabolic subgroup of $G^\vee$.
The geometric crystal $\aD$-module $\BK_{(G,P)}$  and the quantum connection 
$\cQ^{G^\vee/P^\vee}$ for $G^\vee/P^\vee$ are isomorphic.
\end{theorem}

\begin{proof} Let $\ki$ be the minuscule node corresponding to $P$. 
We claim that the isomorphism $Z(L_P)\isom  H^2(G^\vee/P^\vee)/2i\pi 
H^2(G^\vee/P^\vee,\Z)$ of Lemma~\ref{l:2ipiH2} factors as
the composition
\[
Z(L_P) \stackrel{\alpha_\ki}{\longrightarrow}
\Gm = \C^\times_q
\stackrel{\log}{\longrightarrow}
\C/ 2i\pi \Z
\stackrel{\sigma_\ki}{\longrightarrow}
H^2(G^\vee/P^\vee)/2i\pi H^2(G^\vee/P^\vee,\Z).
\] 
Indeed, the Schubert class $\sigma_\ki\in H^2(G^\vee/P^\vee,\C)$ corresponds to the fundamental coweight $\varpi_\ki^\vee\in
\mathfrak{t}^{W_P}$ under Borel's isomorphism. Thus composing with the exponential map, we 
see that the isomorphism
\[
\C_q^\times \stackrel{\log}{\longrightarrow}
\C/ 2i\pi \Z
\stackrel{\sigma_\ki}{\longrightarrow}
H^2(G^\vee/P^\vee)/2i\pi H^2(G^\vee/P^\vee,\Z) 
\isom 
Z(L_P)
=  T^{W_P} 
\]
is given by the cocharacter $q\mapsto \varpi_\ki^\vee(q)$. Composing with 
$\alpha_\ki$ the
claim follows from $\langle \alpha_\ki , \varpi^\vee_\ki \rangle =1$.

The proof of the theorem follows by combining the following three results:
\begin{itemize}
\item Theorem~\ref{t:HNYisBK} says that $\BK_{(G,P)}$ is isomorphic to the Kloosterman $\aD$-module
$\Kl_{(G^\vee,\varpi^\vee_\ki)}$ if we identify the respective bases $Z(L_P)$ and $\Gm$ via 
$\alpha_\ki$;

\item Zhu proved~\cite{Zhu:FG-HNY} that $\Kl_{(G^\vee,\varpi^\vee_\ki)}$ is isomorphic to the Frenkel--Gross connection
$\nabla^{(G^\vee,\varpi^\vee_\ki)}$ (see also Theorem \ref{thm:Zhu} below); 

\item Theorem~\ref{t:FGisquantum} says that $\nabla^{(G^\vee,\varpi^\vee_\ki)}$ is isomorphic to $\cQ^{G^\vee/P^\vee}$, if we
identify the bases via $\Gm = \C^\times_q$.
\end{itemize}
In Zhu's isomorphism the choice of affine generic character $\phi$ in the definition of 
$\Kl_{(G^\vee,\varpi^\vee_\ki)}$ matches with a particular choice of highest root vector in the 
definition of $\nabla^{(G^\vee,\varpi^\vee_\ki)}$ (see~Theorem \ref{thm:Zhu}).  All our sign choices lead to a 
single overall sign, which is equivalent to an isomorphism $q \mapsto \pm q$ of the curve $\Gm$.

To determine this sign and conclude that $\BK_{(G,P)}$ is isomorphic to $\cQ^{G^\vee/P^\vee}$, we consider the quantum period solution $\langle S(q),1 \rangle$ of $\cQ^{G^\vee/P^\vee}$.
From Lemma~\ref{l:q-term}, we know that the first term in the $q$-expansion is positive. 
On the other hand, the corresponding solution of $\BK_{(G,P)}$ is 
\begin{equation}\label{first-coeff}
\frac{1}{(2i\pi)^\ell}\oint e^{f_t(x)} \omega= \frac{1}{(2i\pi)^\ell}\oint e^{a_1 + \cdots + a_\ell + \alpha_\ki(t)P_\ki}
\frac{da_1}{ a_1} \cdots \frac{da_\ell}{a_\ell},
\end{equation}
where we use the expression of the superpotential from Corollary~\ref{cor:superpositive}. Since $P_\ki$ is a Laurent polynomial with positive
coefficients, and $\alpha_\ki(t)=q$, we deduce from Cauchy's residue theorem that the first term in the $q$-expansion of the above integral is also positive.
\end{proof}

If $G$ is of type $A_n$ this proves a conjecture of
Marsh--Rietsch~\cite[\S3]{Marsh-Rietsch:B-model-Grassmannians}, and if $G$ is of type 
$D_n$ a conjecture of
Pech--Rietsch--Williams~\cite[\S4]{Pech-Rietsch-Williams:LG-quadrics}. They construct in 
both cases a $\aD$-module homomorphism $\cQ^{G^\vee/P^\vee}\to \BK_{(G,P)}$ and show that it 
is injective.
The conjecture remained open whether it is an isomorphism, or equivalently whether the 
dimension of $H^*(G^\vee/P^\vee)$ is equal to the rank of $\BK_{(G,P)}$.
This is Theorem~\ref{t:Dmodule-mirror}.

\begin{corollary}\label{cor:identity}
Suppose that $P$ is a cominuscule parabolic of $G$.
The number of paths in Bruhat order inside $W^P$ from $\pi_P(w_0w_0^P s_\qroot)$ to 
$w_0w_0^P$ is equal to $P_\ki(1,1,\ldots,1)$ where $P_\ki$ is the Laurent polynomial of 
Corollary~\ref{cor:superpositive}.
\end{corollary}

\begin{proof}
Lemma~\ref{l:q-term} gives that the first term in the $q$-expansion of the quantum period 
is given by the above number of paths. In Example~\ref{ex:squarefree} we have seen that 
$P_\ki(a_1,\ldots,a_\ell)$ is the ratio of a square-free polynomial by the product $a_1\cdots 
a_\ell$. Hence~\eqref{first-coeff} evaluates to $P_\ki(1,1,\ldots,1)$. 
The corollary follows from Theorem~\ref{t:Dmodule-mirror}.
\end{proof}

\section{Generalization of Zhu's theorem}\label{s:eq-zhu}
In this section, we explain how Zhu's results in \cite{Zhu:FG-HNY} establish
Theorem \ref{thm:Zhu} below.  We assume the reader is familiar with \cite{Zhu:FG-HNY}.

\subsection{Deformation of the Frenkel--Gross connection}\label{s:EFG}
We use the notation from \S \ref{s:FG}, except that $G$ and $G^\vee$ are swapped.  We 
define a deformation of the Frenkel--Gross connection 
parametrized by $h\in \t^*$ as follows
\begin{equation}\label{EFG-def} 
\widetilde \nabla^{G^\vee} :=  d + (y_p + h) \frac{dq}{q} + x_\theta dq.
\end{equation}
This is a connection on the trivial $G^\vee$-bundle over $\Gm \times \t^*$ \emph{relative} to $\t^*$ (i.e., we view  the connection  \(1\)-form \((y_p+h)\frac{dq}{q} + x_\theta dq\) as a relative differential on  \(\Gm\times \t^*\) over  \(\t^*\) for the projection morphism \(\Gm \times \t^* \to \t^*\)).
Thus $\widetilde \nabla^{G^\vee}$ specialized to $0\in \t^*$ is the connection $\nabla^{G^\vee}$ 
of~\cite{Frenkel-Gross} 
considered 
in \S\ref{s:FG}. As before, the 
connection~\eqref{EFG-def} depends on a choice 
of basis vector $x_\theta$, but this choice is suppressed from the notation.  If the choice of 
$x_\theta\in \kg^\vee_\theta$ is not mentioned, by default we will use \eqref{eq:thetasign}.  As before we also 
have the associated vector bundle with connection
$\widetilde \nabla^{(G^\vee,\lambda)}=\widetilde \nabla^{(G^\vee,V_\lambda)}$ over $\Gm\times \t^*$ 
relative to $\t^*$.

\subsection{Rigid automorphic $D$-module}\label{s:Gross-form}
Recall from \S\ref{s:groupscheme} the standard affine character $\phi: I(1)/I(2) \to \A^1$. 
Recall that $\AA = \aD_{\A^1}/(\partial_x -1)$ denotes the exponential $\aD$-module on $\A^1$, and that we write
$
\Exp^\phi := \phi^* \Exp$ for the pullback  \(D\)-module on  \(I(1)/I(2)\).
Let $S=\operatorname{Sym}(\t)$ and identify the complex points of $\operatorname{Spec}(S)$ with $\t^*$.
Define the $(\aD_T\otimes S)$-module $\Mult_T$  as the free rank-one
$(\cO_T\otimes S)$-module with basis element
``$x^{h}$", for  \(h\in \Spec(S)=\kt^*\), with the action of the differential operator $\partial_k \in \aD_T\otimes S$ along \(k\in \kt \subset  \Fun(\kt^*)\) given by 
$$
\partial_k \cdot x^{h} := k \, x^h.
$$  
Equivalently,  \(\Mult_T\) is a rank-one connection on the trivial line bundle over  \(T\times \kt^*\) relative to  \(\kt^*\) (i.e., for every  \(h\in \kt^*\), and locally for 
 \(x\) in an open simply-connected open subset of \(T\) the horizontal sections of  \(\Mult_T\) specialized at  \(h\) are given by constant multiples of the function  \(x\mapsto x^h\)).

Recall our $\aD$-module conventions from Remark~\ref{rem:Dmod}.  Let $j_\comp: T \times I(1)/I(2) \hookrightarrow \Bun_\cG^\comp$ denote the inclusion of the big cell into the $\comp$-component of $\Bun_\cG$.  By \cite[Cor.1.3]{HNY}, $j_{\comp}$ is an affine open embedding.  For an affine map \(f\), we have  \(Rf_* = f_*\), and it follows that $Rj_{\comp,*}= j_{\comp,*}$ and $Rj_{\comp,!} = j_{\comp,!}$; see \cite[p.95]{HTT:D-modules}.

Now consider the $(\aD_{T\times I(1)/I(2)}\otimes S)$-module $\Mult_T\boxtimes \Exp^\phi$ on
$T\times I(1)/I(2)$. 
\begin{lemma}\label{lem:clean}
We have $j_{\comp,!}(\Mult_T \boxtimes \Exp^\phi)  \overset{\sim}{\xrightarrow{}} j_{\comp,*}(\Mult_T \boxtimes \Exp^\phi) $.
\end{lemma}
\begin{proof}
This is the $\aD$-module version of \cite[Lem.2.3]{HNY}.  We repeat the argument.  For $w \in W_\af - \Omega$, let $P_w$ denote the $\cG(1,2)$-bundle defined by a lift of $w$.  Pick $\alpha \in \Inv(w)$, and let $U_\alpha \subset I(1)$ denote the corresponding one-parameter subgroup.  We have an inclusion $U_\alpha \hookrightarrow \Aut(P_w)$, and a commutative diagram
\begin{align*}
\begin{diagram}
\node{U_\alpha \times \pt} \arrow{s,l}{{\rm id} \times P_w} \arrow{r,t}{\pi} \node{\pt}\arrow{s,r}{P_w}
\\ \node{U_\alpha \times \Bun_\cG} \arrow{r,t}{{\rm act}} \node{\Bun_\cG.} \end{diagram}
\end{align*}
Pulling back $j_{\comp,*}(\Mult_T \boxtimes \Exp^\phi)$ in two ways, we obtain
\[\Exp^\phi|_{U_\alpha} \boxtimes j_{\comp,*}(\Mult_T \boxtimes \Exp^\phi)|_{P_w} \cong \pi^* j_{\comp,*}(\Mult_T \boxtimes \Exp^\phi)|_{P_w}.\]
  Since $\Exp^\phi|_{U_\alpha}$ is defined  by a nontrivial character of $U_\alpha$, it follows that the stalk of $j_{\comp,*}(\Mult_T \boxtimes \Exp^\phi)$ vanishes at $P_w$.  Similarly, the stalk of $j_{\comp,!}(\Mult_T \boxtimes \Exp^\phi)$ vanishes at $P_w$.  Since this holds for all $w \in W_\af - \Omega$, the result follows.
\end{proof}

Following Heinloth--Ng\^o--Yun~\cite[Def.2.4 and p.269]{HNY}, we make the following definition.

\begin{definition}\label{def:HNYT}
Define $A_{\cG,T}$ to be the $(\aD_{\Bun_\cG} \otimes S)$-module on $\Bun_\cG$ given by $j_{\comp,!}(\Mult_T \boxtimes \Exp^\phi) 
= j_{\comp,*}(\Mult_T \boxtimes \Exp^\phi)$ on each connected component
$\Bun_\cG^\comp$.
\end{definition}
Thus $A_{\cG,T}$ is the intermediate, or minimal, extension of  \(D\)-modules; for example see \cite[\S3.4]{HTT:D-modules}.

\subsection{Twisted Kloosterman $\aD$-modules and statement of the main result}
It is established in~\cite[Thm.1]{HNY} that  \(A_{\cG,T}\) has the Hecke eigenproperty.
Let $\TKl_{G^\vee}$ denote the corresponding $G^\vee$-Hecke eigenvalue which is a  \(G^\vee\)-connection on $\Gm\times \kt^*$ relative to  \(\kt^*\).  
Our basis element  \(x^h\) for  \(h\in \kt^*=\Spec(S)\) corresponds to the multiplicative character  \(\chi\) in~\cite[Rem.2.5]{HNY} and the  \(\aD\)-module on  \(\Bun_\cG\) given by \(A_{\cG,T} \otimes_S h\) is denoted  \(A_{\phi,\chi}\) in~\cite[Thm.1]{HNY}. The  \(G^\vee\)-connection on  \(\Gm\) given by  \(\TKl_{G^\vee}\otimes_S h\) is denoted  \(\operatorname{Kl}_{{}^L\cG}(\phi,\chi)\) in~\cite[Thm.1]{HNY}.

\begin{theorem}
\label{thm:Zhu} 
There is a choice of basis vector 
$x_{\theta} \in \g^\vee_{\theta}$, which is compatible up to sign with \S\ref{sub:rootvectors}, such that the above  \(G^\vee\)-connections are isomorphic:
$$
\TKl_{G^\vee} \cong \widetilde \nabla^{G^\vee}.
$$
\end{theorem}
Specialized to $0 \in \t^*$, Theorem \ref{thm:Zhu} reduces to  \(
\Kl_{G^\vee} \cong \nabla^{G^\vee}\).  The proof of Theorem~\ref{thm:Zhu} occupies the rest of this section.  We will assume the reader is familiar with \cite{Zhu:FG-HNY}.

\subsection{Classical Hitchin map}
We use notation from \S \ref{s:groupscheme} and we let $\kv := I(1)/I(2)$ in the rest of this section.
\begin{lemma}\label{lem:good}
The stack $\Bun_\cG$ is good in the sense of Beilinson--Drinfeld; that is, we have $\dim 
T^*\Bun_\cG = 2 \dim \Bun_\cG$.
\end{lemma}
\begin{proof}
In \cite[Lem.17]{Zhu:FG-HNY} it is shown that \(\Bun_{\cG(0,1)}\) is good, where $\cG(0,1)$ is the group scheme over $\P^1$ obtained from the dilatation of the constant 
group scheme $G \times \P^1$ along $U \times \{\infty\} \subset G \times \{\infty\}$. The lemma follows after noting that $\Bun_\cG = \Bun_{\cG(1,2)}$ is a principal bundle over 
$\Bun_{\cG(0,1)}$ under the group $I^\opp_0(0)/I^\opp_0(1) \times I(1)/I(2) \cong T \times \kv$.
\end{proof}

Let $\c^* := \Spec \C[\g^*]^G \cong \t^* \sslash W$, where  \(W\) is the Weyl group.  We have a canonical $\G_m$-action on $\c^*$ 
coming from the scalar $\G_m$-action on $\g^*$. It gives rise to a decomposition $\c^* 
= \bigoplus_i \c^*_{d_i}$ into 
one-dimensional subspaces, where the integers $d_1\le d_2\le \cdots \le d_r$ are the degrees of $W$. Recall that  \(d_r=c\) is the Coxeter number of  \(\kg\).

Let $\cE \in 
\Bun_\cG$, and write $\cE' := \cE|_{\Gm} \in \Bun_{G \times \Gm}$.  The cotangent space 
$T^*_\cE\Bun_\cG$ maps to $\Gamma(\Gm, \g^*_{\cE'} \otimes \omega_{\Gm})$ where $\g^*_{\cE'}$ is the 
bundle on $\Gm$ associated to $\cE'$ and the coadjoint representation $\g^*$ of $G$.  The $G$-invariant 
map $\g^* \to \c^*$ gives rise, as $\cE$ varies, to a global analogue of the characteristic polynomial called the \emph{(global) Hitchin map} $h^{\cl}:T^*\Bun_\cG \to \Hitch(\Gm)$, where
\[\Hitch(\Gm) := \Gamma(\Gm, \c^* \times^{\G_m}
\omega_{\Gm}).\]
Let $\Hitch(\P^1)_{\cG}$ be the image of  \(h^{\cl}\), so that we write:
$$
h^{\cl}: T^*\Bun_\cG \to \Hitch(\P^1)_\cG \subset \Hitch(\Gm).
$$

We give an explicit description of $\Hitch(\P^1)_{\cG}$, following \cite{Zhu:FG-HNY}.  For a point $x \in \P^1$, we let $\cO_x$ denote the completed local ring at $x$ and $F_x = \Frac(\cO_x)$ 
denote its fraction field.  Denote by $D_x = \Spec 
\cO_x$ and $D^\times_x = \Spec F_x$ the formal disk and formal punctured disk at $x$.  We 
write $\omega_{\cO_x}$ for the ${\cO_x}$-module ${\cO_x} \cdot dt$ (after choosing a local 
coordinate $t$), and $\omega_F$ for $F_x \cdot dt$.  We have the \emph{local Hitchin map} (\cite[p.258]{Zhu:FG-HNY})
$$
h^{\cl}_x: \g^* \otimes \omega_F \to \c^* \times^{\G_m} \omega_F^\times =: \Hitch(D^{\times}_x),
$$
where $F = F_x$.

For $i = 0,1,2$, let $\p(i)=\p_\infty(i) \subset \g_\infty$ denote the Lie algebra of $I(i)=I_\infty(i)$.  Similarly, 
define $\p_0(i) \subset \g_0$ using $I^\opp_0(i)$.  For an $\cO$-lattice $\p \subset \g \otimes F$, we 
define $\p^\perp := \p^\vee \otimes_{\cO} \omega_\cO$, where $\p^\vee \subset \g^* \otimes F$ is the $\cO$-dual 
of $\p$.
The two local Hitchin maps at $x = 0$ and $x = \infty$ give the following two commutative diagrams 
(\cite[Rem.4.4]{Zhu:FG-HNY} and \cite[Prop.14]{Zhu:FG-HNY}\footnote{In \cite[\S4]{Zhu:FG-HNY}, $\p$ is an arbitrary parahoric subgroup.  We specialize to the case that $\p$ is the Iwahori.  The notation $\Hitch(D_\infty)_{1/c}$ matches \cite[p.272]{Zhu:FG-HNY}.}):
$$
\begin{diagram}
\node{\p_0(1)^\perp} \arrow{e} \arrow {s}{} \node{\t^* \cong \p_0(1)^\perp/\p_0(0)^\perp } \arrow{s}{}
\\ \node{\Hitch(D_0)_\RS} \arrow{e}{} \node{\c^*,}
\end{diagram}
\qquad 
\begin{diagram}
\node{\p(2)^\perp} \arrow{e} \arrow {s}{} \node{\kv^* \cong \p(2)^\perp/\p(1)^\perp} \arrow{s}{}
\\ \node{\Hitch(D_\infty)_{1/c}} \arrow{e}{} \node{\kv^* \sslash T,}
\end{diagram}
$$
where the local Hitchin spaces are defined by \cite[bottom of p.263]{Zhu:FG-HNY}\footnote{In the notation of \cite[p.259]{Zhu:FG-HNY}, the integer $m$ is equal to the Coxeter number $c$ for us.}
\begin{align}\label{eq:Hitch0}
\Hitch(D_0)_\RS &= \bigoplus_{1\le i\le r} \omega_{\cO_0}^{d_i}(d_i) \otimes \c^*_{d_i}, \\\label{eq:Hitchinfty}
\Hitch(D_\infty)_{1/c} &= \bigoplus_{1\le i < r} \omega^{d_i}_{\cO_\infty}(d_i) \otimes \c^*_{d_i} \bigoplus 
\omega^{c}_{\cO_\infty}(c + 1) \otimes \c^*_{c}.
\end{align}
The bottom map of the left diagram is obtained by taking the residue at $0$. 
The bottom map of the right diagram is explained in 
\cite[(4.11)]{Zhu:FG-HNY}.
It involves Kostant's section~\cite{Kostant:polynomial}
\begin{equation}\label{eq:Kos}
y_p + \mathfrak (\g^\vee)^{x_p} \isom 
\mathfrak c^*.
\end{equation}
This map is $\G_m$-equivariant for a suitable 
$\G_m$-action on $y_p + \mathfrak (\g^\vee)^{x_p}$ and the above mentioned $\G_m$-action on 
$\mathfrak c^*$; see~\cite[Prop.2.2]{Ngo:endoscopie}. Important for us later will be the 
corollary 
that 
$y_p + \mathfrak g^\vee_\theta\isom \mathfrak c^*_{c}$ under Kostant's section. 

The following result is a $\cG(1,2)$-version of a similar formula for $\cG(0,1)$ in \cite[p.270]{Zhu:FG-HNY}.

\begin{lemma}
We have an isomorphism
   \begin{equation}\label{eq:Hitch}
   \Hitch(\P^1)_\cG \cong \bigoplus_{1 \leq i < r} \Gamma(\P^1,\omega^{d_i}_{\P^1}(d_i \cdot 0 + d_i \cdot \infty)) 
   \otimes 
   \c^*_{d_i} \bigoplus \Gamma(\P^1,\omega^{c}_{\P^1}(c \cdot 0 + (c+1) \cdot \infty))\otimes \c^*_{c} \cong \A^r \times \A^1.
\end{equation}
\end{lemma}
\begin{proof}
By the same argument as in \cite[Lem.5]{Zhu:FG-HNY}, we have the description 
$$
\Hitch(\P^1)_\cG \cong \Hitch(D_0)_{\RS }\times_{\Hitch(D^\times_0)}\Hitch(\Gm) \times_{\Hitch(D^\times_\infty)} \Hitch(D_\infty)_{1/c}. 
$$
The explicit descriptions \eqref{eq:Hitch0} and \eqref{eq:Hitchinfty} give the first isomorphism in \eqref{eq:Hitch}.  For the second isomorphism we note that $\omega^d_{\P^1} \cong \cO_{\P^1}(-2d)$.  Thus $\omega^{d_i}_{\P^1}(d_i \cdot 0 + d_i \cdot \infty) = \cO_{\P^1}(-2d_i+2d_i) = \cO_{\P^1}$ and $\omega^{c}_{\P^1}(c \cdot 0 + (c+1) \cdot \infty) = \cO_{\P^1}(1)$.
\end{proof}

Let $\mu: T^*\Bun_\cG \to \t^* \times \kv^*$ be the moment map for the action of $T \times \kv$ on 
$\Bun_\cG$.

\begin{proposition}\label{prop:Zhucommute}
The following diagram is commutative, all maps surjective, the bottom map is an isomorphism and the top map is flat:
\begin{align}
\begin{diagram}\label{eq:hitchdiagram}
\node{T^*\Bun_{\cG}} \arrow{e,t}{\mu}  \arrow {s,t}{h^{\cl}} \node{\t^* \times \kv^*} \arrow{s}{}
\\ \node{\Hitch(\P^1)_\cG} \arrow{e,t}{ \cong } \node{\c^* \times \kv^*\sslash T.} \end{diagram}
\end{align}
\end{proposition}
\begin{proof}

The global Hitchin map  \(h^{\cl}\) embeds into the product of the local Hitchin maps $h^{\cl}_0$ and $h^{\cl}_x$ at $0$ and 
$\infty$, which establishes the commutativity of \eqref{eq:hitchdiagram}. 

The explicit description 
\eqref{eq:Hitch} of $\Hitch(\P^1)_\cG$ establishes the isomorphism of the bottom map (see 
\cite[(4.9) and Proof of Lem.19]{Zhu:FG-HNY}).  
The left vertical map of \eqref{eq:hitchdiagram} is surjective by definition.  The right vertical map of 
\eqref{eq:hitchdiagram} is a quotient map and thus surjective.  The top map  \(\mu\) of 
\eqref{eq:hitchdiagram} is surjective because $\Bun_\cG$ is a principal $T \times \kv$-bundle over  \(\Bun_{\cG(0,1)}\).

The proof of the last claim is identical to that of \cite[Lem.18]{Zhu:FG-HNY}, which we repeat.  
The Hamiltonian reduction $\mu^{-1}(0)/(T \times \kv)$ is naturally identified with 
$T^*\Bun_{\cG(0,1)}$.  As remarked in the proof of Lemma~\ref{lem:good}, $\Bun_{\cG(0,1)}$ is good, and thus $T^*\Bun_{\cG(0,1)}$ has 
dimension zero.  This implies that $\dim 
\mu^{-1}(0) = 
\dim(T \times \kv)$.  Let $W \subset T^*\Bun_{\cG}$ be the largest open substack such that the 
fibers 
of $\mu|_W$ have dimension $\dim(T \times \kv)$.  Then $W$ is $\G_m$-invariant, and since 
$\mu^{-1}(0) \subset W$, we have $W = T^*\Bun_{\cG}$, so all fibers of $\mu: T^*\Bun_{\cG} \to \t^* 
\times 
\kv^*$ have dimension $\dim(T \times \kv)$.  Since $\t^* \times \kv^*$ is smooth and $T^*\Bun_{\cG}$ is 
locally a complete intersection, we conclude that $\mu$ is 
flat.
\end{proof}

\subsection{Quantization}\label{s:quant}
We recall the descriptions of certain spaces of $\g^\vee$-opers from 
\cite[\S5]{Zhu:FG-HNY}.  (The Lie algebra $\g^\vee$ 
will be suppressed from the notation, so that we write  \(\Op\) for what is denoted  \(\Op_{{}^L\g}\) in \cite{Zhu:FG-HNY}).  Recall from \S\ref{s:principal-sl2} the principal $\sl_2$-triple $(x_p,2\rho^\vee,y_p)$ in 
$\g^\vee$.  The space $\Op(D_x^\times)$ of opers on the formal 
punctured disk centered at $x$ can be identified with the space of operators
\begin{equation*}
\ d + (y_p + (\g^\vee)^{x_p} \otimes F_x)dz,
\end{equation*}
where $z$ is a local coordinate at $x$.
The space $\Op(D_0)_{\RS}$ of opers on the formal disk centered at  \(0\) with regular singularities can be identified with 
the space of operators
\begin{equation}\label{oper-0}
 d + (y_p + (\g^\vee)^{x_p} \otimes \cO_0)\frac{dq}{q}.
\end{equation}
The space $\Op(D_\infty)_{1/c}$ of opers on the formal disk centered at  \(\infty\) with slope $\leq 1/c$ is the space of operators
\begin{equation}\label{oper-infinity}
\left( d + \left(\frac{y_p}{t} + \frac{1}{t} \b^\vee \otimes \cO_\infty + \frac{1}{t^2} \g^\vee_\theta \otimes 
\cO_\infty\right)dt\right ) / U^\vee(\cO_\infty),
\end{equation}
where $t = 1/q$.  The spaces $\Op(D_0)_{\RS}$ and $\Op(D_\infty)_{1/c}$ are subschemes of 
$\Op(D^\times_0)$ and $\Op(D^\times_\infty)$ respectively.

In \cite[\S2]{Zhu:FG-HNY}, a subscheme of opers $\Op(\P^1)_\cG \subset \Op(\Gm)$ is defined, and 
according to \cite[Lem.5]{Zhu:FG-HNY}, we have (cf. \cite[p.272]{Zhu:FG-HNY})
\begin{equation}\label{eq:oper}
\Op(\P^1)_\cG \cong \Op(D_0)_{\RS}\times_{\Op(D^\times_0)}  \Op(\Gm)\times_{\Op(D_\infty^\times)} 
\Op(D_\infty)_{1/c}.
\end{equation}
The description \eqref{eq:oper} is a quantization of \eqref{eq:Hitch}.

Let $U(\t)$ and $U(\kv)$ denote the universal enveloping algebras of $\t$ and $\kv$ and let $D'$ be the sheaf of algebras on 
the smooth site $(\Bun_\cG)_{sm}$ defined by Beilinson--Drinfeld
\cite[\S1.2.5]{Beilinson-Drinfeld}. The following variant of \cite[Lem.21]{Zhu:FG-HNY} is the quantization of Proposition 
\ref{prop:Zhucommute}.
\begin{proposition}\label{prop:Zhuquantize}
We have a commutative diagram of strict morphisms of filtered commutative algebras, where the top map is an isomorphism and 
the bottom map is flat:
\begin{align}\label{eq:quant}
\begin{diagram}
\node{ U(\t)^W \otimes U(\kv)^T} \arrow{e,t}{\cong}  \arrow {s}{} \node{\C[\Op(\P^1)_\cG]} \arrow{s,r}{h_\nabla}
\\ \node{U(\t) \otimes U(\kv)} \arrow{e}{} \node{\Gamma(\Bun_\cG,D').} \end{diagram}
\end{align}
\end{proposition}
\begin{proof}
The universal enveloping algebras $U(\t)$ and $U(\kv)$ have natural filtrations such that the associated graded algebras satisfy   
$\gr(U(\t))  \cong \C[\t^*]$ and $\gr(U(\kv)) \cong \C[\kv^*]$.  

The ring of functions $\C[\Op(\P^1)_\cG]$ has a filtration such that $\gr(\C[\Op(\P^1)_\cG]) \cong \C[\Hitch(\P^1)_\cG]$, and $\Gamma(\Bun_\cG,D')$ has a filtration such such 
that $\gr(\Gamma(\Bun_\cG,D')) \cong \Fun \, T^* \Bun_\cG$, where we write $\Fun$ to denote the commutative pro-algebra of 
regular functions on an affine ind-stack.  The right vertical map $h_\nabla$ is defined in \cite[(3.3)]{Zhu:FG-HNY}, and is a quantization of the classical Hitchin map $h^{\cl}$.  For these constructions, see \cite[p.254 and \S5.2]{Zhu:FG-HNY}.  The bottom horizontal map is explained in \cite[pp.255-256]{Zhu:FG-HNY}.  The top horizontal map is a quantization of the moment map $\mu:T^*\Bun_\cG \to \t^* \times \kv^*$ (\cite[Prop.15]{Zhu:FG-HNY}). Thus taking associate graded algebras of \eqref{eq:quant}, we recover \eqref{eq:hitchdiagram}.

The commutativity of \eqref{eq:quant} follows from commutative diagrams (see \cite[Prop.15]{Zhu:FG-HNY}) analogous to \eqref{eq:Hitch0} and \eqref{eq:Hitchinfty}.  

By Proposition \ref{prop:Zhucommute}, after taking associate graded algebras the top map is an isomorphism and the bottom map is flat; thus the same statements hold in \eqref{eq:quant}.
\end{proof}

\subsection{Proof of Theorem \ref{thm:Zhu}}
Let $\eta: \Op(D_0)_{\RS} \to \c^*$ be the residue map obtained from the description \eqref{oper-0} and Kostant's isomorphism \eqref{eq:Kos}.  Let 
\[\varpi:\kt^* \to \kt^*\sslash W = \kc^*\] 
be the projection map. We now compute the intersection $\Op(\P^1)_\cG \cap \eta^{-1}(\varpi(h))$ for  \(h\in \kt^*\), which corresponds to a slice of the isomorphism
\begin{equation}\label{OpP1-cong}
\Op(\P^1)_{\cG} \xrightarrow{\ \cong\ } \c^* \times \Spec U(\kv)^T
\end{equation}  
given by the top map of Proposition~\ref{prop:Zhuquantize}.

The space \(\Op(\Gm)\) of opers consists of operators of the form
$$
\nabla = d + y_p\frac{dq}{q} + v(q)dq,
$$
where $v(q) \in (\g^\vee)^{x_p}[q,q^{-1}]$.

Suppose moreover that  \(\nabla \in \Op(\P^1)_\cG \cap \eta^{-1}(\varpi(h))\).
The condition  \(\nabla \in \Op(D_0)_{\RS}\) that  \(\nabla\) has a regular singularity at $0$ implies~\eqref{oper-0}
that $v(q) \in q^{-1}(\g^\vee)^{x_p}[q]$.  Write $v(q) = a/q + v_0(q)$ with $v_0(q) \in 
(\g^\vee)^{x_p}[q]$ and  \(a\in (\g^\vee)^{x_p}\).  

The condition  \(\eta(\nabla)=\varpi(h)\) says that the
residue of 
$\nabla$ at  \(0\) is $\varpi(h) \in \c^*$.  By Kostant's theorem \cite{Kostant:polynomial}, the map $\g^\vee \to \g^\vee\sslash G^\vee \to 
\c^*$ induces the isomorphism \eqref{eq:Kos}.  Thus the element $a$ is uniquely determined by $\varpi(h)\in 
\kc^*$. We denote it by  \(a=a_h\in (\g^\vee)^{x_p}\).

Writing $t = 1/q$, the operator becomes
$$
\nabla = d - (y_p+a_h)\frac{dt}{t} - v_0(\frac{1}{t}) \frac{dt}{t^2}.
$$
The condition  \(\nabla \in \Op(D_\infty)_{1/c}\) at $\infty$ implies~\eqref{oper-infinity} that $v_0(\frac 1t)$ must be constant and must belong to the root space  \(\g^\vee_\theta=\C x_\theta\). 

Thus the 
space 
of 
opers $\Op(\P^1)_\cG \cap \eta^{-1}(\varpi(h))$ is the space of operators of the form
\begin{equation}\label{eq:fac}
\nabla =  d + (y_p+a_h)\frac{dq}{q} + \alpha x_\theta dq,
\end{equation}
for $\alpha \in \C$.  Thus $\Op(\P^1)_\cG \cap \eta^{-1}(\varpi(h)) \cong \A^1$. This bijection  \(\nabla \leftrightarrow \alpha\) corresponds to the isomorphism~\eqref{OpP1-cong} under the identification  \(\A^1 \cong \kv^*/T \cong \Spec(U(\kv)^T)\).  (The $(r+1)$-dimensional $T$-module $\kv$ has weights the simple roots $\alpha_1,\ldots,\alpha_r$ and the negative $-\theta$ of the longest root; hence $\kv^*/T \cong \A^1$.)

By construction, the two elements $y_p+ a_h$ and $y_p + h$ in $\g^\vee$ have the same 
image  \(\varpi(h)\) in $ \g^\vee \sslash G^\vee \cong \t^* \sslash W = \c^*$ and are therefore conjugate by a group 
element $g \in G^\vee$.  Again by Kostant's theorem, $U^\vee$ acts freely on $y_p + \b^\vee$ via 
the adjoint action and the quotient is isomorphic to $y_p + (\g^\vee)^{x_p}$.  Thus $y_p+ 
a_h$ 
and $y_p+h$ are conjugate by an element in the unipotent subgroup $U^\vee \subset G^\vee$. 
It follows that ${\rm Ad}_g(x_\theta) = x_\theta$.

Recall from \S\ref{s:EFG} the deformed Frenkel--Gross connection
\begin{equation} \label{eq:EFG2}
\widetilde \nabla^{G^\vee} = d+ (y_p+h) \frac{dq}{q} + x_\theta dq.
\end{equation}
We deduce from ${\rm Ad}_g(x_\theta) = x_\theta$ that the two connections \eqref{eq:fac} 
with $\alpha=1$ and 
\eqref{eq:EFG2} are gauge equivalent via a {\it constant} gauge transformation.

To complete the proof of Theorem \ref{thm:Zhu}, it remains to show that the twisted 
Kloosterman
$\aD$-module $\TKl_{G^\vee}$ is isomorphic to the connection \eqref{eq:fac} for some $\alpha 
=\pm 1$. This is achieved in the same manner as in
\cite[p.273]{Zhu:FG-HNY}.  Namely, we compare two automorphic $D$-modules.

We begin by constructing a Hecke eigen-$D$-module with Hecke eigenvalue equal to \eqref{eq:fac}.
Let $\phi: \kv \to \C$ be the standard affine character, inducing a character $\varphi: 
U(\kv) \to \C$.  The element $h \in \t^*$ also defines a character $\varphi_h: U(\t) \to \C$.  The $\aD$-module
$$
\Aut(h) := \omega_{\Bun_{\cG}}^{-1/2} \otimes (\aD' \otimes_{U(\t) \otimes U(\kv), \varphi_h \otimes \varphi} \C)
$$
is considered in \cite{Zhu:FG-HNY}, where $\aD'$ is the sheaf of critically twisted differential operators on $(\Bun_\cG)_{sm}$, and the tensor product is defined using the bottom map 
of Proposition \ref{prop:Zhuquantize}.  
Using the flatness statement in Proposition \ref{prop:Zhuquantize}, Zhu's result \cite[Cor.9]{Zhu:FG-HNY} states that $\Aut(h)$ is a Hecke eigen-$\aD$-module with the connection~\eqref{eq:fac} as its Hecke eigenvalue.

Finally, we argue that $\Aut(h)$ is isomorphic to the automorphic $\aD$-module of 
\cite{HNY} from Definition~\ref{def:HNYT}.  Recall from \S\ref{sec:param} the big cell $T \times \kv \cong \cBun_\cG \subset
\Bun_\cG$, which maps to the basepoint $\star \subset \Bun_{\cG(0,1)}$ corresponding to the trivial
$\cG(0,1)$-bundle.  By 
\cite[Rem.6.1]{Zhu:FG-HNY},
$\omega_{\Bun_{\cG}}^{-1/2}$ is canonically trivialized on $\cBun_\cG$.  It follows that the 
restriction of $\Aut(h)$ to
$\cBun_\cG \cong T \times \kv$ is isomorphic to $\Mult_T \boxtimes \Exp^\phi$.  Furthermore, $\Aut(h)$ is a 
$(T \times \kv,
\Mult_T \boxtimes \Exp^\phi)$-equivariant $\aD$-module on $\Bun_{\cG}$ (see \cite[\S{A.4.2}]{Yun:rigid} for the definition of equivariant).  By \cite[Rem.2.5]{HNY}, $\Aut(h)$ is automatically the (intermediate) clean extension of $\Aut(h )|_{\cBun_\cG}$.  Thus $\Aut(h)$ is
isomorphic to $A_{\cG,T}$ specialized at $h\in \t^*$ 
for which 
$\TKl_{G^\vee}$ specialized at $h\in \t^*$ is an eigenvalue.  This
shows that $\TKl_{G^\vee}$ specialized at $h\in \t^*$ is isomorphic to~\eqref{eq:fac} and 
thus to the deformed 
Frenkel--Gross connection~\eqref{eq:EFG2}, completing the proof.

\section{Equivariant quantum cohomology and weighted geometric 
crystals}\label{S:equivariant}

We extend the mirror isomorphism of Theorem~\ref{t:Dmodule-mirror} to the equivariant case. 
Recall from \S\ref{s:eq-zhu} the notation $S := \Sym(\t) \cong H^*_{T^\vee}(\pt)$.

\subsection{Equivariant quantum connection}\label{s:EQC}
Let $QH^*_{T^\vee}(G^\vee/P^\vee)$ denote the torus-equivariant small quantum cohomology 
ring of $G^\vee/P^\vee$.    It is an algebra over $\C[q_i \mid i \notin
I_P] \otimes S$.  For $w \in W^P$, we abuse notation by also writing $\sigma_w \in 
QH^*_{T^\vee}(G^\vee/P^\vee)$ for the equivariant quantum Schubert class. 
The following equivariant quantum Chevalley formula for a general $G^\vee/P^\vee$ is due to 
Mihalcea \cite{Mihalcea:EQ-Chevalley}.  
\begin{theorem}\label{thm:Mih}
For $w \in W^P$, we have in $QH^*_{T^\vee}(G^\vee/P^\vee)$
$$
\sigma_i *_q \sigma_w = (\varpi^\vee_i - w \cdot \varpi^\vee_i)\sigma_w + \sum_{\beta^\vee} \ip{\varpi^\vee_i,\beta} 
\sigma_{ws_\beta} + \sum_{\nu^\vee} \ip{\varpi^\vee_i,\nu} q_{\eta_P(\nu)} \sigma_{\pi_P(ws_\nu)},
$$
where $\varpi^\vee_i \in \t$ denotes a fundamental weight of $\g^\vee$, and $\beta^\vee,\nu^\vee$ 
denote roots of $\g^\vee$.  The last two summations are as in Theorem~\ref{thm:FW}.
\end{theorem}

We have a bundle map $QH^*_{T^\vee}(G^\vee/P^\vee) \to \t^*$.  We 
write $QH^*_h(G^\vee/P^\vee)$ for the fiber of this map over $h \in  \t^*$.  The algebra 
$QH^*_h(G^\vee/P^\vee)$ is again a free $\C[q_i \mid i\not\in I_P]$-module with Schubert basis 
$\{ 
\sigma_w \mid w 
\in W^P\}$.   

Now, assume that $P^\vee \subset G^\vee$ is minuscule.
Let $O(1)$ be the line bundle on $G^\vee/P^\vee$ arising from the natural embedding 
$G^\vee/P^\vee \hookrightarrow \P(V_{\varpi^\vee_\ki})$.
Consider the bundle $QH^*_{T^\vee}(G^\vee/P^\vee)$ over $\C^\times_q\times \t^*$ extended trivially 
to $\C^\times_q$. 
We define the equivariant quantum connection $\cQ^{G^\vee/P^\vee}_{T^\vee}$ to be the 
connection on  
$H^*_{T^\vee}(G^\vee/P^\vee)$ over $\C_q^\times\times \t^*$ relative to $\t^*$ by
$$
\cQ^{G^\vee/P^\vee}_{T^\vee} := d + c_1^T(O(1)) *_{q}\frac{dq}{q},
$$
where $c_1^T(O(1)) $ denotes the equivariant first Chern class of $O(1)$, and  $*_{q}$ denotes 
equivariant quantum multiplication.  We have that $c_1^T(O(1)) = 
\sigma_\ki - \varpi^\vee_\ki \sigma_\id$ in $QH^*_{T^\vee}(G^\vee/P^\vee)$, so by Theorem \ref{thm:Mih},
$$
c_1^T(O(1)) *_q \sigma_w = - w \cdot \varpi^\vee_i \sigma_w + \sum_{\beta^\vee} \ip{\varpi^\vee_i,\beta} 
\sigma_{ws_\beta} + \sum_{\nu^\vee} \ip{\varpi^\vee_i,\nu} q_{\eta_P(\nu)} \sigma_{\pi_P(ws_\nu)}.
$$

Theorem \ref{t:FGisquantum} has the following equivariant generalization.
\begin{theorem}\label{t:EFGisquantum}
If $P^\vee \subset G^\vee$ is minuscule with corresponding minuscule 
representation $V_{\varpi^\vee_\ki}$, then under the isomorphism
$L:H^*(G^\vee/P^\vee)\to V_{\varpi^\vee_\ki}$ of \eqref{eq:L}, the equivariant quantum connection 
$\cQ^{G^\vee/P^\vee}_{T^\vee}$ is isomorphic to the pulled-back connection 
$(\id_q\times\operatorname{inv})^*\widetilde \nabla^{(G^\vee,\varpi^\vee_\ki)}$ where 
$\operatorname{inv}:\t^* \to 
\t^*$ is 
the map $h\mapsto -h$, and $\id_q:\C_q^\times \to \C^\times_q$ is the identity map.
\end{theorem}
\begin{proof}
The extra term in $c_1^T(O(1)) *_q \sigma_w $, not present in the non-equivariant case is $ 
- w \cdot \varpi^\vee_\ki \sigma_w$.  Evaluating at $h\in \t^*$, we get the term $-\ip{w \cdot 
\varpi^\vee_\ki,h} 
\sigma_w$.  This agrees with the calculation $-h \cdot v_w = -\ip{w \cdot \varpi^\vee_\ki, h} v_w$ 
for $\g^\vee$ acting on $v_w \in V$.  The result then follows from the calculation in Theorem 
\ref{t:FGisquantum}.
\end{proof}

\subsection{Character $\aD$-module of a weighted geometric crystal}
Define the {\it weighted character $\aD$-module} of the geometric crystal $X$ by
\begin{equation}\label{def:Crequiv}
\WBK_{(G,P)} := R\pi_*(\Exp^f \otimes \gamma^* \Mult_T),
\end{equation}
where we recall that $\gamma: X \to T$ is the weight 
map \eqref{eq:weight}.
It is a $\aD$-module over $Z(L_P)\times \t^*$ relative to $\t^*$.  By taking the fiber over $h \in \t^*$, we obtain the $\aD$-module \eqref{eq:equivariantBK}.

Theorem \ref{t:HNYisBK} has the following generalization.

\begin{theorem}\label{t:EHNYisBK}  Suppose $P = P_\i$ is cominuscule and identify the bases 
$Z(L_P)\isom \Gm$ via $\alpha_\ki$.  Then the character
$\aD$-module $\WBK_{(G,P)}$ is isomorphic to a pull-back $(\operatorname{id}_q\times \operatorname{inv})^* 
\TKl_{(G^\vee,\varpi_\ki^\vee)}$ of the twisted Kloosterman $\aD$-module.
\end{theorem}
\begin{proof}
The proof is the same as that of Theorem \ref{t:HNYisBK}, so we sketch the main differences.  According to Proposition \ref{prop:comparison}, we have $t^{-1} f_T(\tiota(x)) =\gamma(x)^{-1}$.  Thus adding $f_T$ to the diagram \eqref{eq:geom} we can write, with $\Mult^{-1}$ denoting the pullback under inverse of the multiplicative $D$-module on $\Gm$,
\begin{align*}
\WBK_{(G,P)} &= R\pi_*(\theta^*\Exp^{-\phi}) \otimes f_T^* \Mult_T \otimes \pi^*\Mult^{-1}) \\
&= 
R\pi_*((f_+,f_0)^*\Exp^\phi \otimes f_T^* \Mult_T) \otimes \Mult^{-1} \\
&= \TKl_{G^\vee} \otimes \Mult^{-1},
\end{align*}
where we have used the projection formula (\cite[Cor.1.7.5]{HTT:D-modules}) for the second 
equality. 
For the third equality we have used a variant of Lemma~\ref{lem:clean}, namely that
 $\pr_{2,!}(\pr^*_1 A_{\cG,T} \otimes \aD_{\lambda}) = \pr_{2,*}(\pr^*_1 A_{\cG,T} \otimes \aD_{\lambda})$ and
$R^i\pr_{2,*}(\pr^*_1 A_{\cG,T} \otimes \aD_{\lambda}) = 0$ for $i > 0$ (cf. \cite[\S4.1]{HNY}).
Since $\Mult^{-1}$ is isomorphic to $\cO_{\Gm}$ as $\aD$-modules, the conclusion 
follows.
\end{proof}

\subsection{The equivariant mirror theorem}
Combining Theorems \ref{t:EFGisquantum}, \ref{thm:Zhu} and \ref{t:EHNYisBK} we obtain the 
following equivariant strengthening of Theorem \ref{t:Dmodule-mirror}.
\begin{theorem}\label{t:EDmodule-mirror} Suppose that $P$ is a cominuscule parabolic 
subgroup of $G$ and let $P^\vee$ be the dual minuscule parabolic subgroup of $G^\vee$.  We 
have an isomorphism
$$
\WBK_{(G,P)} \cong \cQ^{G^\vee/P^\vee}_{T^\vee}
$$
of $\aD$-modules over $Z(L_P) \times \t^*$ relative to $\t^*$.
\end{theorem}

In the case that $G^\vee/P^\vee$ is a Grassmannian, an injection from 
$\cQ^{G^\vee/P^\vee}_{T^\vee}$ into 
$\WBK_{(G,P)}$ is constructed by 
Marsh--Rietsch~\cite{Marsh-Rietsch:B-model-Grassmannians}*{Thm.5.5 and Thm.4.10}.

\section{The $\hbar$-mirror theorem}\label{s:mainresult}
Recall that $S = \Sym(\t)$.  We introduce an additional parameter $\hbar$ and work with $\aD_{\hbar} \otimes 
S$-modules. 
We shall establish Theorem~\ref{thm:Shbar}, which is a stronger version of Theorem 
\ref{t:EDmodule-mirror}.

\subsection{$\aD_{\hbar} \otimes S$-modules}\label{s:S-structures}
The definition of the sheaf $\aD_{\hbar,X}$ of \emph{$\hbar$-differential operators} on a complex 
smooth affine algebraic variety 
$X$ equipped with a $\G_m$-action is recalled in \S\ref{s:filtered}.  An \emph{$S$-structure} on a
$\aD_{\hbar,X}$-module $\cM$ is an action of $S$ on $\cM$ that commutes with the 
$\aD_{\hbar,X}$-action.  Equivalently, $\cM$ is a module for the sheaf
$\aD_{\hbar,X} \otimes S$, where elements of $S$ are considered ``scalars".  For any 
$\aD_{\hbar,X}$-module $\cM$, the sheaf $\cM \otimes S$ is a $\aD_{\hbar,X} \otimes S$-module.

Our basic example is the {\it multiplicative $\aD_{\hbar,T} \otimes S$-module} 
$\Mult^{1/\hbar}_T$ on 
$T$, defined as follows. 
\begin{definition}\label{def:mult} For a complex torus $T$ and $\kt=\operatorname{Lie}(T)$, let
\[
 \Mult^{1/\hbar}_T :=
\aD_{\hbar,T} \otimes S
/ \langle \xi_k - k\ |\ k\in \t \subset S \rangle.
\]
We give $T$ the trivial $\G_m$-action and furthermore declare that $k \in \t \subset S$ 
has degree one.  This gives $\Mult^{1/\hbar}_T$ the structure of a \emph{graded 
$D_{\hbar,T}$-module}.
\end{definition}

\begin{remark}
The $\aD_{\hbar,T} \otimes S$-module $\Mult^{1/\hbar}_T$ is a free $\cO_T \otimes 
S$-module 
with basis element
``$e^{\ell/\hbar}$", with the action of $\xi_k:= (h\mapsto \ip{k,h}) \in \aD_{\hbar,T}$ given by 
$$
\xi_k \cdot e^{\ell/\hbar} :=  k e^{\ell/\hbar},
$$
for $k \in \t \subset S$. Here $\xi_k$ should be thought of as ``$\hbar 
\partial_k$''; see \S\ref{s:filtered}.
And ``$\ell$" $:T \to \C[\t^*] = S$ should be thought of as the \emph{multi-valued function}
\begin{equation}\label{eq:ell}
\ell(t)(h) := \ip{\log(t),h}, \quad \text{where $h \in \t^*$.}
\end{equation}
\end{remark}

Recall that an \emph{$\hbar$-connection} on a bundle $E$ over $X$ is a $\C$-linear operator 
$\nabla: 
\Gamma(X,E) \to \Gamma(X,E) \otimes \Omega_X$ such that $\nabla(gs) = g\nabla(s) + \hbar s \otimes dg$ where 
$g \in \cO_X$ and $s \in \Gamma(X,E)$ are sections.  An $\hbar$-connection for $\hbar = 1$ is simply a 
connection in the usual sense.  An alternative description of $\Mult^{1/\hbar}_T$ is as follows: 
take the trivial $S[\hbar]$-bundle on
$T$ and equip it with the $\hbar$-connection 
$
\hbar d - k
$
where $k \in \t \subset S$.  

Suppose we have $\G_m$-actions on $E$ and $X$ such that the projection $E \to X$ is 
$\G_m$-equivariant.  We then say that the $\hbar$-connection $\nabla$ is \emph{graded} if 
$\hbar^{-1}\nabla$ is $\G_m$-equivariant, where $\hbar$ is taken to be degree one for the 
$\G_m$-action.  Equivalently, if $\nabla = \hbar d + \eta$, we require that $\eta$ has degree one 
for the $\G_m$-action.

\subsection{Frenkel--Gross connection revisited}

Let $V$ be a finite-dimensional $G^\vee$-module and let $\mu: V \times \t^* \to V$ denote the 
action map of $\t^*$.  Let $\mu^*: V \to V \otimes S$ denote the map  defined by $\mu^*(v) = v 
\otimes k$ if $v \in V$ has weight $k \in \t$.  By extending scalars, we obtain a map 
$\mu^*: V \otimes S \to V\otimes S$.  For a $G^\vee$-module $V$, define
the \emph{deformed Frenkel--Gross $\hbar$-connection}
$$
\widetilde \nabla^{(G^\vee,V)}_\hbar := \hbar d + (y_p+ \mu^*)\frac{dq}{q} + x_\theta dq,
$$
acting on the trivial $V \otimes S[\hbar]$-bundle on $\Gm$.  Thus for $\hbar =1 $ and $h \in \t^*$ inducing an evaluation homomorphism $h: S \to \C$, we have that $\widetilde \nabla^{(G^\vee,V)}_\hbar |_{\hbar =1} \otimes_S \C \cong 
\widetilde \nabla^{(G^\vee,V)}\otimes_S \C$ 
reduces to \eqref{EFG-def}.

Declaring that $k \in \t \subset S$ sits in degree one, the $\G_m$-action of \S\ref{s:FGgrade} 
extends to this deformation, so that the $1$-form $(y_p+ \mu^*)\frac{dq}{q} + x_\theta dq$ 
has degree one.

\subsection{Equivariant quantum connection revisited}
Define the \emph{equivariant $\hbar$-quantum connection}
$$
\cQ^{G^\vee/P^\vee}_{\hbar,T^\vee}:= \hbar d + c_1^T(O(1)) *_q \frac{dq}{q}.
$$
acting on the trivial bundle over $\Spec \C[q,q^{-1}]$ with fiber the equivariant cohomology 
$H^*_{T^\vee}(G^\vee/P^\vee) \otimes \C[\hbar]$.  Here
$c_1^T(O(1)) *_q $ denotes the equivariant quantum cohomology action.  
Then $\cQ^{G^\vee/P^\vee}_{T^\vee}$ 
from \S \ref{s:EQC} is equal to $\cQ^{G^\vee/P^\vee}_{\hbar,T^\vee}|_{\hbar=1}$.  

As in \S\ref{s:quantum-connection}, we define a $\G_m$-action on $QH^*_{T^\vee}(G^\vee/P^\vee)$ 
by using half the topological degree.  As before, $k \in \t \subset S$ sits in degree one.
The connection $1$-form $(\sigma_{\ki}*_{q,h}- \varpi_\ki^\vee )
 \frac{dq}{q}$ is then homogeneous of degree one for the $\G_m$-action. 

We then have the following variation of Theorem \ref{t:EFGisquantum}.
\begin{theorem}\label{t:hEFGisquantum}
If $P^\vee \subset G^\vee$ is minuscule and with corresponding minuscule representation $V_{\varpi^\vee_\ki}$, then under the isomorphism
$L:H^*(G^\vee/P^\vee)\to V_{\varpi^\vee_\ki}$ of \eqref{eq:L}, the equivariant quantum connection 
$\cQ^{G^\vee/P^\vee}_{\hbar,T^\vee}$ is identified with the deformed Frenkel--Gross
$\hbar$-connection $(\operatorname{id}_q\times \operatorname{inv})^*\widetilde \nabla^{(G^\vee,\varpi^\vee_\ki)}_\hbar$.
This is an isomorphism of graded $\hbar$-connections on $\Gm \times \t^*$ relative to $\t^*$. 
\end{theorem}

\begin{proposition}\label{p:hbar-quantum}
For any  $\lambda\in \C^\times$, there is an isomorphism
\[
\cQ^{G^\vee/P^\vee}_{\hbar,T^\vee}|_{\hbar = \lambda} \cong [q\mapsto q/\lambda^c ]^* [h\mapsto h/\lambda]^* 
\cQ^{G^\vee/P^\vee}_{T^\vee}
\]
of connections on $\Gm \times \t^*$ relative to $\t^*$.
\end{proposition}
\begin{proof}
Recall that $QH^*_{T^\vee}(G^\vee/P^\vee)$ is a graded ring with the topological degree 
$\deg(\sigma_w)=2\ell(w)$ and that it follows from
Lemma~\ref{l:coxeter-chern} that $\deg(q)=2c$. The gauge transformation $\sigma_w \mapsto 
\lambda^{\ell(w)} \sigma_w$ then gives the desired isomorphism between the two connections.
\end{proof}

\subsection{Twisted Kloosterman $\aD_\hbar$-modules}\label{ssec:twisted-Dhbar}
Define the \emph{exponential $\aD_{\hbar,\A^1}$-module} by $$
\Exp^{1/\hbar}:=\aD_{\hbar,\A^1}/\aD_{\hbar,\A^1}(\hbar\partial_x-1).
$$ 
Recall the generic affine character $\phi: I(1)/I(2) \to \A^1$, and let $\Exp^{\phi/\hbar} := 
\phi^*\Exp^{1/\hbar}$ denote the pullback.  

Let $A^{1/\hbar}_{\cG,T}$ denote the $\aD_\hbar \otimes 
S$-module on $\Bun_{\cG(1,2)}$ given by taking the $\aD_\hbar \otimes S$-module 
$\Mult^{1/\hbar}_T 
\boxtimes
\Exp^{\phi/\hbar}$ on $T \times I(1)/I(2)$ and pushing it forward to $\Bun_{\cG(1,2)}$.  We may define
the twisted Kloosterman $\aD_\hbar \otimes S$-module $\TKl^{1/\hbar}_{G^\vee}$ on $\Gm$ as 
$\TKl_{G^\vee} \otimes \C[\hbar]$.  As 
before, 
we have associated $\aD_\hbar \otimes S$-modules $\TKl^{1/\hbar}_{(G^\vee,V)}$ and 
$\TKl^{1/\hbar}_{(G^\vee,\varpi^\vee_\i)}$.
The $\G_m$-action of \S\ref{s:HNYgrade} gives the structure of a graded $\aD_\hbar \otimes 
 S$-module.

\begin{theorem}\label{t:eqzhu} Let $\varpi^\vee$ be a minuscule fundamental weight of $G^\vee$.
There is a choice of basis element $x_\theta \in \g^\vee_\theta$ such that we have an 
isomorphism of graded $\aD_{\hbar,\Gm} \otimes S$-modules
$$
\TKl^{1/\hbar}_{(G^\vee,\varpi^\vee)} \cong \widetilde \nabla^{(G^\vee,\varpi^\vee)}_{\hbar}.
$$
\end{theorem}
As before, we normalize conventions so that $\phi$ matches with 
the choice of $x_\theta$ from \eqref{eq:thetasign}.

\begin{proof}
The proof is a variation of the proof of Theorem \ref{thm:Zhu}.  It suffices to show that 
$\TKl^{1/\hbar}_{G^\vee} \cong \widetilde\nabla^{G^\vee}_{\hbar}$ holds when specializing $\hbar = 1$.  
Notationwise, the convention is that omitting $\hbar$ from the notation of a 
$\aD_\hbar$-module gives the corresponding $\aD$-module at $\hbar = 1$.  
Let $S=\Sym(\t)$ and $\iota: S^W \to S$ denote the natural inclusion.  Consider the automorphic 
sheaf
$$
\Aut_{\cG,T} := \omega_{\Bun_{\cG}}^{-1/2} \otimes (\aD' \otimes_{S^W \otimes U(\kv), \iota \otimes \varphi} (S 
\otimes \C)).
$$
defined using Proposition \ref{prop:Zhuquantize} and the natural isomorphism $U(\t) \cong S$.   The same argument as in the proof
of Theorem \ref{thm:Zhu} gives that $\Aut_{\cG,T}$ is a holonomic $\aD' \otimes S$-module.  The 
technology of
\cite{Zhu:FG-HNY} shows that $\Aut_{\cG,T}$ is a Hecke eigensheaf on $\Bun_\cG$.  Let $\cE$ 
denote 
its
Hecke eigenvalue and for a finite-dimensional $G$-module $V$, let $\cE^V$ denote its 
associated bundle.  Then $\cE^V$ is a
$\aD_{\G_m} \otimes S$-module isomorphic to $\nabla^{(G^\vee,V)}$.

On the other hand, as in the proof of Theorem \ref{thm:Zhu}, $\Aut_{\cG,T}$ restricted to 
$\cBun_\cG \cong T \times \kv$ is
isomorphic to $\Mult_T \boxtimes \Exp^\phi$.  Furthermore, $\Aut_{\cG,T}$ is a $(T \times \kv, \Mult_T \boxtimes
\Exp^\phi)$-equivariant $\aD$-module on $\Bun_{\cG}$.  It follows that $\Aut_{\cG,T} \cong A_{\cG,T}$.  
Thus
$\TKl_{ (G^\vee, V ) } \cong \widetilde \nabla^{(G^\vee,V)}$ for any $V$, or equivalently, $\TKl_{G^\vee} \cong 
\widetilde \nabla^{G^\vee}$.

We note that the $\G_m$-actions of \S \ref{s:FGgrade} and \S\ref{s:HNYgrade} are in 
agreement: they are both induced by the trivial $\G_m$-action on $T$, the dilation action 
on $\kv = I(1)/I(2)$, and the action $\zeta \cdot q= \zeta^c q$ of the curve $\Gm$ (noting 
that the Coxeter numbers of $G$ and $G^\vee$ coincide).  Thus $\TKl_{G^\vee} \cong \widetilde 
\nabla^{G^\vee}$ as filtered $\aD_{\Gm} \otimes S$-modules, where the filtration is induced by 
the $\G_m$-action on $\C^\times_q$ as explained in \S\ref{ssec:Dh}.
\end{proof}

\subsection{Weighted geometric crystal $\aD$-module revisited}\label{s:eq-Cr}
We use notation similar to \S\ref{s:crystal}.
Let $\WBK^{1/\hbar}_{(G,P)} := R\pi_*( \Mult^{\gamma/\hbar}_T \otimes \Exp^{f/\hbar})$ be the 
pushforward 
$\aD_{\hbar,Z(L_P)} \otimes
S$-module on $Z(L_P) \otimes S \cong \G_{m,\Sym(t)}$.  According to Proposition 
\ref{prop:Gmequiv}, we 
have 
that $\pi: X \to Z(L_P)$, $f: X \to \A^1$ and $\gamma: X \to T$ are $\G_m$-equivariant.  Thus 
$\WBK^{1/\hbar}_{(G,P)}$ acquires a natural structure of a graded $\aD_{\hbar,Z(L_P)} \otimes 
S$-module.  In Proposition \ref{prop:hfree}, we show that $\WBK^{1/\hbar}_{(G,P)}$ is 
$\hbar$-torsion-free.

\begin{proposition}\label{p:hbarBK}
(i) For any $\lambda \in \C^\times$, there is an isomorphism of $D_{Z(L_P)}$-modules
\[
\WBK^{1/\hbar}_{(G,P)}|_{\hbar = \lambda} \cong [q\mapsto q/\lambda^c ]^* [h\mapsto h/\lambda]^* \WBK_{(G,P)},
\]
where $c$ is the Coxeter number of $G$.

(ii) There is an isomorphism of $D_{\hbar,Z(L_P)}\otimes S$-modules
\[
\WBK^{1/\hbar}_{(G,P)} \cong \WBK_{(G,P)} \otimes_{\C} \C[\hbar].
\]
\end{proposition}
\begin{proof}
Assertion (i) follows from the homogeneity of the potential $f$ established in \S\ref{s:homogeneous} combined with Corollary~\ref{c:hbar-Cr} and
Lemma~\ref{l:coxeter-cominuscule}. Note that $\Mult^{\gamma/\hbar}_T$ is 
multiplicative and thus invariant under any Kummer
pullback.

From (i) we deduce that 
$\WBK^{1/\hbar}_{(G,P)}$ and $\WBK_{(G,P)} \otimes_{\C} \C[\hbar]$ are isomorphic after localizing 
$D_{\hbar,\Gm}$ at $(\hbar)$. 
Proposition~\ref{prop:hfree} says that $\WBK^{1/\hbar}_{(G,P)}$ is $\hbar$-torsion free, and 
$\WBK_{(G,P)}\otimes \C[\hbar]$ is also $\hbar$-torsion-free by
construction, hence the isomorphism extends to $D_{\hbar,\Gm}$.
\end{proof}

The following result has an identical proof to Theorem \ref{t:EHNYisBK}.
\begin{theorem}\label{t:hEHNYisBK}  Suppose $P = P_\i$ is cominuscule and identify the 
bases $Z(L_P)\isom \Gm$ via $\alpha_\ki$.  Then the graded 
character $\aD_{\hbar,Z(L_P)} \otimes S$-module $\WBK^{1/\hbar}_{(G,P)}$ is isomorphic to the 
graded Kloosterman $\aD_{\hbar,\Gm} \otimes S$-module
$(\operatorname{id}_q\times \operatorname{inv})^* \TKl^{1/\hbar}_{(G^\vee,\varpi_\ki^\vee)}$. 
\end{theorem}

\subsection{The $\aD_\hbar \otimes \operatorname{Sym}(\t)$ mirror theorem}
Combining Theorems \ref{t:hEFGisquantum}, \ref{t:eqzhu} and \ref{t:hEHNYisBK} we obtain the 
following result.

\begin{center}
\begin{tikzpicture}
\node at (0,1) {$B$-model};
\node at (10,1) {$A$-model};
\node at (10,-4.3) {Galois};
\node at (0,-4.3) {Automorphic};
\node[draw, align = center] (A) at (0,0) {$\WBK^{1/\hbar}_{(G,P)}$: character $\aD_\hbar$-module 
of the \\ weighted geometric crystal for $(G,P)$};
\node[draw,align=center] (B) at (10,0) {$\cQ^{(G^\vee,P^\vee)}_{\hbar,T^\vee}$: $T^\vee$-equivariant 
quantum \\
$\aD_\hbar$-module for $G^\vee/P^\vee$};
\node[draw,align=center] (C) at (0,-3) {$\TKl^{1/\hbar}_{ (G^\vee, V ) }$: twisted Kloosterman \\
$\aD_{\hbar}$-module for \\ 
minuscule representation $V$ of $G^\vee$};
\node[draw,align=center] (D) at (10,-3) {$\widetilde \nabla^{(G^\vee,V)}_{\hbar}$: deformed 
Frenkel--Gross 
\\ $\hbar$-connection for \\ minuscule 
representation $V$ of $G^\vee$};
\draw (A) -- 
node [above, align=center] {Theorem~\ref{thm:Shbar}} (B) 
-- node[right,align=center] {Theorem \ref{t:hEFGisquantum}} (D) -- node 
[below, align=center] {Theorem~\ref{t:eqzhu}}(C) 
-- node[left,align=center] {Theorem~\ref{t:hEHNYisBK} }
(A);
\end{tikzpicture}
\end{center}

\begin{theorem}\label{thm:Shbar}
Suppose $P$ is a cominuscule parabolic subgroup of an almost simple algebraic group $G$. 
We have an isomorphism of graded 
$\aD_{\hbar,Z(L_P)} \otimes \operatorname{Sym}(\t)$-modules
$$
\WBK^{1/\hbar}_{(G,P)} \cong \cQ^{G^\vee/P^\vee}_{\hbar,T^\vee}. $$
\end{theorem}

\section{Proof of the Peterson isomorphism}\label{s:proof-Peterson}

We deduce the equivariant Peterson isomorphism (Theorem~\ref{t:Jac}) by specializing 
$\hbar\to 0$ in Theorem~\ref{thm:Shbar}.

\subsection{The Gauss--Manin model}
Recall that $S=\Sym(\t)=\C[\t^*]$.  In this subsection, we describe the $\aD_{\hbar,\Gm} \otimes S$-module 
$\WBK^{1/\hbar}_{(G,P)}$ 
more explicitly.

By \cite{HTT:D-modules}*{Prop.1.5.28(i)} and Proposition~\ref{LaumonPushFiltre}, we may 
compute $R\pi_*(\Mult^{\gamma/\hbar}\Exp^{f/\hbar})$ by
computing the sheaf pushforward $\GM^\bullet_\hbar$ along $\pi$ of the relative de Rham 
complex $  
\DR^\bullet_{X/Z(L_P)}(\Mult^{\gamma/\hbar}\Exp^{f/\hbar})$.  Since $X \cong B_-^{w_P} \times Z(L_P)$ 
where 
$B_-^{w_P}$ and
$Z(L_P)$ are both affine (and thus also $\aD$-affine), it suffices to work with the modules 
of global sections. The complex $\GM^\bullet_\hbar$ is given by
$$
\Omega^0(X/Z(L_P)) \otimes_{\C} S \to \cdots \to  \Omega^{d-1}(X/Z(L_P)) \otimes_{\C} S  \to 
\Omega^{d}(X/Z(L_P)) \otimes_{\C} S,
$$
where $d := \dim X - \dim Z(L_P) = \dim B_-^{w_P}$, and $\Omega^k(X/Z(L_P))$ is the module of 
relative 
global differentials.  Here, the space of global
sections of the rank-one $\aD_{\hbar,X} \otimes S$-module $\Mult^{\gamma/\hbar}\Exp^{f/\hbar}$ has 
been 
identified 
with $S[X] = \C[X] \otimes S$ and the $\hbar$-differential is given accordingly by 
\[\hbar d + df + \gamma^{-1} d\gamma.
\] 
Here, the differential $\hbar d$ and the forms $df$ and
$d\gamma$ are both relative: no differentiation is made in the $Z(L_P)$-direction.  The form $d\gamma$ is the differential of the weight map, valued in $\t = \Lie(T)$.

By Proposition \ref{prop:hfree}, we know that $R\pi_*(\Mult^{\gamma/\hbar}\Exp^{f/\hbar})$ vanishes
except in one degree, so the only nonzero cohomology group of $\GM^\bullet_\hbar$ is
$$
\GM_\hbar:= {\rm coker} (\Omega^{d-1}(X/Z(L_P)) \otimes S \to \Omega^d(X/Z(L_P)) \otimes S).
$$

Now, $X\cong  B_-^{w_P}\times Z(L_P)$ is an open subset of affine space: specifically, 
$B_-^{w_P}$ is an open subset of a Schubert cell in $G/P$.
Let $x_1,x_2,\ldots,x_d$ be coordinates for this Schubert cell.  Let 
\[A := 
\Sym(\t)[Z(L_P)] = \C[X^*(Z(L_P))] \otimes S = \C[\t^*,q^{\pm 1}_i\ |\ i\not\in I_P],
\]  
which is a Laurent polynomial ring over $S$ in $\dim Z(L_P)$ variables, and 
$A[B_-^{w_P}]\cong S[X]$.
Then $\C[B_-^{w_P}]$ is a localization of $\C[x_1,\ldots,x_d]$, and we
have isomorphisms of $A[B_-^{w_P}]$-modules
\begin{align*}
\Omega^d(X/Z(L_P)) \otimes S &\cong A[B_-^{w_P}] \cdot \omega, \\
\Omega^{d-1}(X/Z(L_P)) \otimes S &\cong \sum_i A[B_-^{w_P}] \cdot \omega_i,
\end{align*}
where $\omega = \prod_{j=1}^d dx_j$ and $\omega_i = \prod_{j\neq i} dx_j$.  Thus
the Gauss--Manin module $\GM_\hbar$ can be written explicitly in terms of
coordinates by computing the partial derivatives $\frac{\partial f}{\partial x_j} + \gamma^{-1}\frac{\partial 
\gamma}{ 
\partial 
x_j}$.  Here, $\frac{\partial \gamma}{ \partial x_j}$ are the components of the differential $d\gamma$.

The Gauss--Manin module $\GM_\hbar$ is a $A\langle \hbar \partial_{q_i}\ |\ i\not\in I_P \rangle$-module 
where $\hbar 
\partial_{q_i}$ acts via the operator
$
 \hbar \partial_{q_i} + \frac{\partial f}{\partial q_i}.
$

\begin{remark}
We may write ``$f_S$" for the weighted \emph{multi-valued} potential, that is 
$f_S:=f+\ell \circ \gamma$, where $\ell$ is defined in \eqref{eq:ell}. The multi-valueness implies that 
it is not quite an element of $S[X]$. Hovewer its differential $df_S = df + 
\frac{d\gamma}{\gamma}$ is well-defined.
\end{remark}

\subsection{Peterson isomorphism}
Let 
$$
\Jac(X/Z(L_P),f,\gamma) := \Sym(\t)[X]/
(
\frac{\partial f}{\partial x_1} + \frac{\partial 
\gamma}{\gamma 
\partial 
x_1},\ldots,\frac{\partial f}{\partial x_d}+ \frac{\partial \gamma}{\gamma \partial x_d})
$$
denote the Jacobian ring of the weighted potential.
It is independent of the choice of coordinates of $B_-^{w_P}$ because it can be identified 
with the 
cokernel of the wedge map with $df+\frac{d\gamma}{\gamma}$ from $\Omega^{d-1}(X/Z(L_P))$ to
$\Omega^d(X/Z(L_P))$.

\begin{theorem}\label{t:Jac}
If $P^\vee$ is minuscule, then we have an isomorphism of 
$\Sym(\t)[Z(L_P)]$-algebras $$\Jac(X/Z(L_P),f,\gamma) \cong 
QH^*_{T^\vee}(G^\vee/P^\vee).$$
Moreover, multiplication by $q\frac{\partial f}{\partial q}$ on the left-hand side corresponds to 
quantum multiplication by
$c_1^T(O(1)) = \sigma_{\i} - \varpi^\vee_\ki$ on the right-hand side. 
\end{theorem}
\begin{proof}
Recall that $A=\Sym(\t)[Z(L_P)] = \C[\t^*,q^{\pm 1}]$.
By Theorem \ref{thm:Shbar} we have an isomorphism of $A\langle  \hbar\partial_q  \rangle$-modules between $\GM_{\hbar}$ and the equivariant quantum
connection $\cQ^{G^\vee/P^\vee}_{\hbar,T^\vee}$,
which is the $A\langle \hbar \partial_q \rangle$-module $QH^*_{T^\vee}(G^\vee/P^\vee) \otimes \C[\hbar]$ with the action of $\hbar q \partial_q$ given by 
$
 \hbar q\partial_q + (\sigma_{\i} *_{q}) - \varpi_\ki.
$

At $\hbar = 0$, the map is given by wedging with the relative differential 
$df+\frac{d\gamma}{\gamma}$, so 
we have
$
\GM_0 \cong \Jac(X/Z(L_P),f,\gamma)
$
as an $S[q^{\pm 1}]\langle\hbar  \partial_q \rangle$-module with the action of $\hbar \partial_q$ given by 
multiplication by $\frac{\partial f}{\partial
q}$ in the right-hand side which we denote by $\Jac(f,\gamma)$ for short.

Under the above isomorphism 
\[
\alpha: QH^*_{T^\vee}(G^\vee/P^\vee) \cong \Jac(f,\gamma)
\]
of $S[q^{\pm 1}]$-modules, quantum multiplication by
$\sigma_{\i}- \varpi_\ki$ corresponds to  multiplication by $q\frac{\partial f}{\partial q}$ in $\Jac(f,\gamma)$. 

Since $  QH^*_{T^\vee}(G^\vee/P^\vee) $ is a free $S[q^{\pm 1}]$-module we deduce that 
$\Jac(f,\gamma)$ is also free.
Let $\alpha(1_H)$ be the image of
the identity $1_H$ of the ring $H^*_{T^\vee}(G^\vee/P^\vee)$, and let $1_J \in \Jac(f,\gamma)$ 
denote the identity of the ring $\Jac(f,\gamma)$.
It also follows that there exists $\zeta
\in \Jac(f,\gamma)\otimes_S  \C(\t^*) $ so that $\alpha(1_H) \cdot \zeta = 1_J$.
Let $\zeta \alpha: H^*_{T^\vee}(G^\vee/P^\vee) \cong \Jac(f,\gamma)$ denote the composition
of the $S[q^{\pm 1}]$-module isomorphism $\alpha$ with left multiplication by $\zeta$.  Then 
$\zeta 
\alpha(1_H) =
1_J$ and $\zeta \alpha$ sends quantum multiplication by $\sigma_{\i}$ to multiplication by 
$q\frac{\partial f}{\partial q}$.

Recall that $S\cong \C[\t^*]$, so the fraction field is $\C(\t^*)$. 
By~\cite{Mihalcea:EQ-Chevalley}*{Cor.6.5} and 
\cite[Lem.4.1.3]{Ciocan-Fontanine-Kim-Sabbah:nonabelian-Frobenius}, 
$QH^*_{T^\vee}(G^\vee/P^\vee)\otimes_S \Frac(S)$
is generated over $\Frac(S)[q^{\pm 1}]$ by $\sigma_{\i}$, and thus also by $c_1^T(O(1)) = \sigma_\i- \varpi^\vee_\ki$. 
We deduce that $\zeta \alpha$ induces a  $\Frac(S)[q^{\pm 1}]$-algebra isomorphism after 
localization. 
Since the $S[q^{\pm 1}]$-algebras $QH^*_{T^\vee}(G^\vee/P^\vee)$ and $\Jac(f,\gamma)$ are already 
free as $S$-modules, it follows 
that $\zeta \alpha$ is an isomorphism of $S[q^{\pm 1}]$-algebras.
\end{proof}

Recall from \eqref{eq:Peterson} the definition of the Peterson stratum $\cY_P^*$.  Rietsch 
\cite{Rietsch:mirror-construction-QH-GmodP} has proved that $\Jac(X/Z(L_P),f,\gamma)$ is 
isomorphic to $\C[\cY_P^*]$.  We thus obtain the following corollary.

\begin{corollary} \label{cor:Peterson}
If $P^\vee$ is minuscule, then we have an isomorphism of $\Sym(\t)[Z(L_P)]$-algebras
\[
QH^*_{T^\vee}(G^\vee/P^\vee) \cong \C[\cY_P^*].
\]
\end{corollary}

\subsection{Example}
Consider the case $G^\vee/P^\vee = \mathrm{Gr}(1,n+1)= \P^{n}$. 
We have that the equivariant quantum $A\langle \hbar \partial_q \rangle$-module 
$\cQ^{\P^{n}}_{\hbar,T^\vee}$ is given by the connection
\[
\hbar q\frac{d}{dq} +
\left(
\begin{smallmatrix}
h_1 & & & q \\
1 & h_2 & & \\
& \ddots & \ddots & \\
& & 1 & h_{n+1}
\end{smallmatrix}
\right),
\]
where $\sum_{i=1}^{n+1} h_i = 0$, and we identify $h =(h_1,h_2,\ldots,h_{n+1}) \in \t^*$ in the usual way.
Its dual is isomorphic to $A\langle \hbar \partial_q \rangle/A\langle \hbar \partial_q \rangle L$, where 
\[
L:= \prod_{i=1}^{n+1} (\hbar q \frac{d}{dq} - h_i) - q.
\]
This is a hypergeometric differential operator of type ${}_0F_n$.
In the notation of~\cite{Katz:exp-sums-diff-eq}*{\S3}, we see that for $\lambda \in \C^\times$,
$\cQ^{\P^n}_{\hbar,T^\vee}|_{\hbar = \lambda}$ is 
the hypergeometric $\aD$-module $\cH_\hbar(\frac{h_i{}'s}{\lambda},\emptyset)$. 
On the other hand the character $A\langle \hbar \partial_q \rangle$-module $\WBK^{1/\hbar}$ is given 
by 
the $\pi_*$-pushforward of 
$\Mult^{\gamma/\hbar} \Exp^{f_q/\hbar}$, that is
\[
\int_{x_1\cdots x_{n+1}=q} x_1^{h_1/\hbar} \cdots x_{n+1}^{h_{n+1}/\hbar}
\Exp^{(x_1 + \cdots + x_n + x_{n+1})/\hbar}
\frac{dx_1 \cdots dx_{n+1}}{x_1\cdots x_{n+1}}.
\]
The mirror isomorphism $\cQ^{\P^n}_{\hbar,T^\vee}\cong \WBK^{1/\hbar}_{(G,P)}$ of 
Theorem~\ref{thm:Shbar} 
follows in this case from a result of Katz on convolution of hypergeometric 
$\aD$-modules~\cite{Katz:exp-sums-diff-eq}*{Thm.5.3.1}.  
In the semiclassical limit $\hbar\to 0$, we recover the equivariant quantum cohomology algebra
\[
QH^*_{T^\vee}(\P^n)=
\C[x,q^{\pm 1},\t^*]/
\left(
\prod^{n+1}_{i=1} (x-h_i)=q
\right),
\]
from the quantum connection $\cQ^{\P^n}_{\hbar,T^\vee}$ on the one hand. And on the other 
hand, from the potential function $f+\gamma^{-1}d\gamma$, and in view of 
\[
x_i \frac{\partial f}{\partial x_i}  + 
\gamma^{-1}\frac{x_i \partial \gamma}{\partial x_i}
=
x_i+h_i - \frac{q}{x_1\cdots x_n},
\]
we recover the Jacobi ring $\Jac(f,\gamma)$. By letting $x:=x_i+h_i$, which is independent of 
$i$, we see that $\frac{\partial f}{\partial x_i}+\gamma^{-1}\frac{\partial \gamma}{\partial x_i}=0$ is
equivalent to $\prod^{n+1}_{i=1}(x-h_i)=q$, in agreement with Theorem~\ref{t:Jac}.

\section{An enumerative formula}\label{s:parabolic-bessel}

In this section, we calculate the quantum period of minuscule flag varieties
(Theorem~\ref{t:hyper-series}).
The first coefficient in the $q$-expansion corresponds to the identity of 
Corollary~\ref{cor:identity}.

\subsection{Solution of the geometric crystal $\aD$-module}\label{s:sol}
We allow $P$ to be
arbitrary until \S\ref{s:rd-minuscule}. 
The Givental integral formula \cite{Givental:mirror} for Whittaker functions (see also 
\cite{GLO:Givental-Integral}) arises in the present 
context as solutions to $\BK_{(G,P)}$ via a natural pairing with homology groups. 
Equivalently these are special functions that are solutions of the quantum differential 
equation.  The final \S\ref{s:Bessel} treats the case of the classical $I_0$ and $K_0$-Bessel 
functions as an illustration of the main concepts.

The solution complex of $\BK_{(G,P)}$ is 
defined~\cite[\S4.2]{HTT:D-modules} to be 
\[
\mathrm{Sol}_{(G,P)}:=
\mathrm{RHom}_{\aD}(\BK^{\mathrm{an}}_{(G,P)},\cO^{\mathrm{an}}_{Z(L_P)}).
\]
Recall that $\BK_{(G,P)}=R\pi_* \Exp^{f}$.
By~\cite[Thm.4.2.5]{HTT:D-modules} we can interpret the stalks of $\mathrm{Sol}_{(G,P)}$ 
as dual to the algebraic de Rham
cohomology $H^\bullet_{\mathrm{dR}}(\mathring{G/P},\Exp^{f_t})$ for $t\in Z(L_P)$.  
Concretely, $ \mathrm{Sol}_{(G,P)}$ is the local system of holomorphic horizontal sections of 
the connection dual to $\BK_{(G,P)}$.

If $P$ is cominuscule, then by Theorem~\ref{t:Dmodule-mirror}, $\BK_{(G,P)}$ is a 
coherent $\aD$-module, hence $\mathrm{Sol}_{(G,P)}$ is a local
system on $Z(L_P)$.
For every $t\in Z(L_P)$, we deduce that $H^i_{\mathrm{dR}}(\mathring{G/P},\Exp^{f_t})$ is 
zero unless $i=d=\dim(G/P)$, and that
$
\dim H^d_{\mathrm{dR}}(\mathring{G/P},\Exp^{f_t}) 
$  
is constant and equal to $|W^P|$.

\subsection{Compact cycles}
The following proposition holds for any open Richardson variety so we state and prove it 
in that generality.
For $u \le w$ in $W$, recall that $\mathcal{R}_u^w$ denotes the open Richardson variety, defined to be the intersection of $B_- \dot{u} B/B$ with $B \dot{w} B/B$.
\begin{proposition}\label{p:Hmid}
 $H_{\ell(w)-\ell(u)}(\mathcal{R}_u^w)$ is one-dimensional.
\end{proposition}

\begin{proof}
By Poincar\'e duality it is equivalent to treat the cohomology with compact support 
$H^{\ell(w)-\ell(u)}_c(\mathcal{R}_u^w)$. By~\cite[Prop.4.2.1]{Riche-Soergel-Williamson} there is 
a canonical isomorphism 
\[
H^\bullet_c(\mathcal{R}_u^w) \cong \mathrm{Ext}^{\bullet+\ell(u)-\ell(w)}(M_w,M_u),
\]
 where $M_w$ and $M_u$ denote the Verma modules in the principal block.
Since $\mathcal{R}_u^w$ has real dimension $2(\ell(w)-\ell(u))$ we have 
$H^{\ell(w)-\ell(u)}_c(\mathcal{R}_u^w) \cong \mathrm{Hom}(M_w,M_u)$.
This space is one-dimensional as follows from the BGG-correspondence.
\end{proof}

To construct a middle dimension cycle generating $H_{\ell(w)-\ell(u)}(\mathcal{R}_u^w)$, we use 
that $\mathcal{R}_u^w$ contains many tori. (In fact by 
Leclerc~\cite{Leclerc:cluster-Richardson}, $\C[\cR_u^w]$ contains a cluster algebra, and is 
conjectured to be equal to one.)
We choose any cluster torus $(\C^\times)^{\ell(w)-\ell(u)} \subset \mathcal{R}_u^w$ and consider the middle dimension cycle given by a compact torus $(S^1)^{\ell(w)-\ell(u)}$. We denote integration along this cycle by $\oint$.
We can normalize the form $\omega$ from~\cite{KLS:projections-Richardson} which has simple poles along the boundary of $\mathcal{R}_u^w$ such that 
\[
\oint \frac{\omega}{(2i\pi)^{\ell(w)-\ell(u)}} = 1.
\]
In view of Proposition~\ref{p:Hmid}, the cycle is well-defined and independent of the choice of tori.

Recall from \S\ref{s:Richardson} that $\mathring{G/P}\cong \mathcal{R}^{w_0}_{w_Pw_0}$ which can be identified with the open projected
Richardson variety in $G/P$. Thus we have shown that the space $H_{d}(\mathring{G/P})$ is 
one-dimensional and generated by the above
compact cycle.  For the case of full flag varieties $G/B$ a related construction appears in \cite[\S7.1]{Rietsch:mirror-solution-Toda}, and for the
Grassmannian in~\cite[Thm.4.2]{Marsh-Rietsch:B-model-Grassmannians}. 

\subsection{Poincar\'e duality}\label{s:rd-minuscule}
For $w\in W^P$, we define $\PD(w):=w_0ww_0^P$, which is still an element of $W^P$.
This is an involution and we have $\ell(\PD(w))=d-\ell(w)$.
Moreover the Schubert class $\sigma_w\in H^{2\ell(w)}(G^\vee/P^\vee)$ is Poincar\'e dual to $\sigma_{\PD(w)}\in H^{2\ell(\PD(w))}(G^\vee/P^\vee)$.
Since $G^\vee/P^\vee$ is minuscule, a reduced expression for $w\in W^P$ is unique up to commutation relations.
It is always~\cite[\S2.4]{Chaput-Manivel-Perrin:QH-minuscule} a subexpression in any reduced expression for the longest element $w_P^{-1}=w_0w_0^P=\PD(1)$ of $W^P$.

\subsection{Givental fundamental solution} 
Givental has introduced solutions $S_w(\hbar,q)$ of the quantum $\hbar$-connection 
$\cQ^{G^\vee/P^\vee}_\hbar$ in terms of a generating series of gravitational descendants of 
Gromov--Witten
invariants; see~\cite{Givental:equivariant-GW}, \cite[\S4.1]{BCKS:acta}, 
\cite[\S10]{Cox-Katz:mirror},  \cite[\S5]{Guest:book:QH} and \cite[\S2.3]{Iritani:integral-structure} for details.
The functions $S_w(1,q)$, for $w\in W^P$, form a fundamental solution of 
$\cQ^{G^\vee/P^\vee}$ near the regular singular point $q=0$,
see~\cite[\S2]{Galkin-Golyshev-Iritani:gamma-conj}.

The Givental $J$-function is defined by
\[
J^{G^\vee/P^\vee}(\lambda,q):=\sum_{w\in W^P}
\langle S_w(\lambda,q), 1 \rangle
\sigma_{\PD(w)}.
\]
It gives rise to a multivalued holomorphic section
\[
J^{G^\vee/P^\vee} : \C^\times_\lambda \times \C^\times_q \to H^*(G/P),
\]
which becomes single-valued when factored through the universal cover $H^2(G^\vee/P^\vee)\to \C^\times_q$. 
Using the notation of~\cite[Lem.10.3.3]{Cox-Katz:mirror},
\[
J^{G^\vee/P^\vee}(\lambda,q)=
 \exp\left({\dfrac{\log q}{\lambda}\sigma_\ki}\right)
\left(
1 + \sum_{e=1}^\infty \sum_{w\in W^P}
q^e\left\langle
\frac{\sigma_w}{\lambda - \kc}
\right\rangle_{0,e}
\sigma_{\PD(w)}
\right)
.
\]
Intrinsically, the $J$-function is the solution to the \emph{dual} connection to 
$\cQ^{G^\vee/P^\vee}_\hbar$ that is asymptotic to $1$ as $q$ approaches the
regular singular point $0$; see~\cite{Galkin-Iritani:Gamma-conj}.

\begin{example}
For $\P^n$, we have~\cite[\S10]{Cox-Katz:mirror}
\begin{equation}\label{J-Pn}
J^{\P^n}(\lambda, q)= \exp\left({\dfrac{\log q}{\lambda}\sigma_\ki}\right) \sum^\infty_{e=0} q^e 
\prod^e_{j=1} 
\frac{1}{(\sigma_\ki+j\lambda)^{n+1}}.
\end{equation} 
The case of quadrics is treated in~\cite[\S5]{Pech-Rietsch-Williams:LG-quadrics}.
\end{example}

Of particular importance is the component $\langle J^{G^\vee/P^\vee}(\lambda,q),\sigma_{\PD(1)} 
\rangle=  \langle S_{\PD(1)}(\lambda,q),1 \rangle$ which is a
power series in $\lambda^{-1},q$. In \S\ref{sub:quantum-period} we used the notation $S(q)$ 
for 
$S_{\PD(1)}(1,q)$.
The single-valuedness follows from considering the kernel of the monodromy
operator which is the usual cup product with $\sigma_\ki$.
Precisely,
\begin{equation}\label{def:hyp-series}
\langle S_{\PD(1)}(\lambda,q),1\rangle 
= 1 + \sum_{e=1}^\infty
q^e
\left\langle
\frac{\sigma_{\PD(1)}}{\lambda-\kc},1
\right\rangle_{0,e}.
\end{equation}
It is called the hypergeometric series of $G^\vee/P^\vee$ in~\cite{BCKS:conifold,BCKS:acta} and called the quantum period
in~\cite{Galkin-Iritani:Gamma-conj,Galkin-Golyshev:QH-Grassmannians}.
The sum can be simplified further by expanding $(\lambda-\kc)^{-1}$ in power series of 
$\lambda^{-1}$; see~\cite[\S5.2]{Pech-Rietsch-Williams:LG-quadrics} who
also consider more generally $\langle S_{\PD(1)}, \sigma_w \rangle$ for any $w\in W^P$.

\subsection{Degrees and irregular Hodge filtration}
The isomorphism $
\WBK_{(G,P)} \cong \cQ^{G^\vee/P^\vee}_{T^\vee}$
from Theorem \ref{thm:Shbar} induces 
for every $t\in Z(L_P)$ an isomorphism 
\begin{equation}\label{mirror-h}
\bigoplus^d_{i=0} H^{2i}(G^\vee/P^\vee) \cong H^d_{\mathrm{dR}}(\mathring{G/P},\Exp^{f_t}).
\end{equation}
The cohomology group on the right hand side is concentrated in a single degree since $\WBK_{(G,P)}$ is a $\aD$-module concentrated in a single degree.  In this isomorphism, the left-hand side visibly carries a gradation by degree, which can be transported to the right-hand side. We
want to spell this out precisely and derive an important proposition.

It is easy to see that the filtration associated to the Jordan decomposition of the linear endomorphism given by the cup-product by $\sigma_\ki$ coincides
with the filtration by degree on $H^*(G^\vee/P^\vee)$. The cup-product by $\sigma_{\ki}$ is the 
monodromy at $q=0$ of the connection $\cQ^{G^\vee/P^\vee}$.
Thus we conclude from the mirror isomorphism $\cQ^{G^\vee/P^\vee}\cong \BK_{(G,P)}$ that the 
filtration by degree is transported on
$H^d_{\mathrm{dR}}(\mathring{G/P},\Exp^{f_t})$ to the monodromy filtration of $\BK_{(G,P)}$. 

\begin{proposition}\label{p:one-to-omega}
In the isomorphism~\eqref{mirror-h}, the line $\C \cdot \sigma_{PD(1)}=H^{2d}(G^\vee/P^\vee)$  spanned by the top class corresponds to the line spanned by the cohomology class of the
volume form $\omega$ from~\S\ref{s:Richardson}.
In particular this cohomology class is nonzero in  \(H^d_{\mathrm dR}(\mathring{G/P},\Exp^{f_t})\).
\end{proposition}

This was previously established for Grassmannians by 
Marsh--Rietsch~\cite{Marsh-Rietsch:B-model-Grassmannians} and for quadrics by 
Pech--Rietsch--Williams~\cite{Pech-Rietsch-Williams:LG-quadrics}.

\begin{proof}
In view of the above discussion we only need to analyse the monodromy filtration on 
$H^d_{\mathrm{dR}}(\mathring{G/P},\Exp^{f_t})$ near
$\alpha_\ki(t)=0$.
A convenient way to do so is via the Kontsevich complex $\Omega^\bullet_{f_t}$ of $f_t$-adapted log-forms, which again involves a resolution
$\widetilde{G/P}$ of the singularities of $(G/P,f_t)$. It is established 
in~\cite[Cor.1.4.8]{Esnault-Sabbah-Yu} that
\[
H^d_{\mathrm{dR}}(\mathring{G/P},\Exp^{f_t}) \cong \bigoplus_{p+q=d}
H^q\left(\widetilde{G/P}, \Omega_{f_t}^p\right).
\]
It is possible to write down the monodromy operator and to verify that the decreasing 
monodromy filtration corresponds to the gradation by $p-q$, 
following~\cite{Katzarkov-Kontsevich-Pantev:LG-models,Esnault-Sabbah-Yu,Harder:phd}.

Then by the above $H^{2d}(G^\vee/P^\vee)$ corresponds under the isomorphism~\eqref{mirror-h} to the space $H^0(\widetilde{G/P},\Omega_{f_t}^d)$ where by the
definition of $\Omega^d_f$, this coincides with the space $H^0(\widetilde{G/P},\Omega^d(\log))$ of log differential holomorphic top forms. It is known from~\cite{KLS:projections-Richardson}
that $H^0$ is one-dimensional and spanned by the form $\omega$.
\end{proof}

In the isomorphism~\eqref{mirror-h}, the left-hand side is of Hodge--Tate type, namely 
$H^{2i}=H^{(i,i)}$, because it is spanned by the Schubert classes
$\sigma_w$ which are algebraic. In the mirror isomorphism, $H^*(G^\vee/P^\vee)$ being of 
Hodge--Tate type translates to $(\mathring{G/P},f_t)$ being pure in
the sense that $\BK_{(G,P)}$ is a complex supported in one-degree.  
We obtain also parts of~\cite[Conj.3.11]{Katzarkov-Kontsevich-Pantev:LG-models} 
concerning the matching of nc-Hodge structures on both sides
of~\eqref{mirror-h}, in particular on the right-hand side we have identified the irregular Hodge filtration constructed by Deligne and
J.-D.~Yu~\cite{Esnault-Sabbah-Yu}. For the case of certain toric mirror pairs  this matching and much more is established in~\cite{Reichelt-Sevenheck:hypergeometric-Hodge} and~\cite{Mochizuki:twistor-GKZ}.

\begin{remark}
We observe that in the case of $\P^n$, the above essentially amounts to a remarkable theorem of Sperber~\cite{Sperber:hyperKloosterman-congruences} on the slopes of hyper-Kloosterman sums.
Thus we are led to conjecture that for almost all primes, the slopes of the minuscule Kloosterman sums $\Kl_{(G,\varpi_\ki)}$ can be read from the cohomology of $G^\vee/P^\vee$.
This would follow\footnote{Added after submission: Our conjecture can now be established from the main result of Xu--Zhu~\cite[Thm.5.3.2.(1)]{Xu-Zhu} as follows. They construct an overconvergent \(F\)-isocrystal which is a  \(p\)-adic companion of the Kloosterman  \(\ell\)-adic sheaf  \(\Kl_G\) of~\cite{HNY}, and show that its Newton polygon is the half-sum of positive co-roots \(\rho^\vee\) for almost all primes. The slopes of the minuscule Kloosterman sums $\Kl_{(G,\varpi_\ki)}$ under the minuscule representation  \(V_{\varpi_\ki}\) are therefore  \(\langle w\varpi_\ki,\rho^\vee \rangle + \langle \varpi_\ki,\rho^\vee \rangle\) for  \(w\in W^P\) (recall from \S\ref{sub:minuscule} that the weigths of  \(V_{\varpi_\ki}\) are the orbit \(W\cdot \varpi_\ki\)). For every  \(d\ge 0\), the number of  \(w\in W^P\) such that \(\langle w\varpi_\ki,\rho^\vee \rangle + \langle \varpi_\ki,\rho^\vee \rangle=d\) coincides with the Betti number  \(\dim H^{2d}(G^\vee/P^\vee)\) (see the proof of Proposition~\ref{p:Lefschetz-sl2}).} from a suitable $p$-adic comparison isomorphism for differential equations of exponential type between Dwork $p$-adic cohomology and
complex Hodge theory, which does not seem to be available in the literature yet.
Interestingly the same Hodge numbers appear in the $(\mathfrak{g},K)$-cohomology of a certain $L$-packet of discrete series~\cite{Gross:minuscule-principalsl2}.
\end{remark}

\subsection{Enumerative formula} We can deduce from the above mirror theorem an 
integral representation for the quantum period and combinatorial
formulas for certain Gromov--Witten invariants. Property (ii) is referred to as weak 
Landau--Ginzburg model in~\cite{Przyjalkowski:LG-Fano}.
\begin{theorem}\label{t:hyper-series}
(i) The quantum period~\eqref{def:hyp-series} of $G^\vee/P^\vee$ is equal to the integral of 
the potential on the middle-dimensional compact cycle of $\mathring{G/P}$,
\[
I_{\mathrm{cpt}}(\lambda,q):= \oint e^{f_q/\lambda} \frac{\omega}{(2i\pi)^{\dim(G/P)}}.
\] 

\noindent
(ii) For every integer $e\ge 0$, the genus $0$ and degree $e$ Gromov--Witten correlator 
$\langle \tau_{ce-2}\sigma_{\PD(1)}\rangle_{0,e}$ of
$G^\vee/P^\vee$ is equal to the constant term of $f_1^{ce}$ in any cluster chart of 
$\mathring{G/P}$, divided by $(ce)!$.
\end{theorem}

For quadrics, the theorem can be established directly by computing both sides as shown 
in~\cite[\S5.3]{Pech-Rietsch-Williams:LG-quadrics}. 
For Grassmannians, the theorem is due to 
Marsh--Rietsch~\cite[Thm.4.2]{Marsh-Rietsch:B-model-Grassmannians}.

\begin{proof}
Assertion (ii) follows from (i) by taking residues. Recall that $c$ is the Coxeter number of 
$G$.  
Also note that $\omega$ is $T$-invariant by Lemma~\ref{l:Ad-omega}, in particular invariant 
by $\rho^\vee$. This implies the identity
$I_{\mathrm{cpt}}(\lambda,q)=I_{\mathrm{cpt}}(1,q/\lambda^c)$, which is also satisfied by the 
quantum period~\eqref{def:hyp-series}.

To establish (i) we observe as consequence of the mirror Theorem~\ref{t:Dmodule-mirror} 
that $I_{\mathrm{cpt}}(1,q)$ is solution of the quantum connection
$\cQ^{G^\vee/P^\vee}$.
It is a power series in $q$ by Cauchy's residue formula.
The same holds for the fundamental solution $S_{\PD(1)}(1,q)$.
We can then deduce the desired equality of the two solutions up to scalar from the 
Frobenius method at the regular singularity $q=0$. More precisely, we need to consider 
the equivariant connection $\cQ^{G^\vee/P^\vee}_{T^\vee}$ and equivariant 
Gromov--Witten 
correlators.
For generic $h\in \t^*$, the monodromy at $q=0$ is regular semisimple. We then specialize the equivariant parameter to $h=0$.
 
To conclude the proof of (i) we need to specialize the solution $\langle 
S_{\PD(1)}(1,q),1\rangle$ 
as in~\eqref{def:hyp-series}. It is a power series in $q$ with constant term $1$. Moreover, 
the integral $I_{\mathrm{cpt}}(1,q)$ is against the form $\omega$ which implies
 the identity in view of Proposition~\ref{p:one-to-omega}.
\end{proof}

The quantum period typically has infinitely many zeroes. As explained by 
Deligne~\cite[p.128]{Deligne-Malgrange-Ramis} this implies that the irregular Hodge
filtration on $H^d_{\mathrm{dR}}(\mathring{G/P},\Exp^{f_t/\hbar})$
does not come from a Hodge structure.

\begin{remark}
The works of Marsh--Rietsch~\cite{Marsh-Rietsch:B-model-Grassmannians} for 
Grassmannians and
Pech--Williams--Rietsch~\cite{Pech-Rietsch-Williams:LG-quadrics} for quadrics lead us to 
suggest a 
more general formula that $\langle S_{\PD(1)}, \sigma_w\rangle$ should be equal to the residue
integral $\oint p_w e^{f_q/\lambda} \frac{\omega}{(2i\pi)^{\dim(G/P)}}$, with the Pl\"ucker coordinate 
$p_w$ added. This formula generalizes Theorem~\ref{t:hyper-series}(i), corresponding 
to the case $w=1$, and is compatible with the Gamma conjecture and central 
charges discussed in~\cite{Galkin-Iritani:Gamma-conj,Iritani:integral-structure}.
\end{remark}

\begin{remark} 
In a series of works (see e.g.,~\cite{GLO:parabolic-Whittaker}),
Gerasimov--Lebedev--Oblezin study the Givental integral from various viewpoints, 
motivated by archimedean $L$-functions, integrable systems of Toda type, and Whittaker 
functions. 
\end{remark}

\subsection{Projective spaces} 
For $\P^{n}=\Gr(1,n+1)$, the Coxeter number is $c=n+1$.
We deduce from~\eqref{J-Pn} that the quantum period $ \langle S_{\PD(1)}(\lambda,q),1 
\rangle 
= \langle J^{\P^{n}}(\lambda,q),\sigma_{\PD(1)},1 \rangle$ is equal to
\[
\sum_{e=0}^\infty \frac{1}{(e!)^{n+1}} \left( \frac{q}{\lambda^{n+1}} \right)^e 
= {}_{0}F_n\left( \begin{smallmatrix}
 & \text{\----} &  \\ 1 & \cdots &  1  
\end{smallmatrix}
; \frac{q}{\lambda^{n+1}}
 \right).
\]
On the other hand,
\[
I_{\mathrm{cpt}}(\lambda,q)=
\oint e^{\frac{1}{\lambda}
\left( x_1 +\cdots +x_n + \frac{q}{x_1\cdots x_n} \right)
}
\frac{dx_1 \cdots dx_n}{(2i\pi)^n x_1 \cdots x_n}.
\]
Hence Theorem~\ref{t:hyper-series} reduces to Erd\'elyi's integral representation.

\begin{remark}
The quantum period for a general minuscule homogeneous space $G^\vee/P^\vee$ is related 
to the Bessel functions of matrix argument
introduced by C.~Herz; see~\cite{Macdonald:hypergeometric,Shimura:generalized-Bessel}.
\end{remark}

\subsection{Classical Bessel functions}\label{s:Bessel}
For $\P^1=\Gr(1,2)$, we have $f_q(x)=x+\frac{q}{x}$ for $x \in \G_m=\mathring{\P^1}$, $\omega=\frac{dx}{x}$, and
\[
H^i_{\mathrm{dR}}(\G_m,\Exp^{f_q/\hbar}) =
\begin{cases}
0 &
\text{if $i=0$,}\\ 
\C \omega \oplus \C x \omega
&
\text{if $i=1$.}
\end{cases}
\]
Deligne defines an irregular Hodge filtration and shows in~\cite[p.127]{Deligne-Malgrange-Ramis} that 
\(F^1 H^1_{\mathrm{dR}}(\G_m,\Exp^{f_q/\hbar})=\C\omega,\) which
corresponds to Theorem~\ref{p:one-to-omega} above.

The dual space
is generated by the two cycles $\oint$ and $\int_0^\infty$, denoted by $e_1$, $-e_2$ in~\cite{Deligne-Malgrange-Ramis}. Note that the cycle $\int_0^\infty$
depends on $q$ and $\lambda$ and approaches $0$ and $\infty$ in the direction of rapid 
decay of 
the exponential.

The quantum period is $ { }_{0}F_1\left( \begin{smallmatrix}
  \text{\----}   \\ 1   
\end{smallmatrix}
; \frac{q}{\lambda^2}
 \right)= I_0(2\sqrt{q}/\lambda)=\oint e^{f_q/\lambda}\frac{\omega}{2i\pi}$.
The other integrals are expressed as follows: $\int_0^\infty e^{f_q/\lambda} \omega = 
2K_0(2\sqrt{q}/\lambda)$; $ \oint
e^{f_q/\lambda} x\frac{\omega}{2i\pi} = \sqrt{q}I_1(2\sqrt{q}/\lambda)$; $\int_0^\infty 
e^{f_q/\lambda}x\omega = -2\sqrt{q}K_1(2\sqrt{q}/\lambda)$.
Note that $I_0'=I_1$ and $K_0'=-K_1$.

The determinant of periods 
\[
\left| 
\begin{matrix}
\oint e^{f_q/\lambda} \omega  & \oint e^{f_q/\lambda} x\omega\\
 \int_0^\infty e^{f_q/\lambda}\omega&  \int_0^\infty e^{f_q/\lambda} x \omega\\
\end{matrix}
\right| = -2i\pi \lambda,
\]
 established in the last paragraph\footnote{The minus sign compared 
 to~\cite{Deligne-Malgrange-Ramis} is because we chose the cycle $\int_0^\infty$ which is 
 $-e_2$.} of~\cite{Deligne-Malgrange-Ramis} corresponds to the Wronskian formula
\[
I_\nu(y)K_{\nu+1}(y) + I_{\nu+1}(y)K_\nu(y) = 1/y,
\] 
for all $y\in \R_{ >  0}$ and $\nu \in \C$.

More generally, we consider the equivariant version.  Let $h \in \C$ and $h\alpha \in \t^*$, where $\alpha$ denotes the positive simple root.  We consider
the integral solutions to $\WBK^{1/\hbar}|_{\hbar=\lambda}$ specialized at $\lambda h\alpha\in \t^*$,
$$
\oint\frac{x^{2h}}{q^h}
e^{f_q/\lambda}
\frac{\omega}{2i \pi}=
q^h\sum_{k=0}^{\infty}
\frac{\lambda^{-2k-2h} q^k}{k!\Gamma(k+2h+1)}
=
\frac{(q/\lambda^2)^h}{\Gamma(1+2h)}
 {  }_{0}F_1\left( \begin{smallmatrix}
  \text{\----}   \\ 1+2h   
\end{smallmatrix}
; \frac{q}{\lambda^2}
 \right)
= I_{2h}(2\sqrt{q}/\lambda),
$$
where compared to Example \ref{ex:p1}, we have $q = t^2$ and the factor $\frac{x^{2h}}{q^h}$ is equal to $(h\alpha)(\gamma(x))$. Similarly the
integral from $0$ to $\infty$ is equal to $K_{2h}(2\sqrt{q}/\lambda)$.

On the quantum connection side, let $\{1,\sigma\}$ be the Schubert basis of $H^*(\P^1)$.  Then the equivariant quantum Chevalley formula is
$$
\sigma *_q \sigma = q\cdot 1 + (\varpi - s\cdot \varpi)\cdot  \sigma = q\cdot 1 + \alpha^\vee \cdot \sigma.
$$
Here, $s$ denotes the unique simple reflection.  Thus
$$
\sigma *_{q,h} \sigma = q\cdot 1 + 2h\cdot \sigma
$$
since $\ip{\alpha^\vee,\alpha} = 2$.
The equivariant quantum connection $\cQ^{\P^1}_{\hbar,T^\vee}(h\alpha)$ is
$$
\hbar q\frac{d}{dq} + \left(\begin{array}{cc} -\hbar h & q \\ 1 & \hbar h \end{array} \right).
$$
This is equivalent to the second order differential operator
$$
(\hbar q \frac{d}{dq})^2  - (q+\hbar^2 h^2),
$$
which has solutions specialized to $\hbar=\lambda$ the modified Bessel functions 
$I_{2h}(2\sqrt{q}/\lambda)$ and $K_{2h}(2\sqrt{q}/\lambda)$. 
This agrees with Theorem \ref{t:EDmodule-mirror}.

\section{Compactified Fano and log Calabi--Yau mirror pairs}\label{s:mirror}
Our Theorems \ref{t:Dmodule-mirror} and \ref{thm:Shbar} verify two specific mirror 
symmetry predictions.  In this brief section, the goal is to recast the mirror symmetry of 
flag varieties in view of recent advances, and provide some evidence for
potential generalizations.

\subsection{Mirror pairs of Fano type}
The notion of mirror pairs of Fano type is explained in~\cite[\S2.1]{Katzarkov-Kontsevich-Pantev:LG-models}. In the context of Rietsch's conjecture that we study in this paper, we have a family of 
mirror pairs indexed on one side by $H^2(G^\vee/P^\vee)$ and on the other side by $Z(L_P)$. 
Lemma~\ref{l:2ipiH2} is interpreted as the ``mirror map'' and can be compared 
with~\cite[Lem.4.2]{Iritani:integral-structure} in the toric case.

The A-model is a triple $(X,g,1/\omega_X)$ consisting of a projective Fano variety $X$, a 
complexified K\"ahler form $g$ and an anticanonical section $1/\omega_X$. In the context of
Rietsch's conjecture, the variety is $X = G^\vee/P^\vee$, the K\"ahler class
 is varying in $H^2(G^\vee/P^\vee)$ and the anticanonical section is the one constructed in~\cite{KLS:projections-Richardson}; see also \cite{Rietsch:mirror-construction-QH-GmodP}.
The B-model is another triple $((Y,f),\eta,\omega_Y)$ consisting of a Landau--Ginzburg model, 
namely a smooth variety $Y$ with trivial canonical class, a regular function $f:Y \to \C$ 
(Landau--Ginzburg potential), a K\"ahler form $\eta$ and a non-vanishing canonical section 
$\omega_Y$ (holomorphic volume form). 

\subsection{$D$-module version of Rietsch's mirror conjecture}
In the context of Rietsch's conjecture, the B-model is as follows.  The Landau--Ginzburg model is given by the 
Berenstein--Kazhdan geometric crystal. The underlying variety is the open 
Richardson $Y=\oGP$ in $G/P$, and the Landau--Ginzburg potential is the decoration 
function $f_t$ 
of Berenstein--Kazhdan, which depends on the parameter $t\in Z(L_P)$.
The volume form $\omega_Y$ is again the one constructed in~\cite{KLS:projections-Richardson}.

\begin{conjecture}\label{conj:Dhbar}
Let $P$ is a parabolic subgroup of $G$ and let $P^\vee$ be the dual parabolic subgroup of 
$G^\vee$.  There exists an isomorphism
$$
\WBK^{1/\hbar}_{(G,P)} \cong \cQ^{G^\vee/P^\vee}_{\hbar,T^\vee},
$$
of graded $\aD_\hbar$-modules over $Z(L_P) \times \t^*$ relative to $\t^*$.
\end{conjecture}

Theorem~\ref{thm:Shbar} establishes Conjecture~\ref{conj:Dhbar} in the case where $P$ is 
minuscule.  The superpotential of $\WBK_{(G,P)}$ (explicitly described in \S\ref{s:explicit}) 
agrees with the Landau--Ginzburg model defined by Rietsch; 
see~\cite[Lem.5.2]{Rietsch:mirror-construction-QH-GmodP}.  Thus 
Conjecture~\ref{conj:Dhbar} is compatible with Rietsch's 
conjecture~\cite{Rietsch:mirror-construction-QH-GmodP}.

\subsection{Mirror pairs of compactified Landau--Ginzburg models} Following
\cite[\S3.2.4]{Katzarkov-Kontsevich-Pantev:LG-models}, one may also consider quadruples 
$(X,g,\omega,f)$ consisting of a projective Fano variety $X$, a complexified K\"ahler form $g$, 
a canonical section $\omega_X$, and a potential function
$f_X$.
We may then examine, for appropriate choices of K\"ahler forms, the mirror symmetry 
between \[(G/P,g,\omega_{G/P},f_{G/P}) \text{ and }
 (G^\vee/P^\vee,g^\vee ,\omega_{G^\vee/P^\vee},f_{G^\vee/P^\vee}). 
\]
The A and B-sides now play a symmetric role.
Rietsch's mirror conjecture corresponds to omitting some of the data on both sides.
The full mirror conjecture between these compactified mirror pairs involves the matching of a variety of homological data on both sides.

For example, a Fano type mirror pair gives rise to a pair of open Calabi--Yau manifolds by taking the complement of the anticanonical divisor. One obtains triples
of a log Calabi--Yau manifold, a K\"ahler form and a volume form. In our 
setting, the log Calabi--Yau varieties are
$\mathring{G^\vee/P^\vee}$ and $\oGP$ respectively.   The volume forms are as before. 
Thus from the general mirror 
predictions~\cite[Table~2]{Katzarkov-Kontsevich-Pantev:LG-models}, we expect a matching 
of cohomology of the open projected Richardson variety
$H^*(\mathring{G/P})$ and the cohomology of the Langlands dual open projected Richardson 
variety $H^*(\mathring{G^\vee/P^\vee})$.  We show in the
next subsection that this matching holds more generally for arbitrary Richardson 
varieties.

\subsection{Open Richardson varieties}
Recall the open Richardson varieties $\cR^w_u \subset G/B$, where $u,w\in W$ with $u\le w$, and the special case $\cR^{w_0}_{w_0^P}\cong \oGP$. They are log
Calabi--Yau varieties with canonical volume form~\cite{KLS:projections-Richardson}. We 
denote by $\check \cR^w_u\subset G^\vee/B^\vee$ the open Richardson varieties
inside the flag variety of the dual group.  

\begin{proposition}\label{prop:Richardson}
For any $i \geq 0$ and $u,w \in W$ with $u\le w$, there is an isomorphism
$H^i(\cR^w_u) \cong H^i(\check\cR^{w}_{u})$. 
\end{proposition}
\begin{proof}
Since $\cR^w_u$ and $\cR^w_u$ are smooth complex algebraic varieties of the same 
dimension, by Poincar\'e duality, the stated isomorphism is equivalent to the same 
statement for cohomology with compact support.  As in the proof of Proposition 
\ref{p:Hmid}, this is in turn equivalent to the isomorphism $\mathrm{Ext}^{\bullet}(M_w,M_u) \cong 
\mathrm{Ext}^{\bullet}(M_{\g^\vee,w},M_{\g^\vee,u})$ where $M_w$ (resp. $M_{\g^\vee,w}$) denotes a 
Verma module in the principal block of category $\cO$ for $\g$ (resp. $\g^\vee$).  By the work 
of Soergel \cite{Soergel:CategoryO}, the principal blocks of category $\cO$ for $\g$ and 
$\g^\vee$ are equivalent, and the isomorphism of $\mathrm{Ext}$-groups follows.
\end{proof}

\begin{question}
Can this isomorphism be an indication of mirror symmetry between open Richardson varieties $\cR^w_u \subset G/B$ and $\check\cR^{w}_{u} \subset G^\vee/B^\vee$?
\end{question}

\section{Proofs from Section \ref{ssec:gamma}}\label{sec:proofs}
\subsection{Proof of Proposition \ref{prop:gamma}}
(1) $\implies$ (2) and (1) $\implies$ (3) are easy to check directly.

We thank the referee for the following argument showing (2) $\implies$ (1), simplifying our 
original proof.  Recall that $Q^\vee$ (resp. $Q^\vee_P$) denotes the coroot lattice spanned 
by $\alpha_i^\vee, i \in I$ (resp. $\alpha_i^\vee, i \in I_P$).  By a result of Peterson and Woodward 
(see \cite[Thm.10.13(1)]{Lam-Shimozono:GmodP-affine}), for each $\lambda_P^\vee \in 
Q^\vee/Q^\vee_P$ there exists a unique lift $\lambda^\vee \in Q^\vee$ (called the 
``Peterson--Woodward lift") of $\lambda_P^\vee$ such that $\langle \alpha,\lambda^\vee \rangle \in 
\{0,-1\}$ 
for all $\alpha \in R^+_P$.  Since $\i$ is minuscule and $\beta \in R^+\setminus R^+_P$, we have 
$\beta^\vee + Q^\vee_P = \alpha^\vee_\i +Q^\vee_P \in Q^\vee/Q^\vee_P$.  Thus, if $\beta^\vee$ satisfies 
(2), then it is the Peterson--Woodward lift, which is unique, and we must have $\beta^\vee = 
\qroot^\vee$.

\medskip

We show that (3) $\implies$ (1).
Suppose $G$ is simply-laced.  Suppose $\beta = -w^{-1}(\theta) \in R^+\setminus R^+_P$, but $\beta \neq \alpha_\i$.  Then $\i$ is also cominuscule so $\beta = \alpha_\i + \beta'$ where $\beta'$ is a nonzero linear combination of $\alpha_j$ for $j \neq \i$.  Since $w \in W^P$, we have $w \alpha_j \in R^+$ for $j \neq \i$.  Thus $w \beta - w\alpha_\i \in \Z_{\geq 0} R^+ \setminus \{0\}$.  Since $w\alpha_\i$ is a root, it would be impossible for $w\beta = -\theta$. 

Suppose $G$ is of type $B_n$.  Choosing coordinates for $R$, we have $\theta = \alpha_1 + 2\alpha_2 + \cdots + 2 \alpha_n = \epsilon_1+\epsilon_2$.  We may identify $W$ with the group of signed permutations on $\{1,2,\ldots,n\}$, and $W^P$ is identified with signed permutations that are increasing, under the order $1 < 2 < \cdots < n < -n < -(n-1) < \cdots < -1$.  We have $|W^P| = 2^n$.  For example, $w = (2,4,5,-3,-1) \in W^P$ and $w^{-1}(\epsilon_1+\epsilon_2) = -\epsilon_5 + \epsilon_1$.  It follows by inspection that $-w^{-1}(\theta)  \in R^+\setminus R^+_P$ implies that $-w^{-1}(\theta) = \epsilon_4 + \epsilon_5 = \alpha_{n-1}+ 2\alpha_n$.

Suppose $G$ is of type $C_n$.  We have $\theta = 2\alpha_1 + 2\alpha_2 + \cdots + 2\alpha_{n-1} + \alpha_n$.  The elements of $W^P$ are $$1, s_1, s_{2} s_1, s_3 s_2 s_1, \ldots, s_n s_{n-1} \cdots s_1, s_{n-1} s_{n} s_{n-1} \cdots s_2 s_1, \ldots, s_1 s_2 \cdots s_{n-1} s_n s_{n-1} \cdots s_2 s_1,$$ and we have $|W^P| = 2n$.  We have $-w^{-1}(\theta)  \in R^+\setminus R^+_P$ if and only if $w = s_1 s_2 \cdots s_{n-1} s_n s_{n-1} \cdots s_2 s_1$, and the statement follows.

\subsection{Proof of Proposition \ref{prop:Wgamma}}
We first note the following properties of $w_{P/Q}$.
\begin{lemma}\label{lem:z}\
\begin{enumerate}
\item
$\Inv(w_{P/Q}) = R^+_P \setminus R^+_Q$,
\item
$\ell(w_{P/Q}) = \ip{-2\rho_P,\qroot^\vee}$,
\item
$\ell(w_{P/Q} s_\qroot) = \ell(w_{P/Q}) + \ell(s_\qroot) = \ip{2(\rho-\rho_P),\qroot^\vee}-1$.
\end{enumerate}
\end{lemma}
\begin{proof}
Let $w_P$ (resp. $w_Q$) be the maximal element of $W_P$ (resp. $W_Q$).  Then $w_{P/Q} w_Q = w_P$ is length-additive, so $\Inv(w_{P/Q}) = w_Q(\Inv(w_P) \setminus \Inv(w_Q)) = R^+_P \setminus R^+_Q$, proving (1).  Formula (2) follows from Lemma \ref{lem:gamma}(2).  Since $\Inv(s_\qroot) \cap R^+_P = \emptyset$, it follows that the product $w_{P/Q} s_\qroot$ is length-additive.  (3) follows from (2) and $\qroot \in \tR$.
\end{proof}

\subsubsection{Proof of (1) in Proposition \ref{prop:Wgamma}}
 It is equivalent to show that $\Inv(w) \supset \Inv(s_\qroot)$.  Suppose $\alpha \in \Inv(s_\qroot)$.  Then $\alpha - \ip{\alpha,\qroot^\vee} \qroot = s_\qroot \alpha < 0$, where $a = \ip{\alpha,\qroot^\vee} > 0$.  Thus
$$
-a\theta = a w(\qroot) = w(\alpha)-w(s_\qroot \alpha),
$$
and it follows that $w \alpha < 0$ because $w(s_\qroot \alpha)$ is a root.  

\subsubsection{Proof of (2) in Proposition \ref{prop:Wgamma}}
After Lemma \ref{lem:z}(1,3), it is equivalent to show that $\Inv(ws_\qroot) \supset R^+_P \setminus R^+_Q$.  Let $\alpha \in R^+_P \setminus R^+_Q$.  Then $s_\qroot \alpha = \qroot+\alpha$ by Lemma \ref{lem:gamma}.  Thus $w s_\qroot \alpha = w\alpha + w\beta = w\alpha -\theta$.  Since $\theta$ is the highest root, we deduce that $\alpha \in \Inv(ws_\qroot)$.

\subsubsection{Proof of (3) in Proposition \ref{prop:Wgamma}}
Since $w_{P/Q}^{-1} \in W_P$, it suffices to show that $ws'_\qroot \in W^P$.  It suffices to check that $\Inv(ws_\qroot) \cap R^+_Q = \emptyset$.  But $s_\qroot$ fixes every element in $R^+_Q$, and $\Inv(w) \cap R^+_Q = \emptyset$ since $w \in W^P$.  The claim follows.

\subsubsection{Proof of (4) in Proposition \ref{prop:Wgamma}}
The standard parabolic subgroup $J$ is given as follows: for type $A_n$, we have $J = \{2,3,\ldots,n-1\}$; for type $D_n$ or $E_6$, we have $J = [n] \setminus \{2\}$; for type $E_7$, we have $J = [n] \setminus \{1\}$; for type $B_n$, we have $J = [n] \setminus \{1,2\}$; for type $C_n$, we have $J = \emptyset$.  In all cases, it is clear that $W_J$ stabilizes $\theta$.  

If $w,v \in W(\qroot)$ then clearly $wv^{-1}$ belongs to the stabilizer of $\theta$.  In the simply-laced types, this stabilizer is exactly the group $W_J$.  In type $B_n$, the stabilizer of $\theta$ is $W_{[n] \setminus \{2\}}$, but from the description in the proof of Propositio \ref{prop:gamma}, it is clear that $wv^{-1} \in W_J$.  In type $C_n$, as noted previously we have $W(\qroot) = \{s_\qroot\}$ consists of a single element.

The double coset $W_J w W_P$ contains a unique minimal element $w'$, and since $w \in W^P$, we have a length-additive factorization $w = uw'$, where $u \in W_J$.  Since $w' \in W^P$ and $(w')^{-1}(\theta) = w^{-1}(\theta) = -\qroot$, we have $w' \in W(\qroot)$.

\subsubsection{Proof of final sentence in Proposition \ref{prop:Wgamma}}
We assume that $w \in W^P$ satisfies $\Inv(w) \supset \Inv(s_\qroot)$ and $\Inv(ws_\qroot) \supset R^+_P \setminus R^+_Q$.  Suppose first that $G$ is simply-laced, so that $\qroot = \alpha_\ki$.  Suppose that $-w^{-1}(\theta) = \alpha \neq \alpha_\ki$.  Let $w\alpha_\ki = -\eta < 0$.  Since $w \in W^P$, we have $\alpha \notin R^+_P$.  On the other hand, we have $w(\alpha - \alpha_\ki) = -\theta + \eta < 0$.  Again because $w \in W^P$, this shows that $\alpha \notin R^+ \setminus R^+_P$.  Thus $\alpha \in R^-$.  

Let $\delta = -\alpha \in R^+$.  Since $w\delta = \theta$ and $w \in W^P$, we have that $\delta + \lambda$ cannot be a root whenever $0 \neq \lambda \in \sum_{j \in I_P} \Z_{\geq 0} \alpha_j$.  If $\delta \in R^+_P$, it follows that $\delta \in R^+_P \setminus R^+_Q$.  But then $s_\qroot \delta = \delta + \qroot$  implies that $(ws_\qroot)\delta = w \delta + w\qroot > 0$ contradicts the assumption that $\Inv(ws_\qroot) \supset R^+_P \setminus R^+_Q$.

Thus $\delta \in R^+\setminus R^+_P$, and again since $w \in W^P$, we may assume that $\delta = \theta$.  Thus $w \theta = \theta$, so $w$ lies in the stabilizer $W' \subset W$ of $\theta$.  In types $E_6$, $E_7$, or $D_n$, $n \geq 4$, it is easy to see that $\Inv(ws_\ki)$ for $w \in W'$ cannot contain $R^+_P \setminus R^+_Q$ since $W'$ is a parabolic subgroup that contains the minuscule node $\i$, but does not contain the adjoint node (node $2$ in types $D_n$ or $E_6$ and node $1$ in type $E_7$).  In type $A_n$, the whole claim is easy to check directly, and we conclude that $w \in W(\qroot)$.

Suppose that $G$ is of type $B_n$.  We use notation from the proof of Proposition \ref{prop:gamma}.  We have $s_\qroot = s_{n-1}s_ns_{n-1}$ and for a signed permutation $w = w_1 w_2 \cdots w_n \in W^P$, we have $ws_\qroot = w_1 w_2 \cdots (-w_{n}) (-w_{n-1})$.  Thus the first condition $\Inv(w) \supset \Inv(s_\qroot)$ is equivalent to $w_{n-1}, w_n < 0$.  The second condition $\Inv(ws_\qroot) \supset  R^+_P \setminus R^+_Q$ is equivalent to the condition that $\{w_1,w_2,\ldots,w_{n-2}\}$ are all bigger than $-w_n$ and $-w_{n-1}$ under the order $1 < 2 < \cdots n < -n < -(n-1) <\cdots < -1$.  It follows that $w_{n-1} = -2$ and $w_{n} = -1$, so that $w \qroot = -\theta$.

Suppose that $G$ is of type $C_n$.  Then $\ell(s_\qroot) \geq \ell(w)$ for $w \in W^P$ with equality if and only if $w = s_\qroot = s_1 s_2 \cdots s_{n-1} s_n s_{n-1} \cdots s_2 s_1$.  The claim follows easily.

\section{Background on $D_{\hbar}$-modules}
\label{s:filtered}

 The main purpose of this section is to establish Proposition~\ref{prop:hfree}, which is used in \S\ref{s:proof-Peterson}.

\subsection{Filtered and graded categories}
\label{ssec:Dh}
Let $X$ be a complex smooth affine algebraic variety equipped with a 
$\G_m$-action. Its structure sheaf $\cO_X$ is naturally graded by $\G_m$-homogeneous
sections. Denote by $\ppi: T^*X \to X$ the cotangent bundle of $X$.  Denote by $D_X$ the sheaf of differential operators on $X$.  This is a sheaf of
noncommutative rings.  It is equipped with a \emph{filtration}
$$
\cdots \subset D_{X,-1} \subset D_{X,0}  \subset D_{X,1} \subset \cdots
$$
induced by the gradation of $\cO_{T^*X}$ plus the order of the differential operator.

Let $\MF(D_{X,\bullet})$ denote the category of \emph{filtered left $D_{X,\bullet}$-modules} that 
are quasi-coherent
as $\cO_X$-modules.  
An object $M_\bullet \in \MF(D_{X,\bullet})$ is equipped with a filtration $ \cdots M_{-1}\subset M_0 \subset M_1 \cdots$ satisfying 
$D_{X,j}M_i \subset M_{i+j}$.  The category $\MF(D_{X,\bullet})$ is an additive category but not an abelian category; it can be made into an exact category
by declaring a sequence $0 \to M'_\bullet \to M_\bullet \to M''_\bullet \to 0$ to be exact if $0 \to M'_i \to M_i \to M''_i \to 0$ is exact for all $i$.  (This
is stronger than asking for the sequence of underlying unfiltered $D_X$-modules to be 
exact.)  As shown in \cite{Laumon:filtres}, one can define the derived
category of $\MF(D_{X,\bullet})$; we let $D^bF(D_{X,\bullet})$ denote the bounded derived 
category of $\MF(D_{X,\bullet})$.  There are natural forgetful
functors $\MF(D_{X,\bullet}) \to \rM(D_X)$ and $D^bF(D_{X,\bullet}) \to D^b(D_X)$ sending a filtered module $M_\bullet$ to the underlying $D_X$-module $M$, and a complex $M^\cdot_\bullet$ of filtered modules to the underlying complex $M^\cdot$.

The associated graded of $D_{X,\bullet}$ is the sheaf $\gr D_{X,\bullet} = \ppi_* \cO_{T^*X}$ of graded commutative rings on $X$, where the grading comes from the grading of $\cO_X$ together with the declaration that vector fields have degree one.  Since $\ppi$ is affine, we have equivalences of categories
$$
\rM(\cO_{T^*X}) \cong \rM(\ppi_* \cO_{T^*X}), \qquad \text{and} \qquad D^b(\cO_{T^*X}) \cong D^b(\ppi_* \cO_{T^*X}),
$$
between the corresponding categories of quasi-coherent $\cO_{T^*X}$-modules and 
quasi-coherent $\ppi_* \cO_{T^*X}$-modules, and bounded derived categories.  We have an 
associated graded functor, and derived functor
$$
\gr: \MF(D_{X,\bullet}) \to \rM(\cO_{T^*X}), \qquad \text{and} \qquad  \gr: D^bF(D_{X,\bullet}) \to D^b(\cO_{T^*X}).
$$

\begin{definition}
Let $D_{\hbar,X}$ denote the sheaf of graded noncommutative rings with a central section 
$\hbar$, locally generated by $f \in \cO_X$ and sections $\xi \in
\Theta_X$ of the tangent sheaf with the relations $[f,\xi] = \hbar (\xi \cdot f)$ and $\xi \eta-\eta \xi = \hbar [\xi,\eta]$.  The grading is given by the
assignment $\deg(\hbar) = 1$ and the homogeneous degrees of $f$ and $\xi$ induced by the 
$\G_m$-action.  
\end{definition}
The sheaf $D_{\hbar,X}/\hbar$ is isomorphic to the sheaf $\ppi_* \cO_{T^*X}$, while the 
localization $D_{\hbar,X}$ at $(\hbar)$ is isomorphic to $D_X[\hbar^{\pm 1}]$.  Let $\MG(D_{\hbar,X})$ 
denote the category of sheaves of graded left $D_{\hbar,X}$-modules that are 
quasi-coherent as graded $\cO_X$-modules. To an object $M_\bullet \in \MF(D_{X,\bullet})$ we associate an object 
\[
M_\bullet \otimes \C[\hbar] =: M^\hbar = \bigoplus_i M^\hbar_i \in \MG(D_{\hbar,X}),
\]
by defining $M^\hbar_i = M_i$.  The section $\hbar$ acts by the identity, thought of as a map 
from $M^\hbar_i$ to $M^\hbar_{i+1}$.  It is clear that $\otimes \C[\hbar]: \MF(D_{X,\bullet}) \to 
\MG(D_{\hbar,X})$ is an exact functor.

For the following result see \cite[\S7]{Laumon:filtres} \cite[\S4]{Schapira-Schneiders}.
\begin{proposition}\label{prop:h0}
The functor
$$
{\otimes} \C[\hbar]:  \MF(D_{X,\bullet}) \to \MG(D_{\hbar,X}), \qquad M_\bullet \mapsto M^\hbar = M_\bullet \otimes 
\C[\hbar],
$$
is an equivalence between $\MF(D_{X,\bullet})$ and the full subcategory of $\hbar$-torsion-free 
$D_{\hbar,X}$-modules.  It induces a derived functor ${\otimes} \C[\hbar]$ giving an equivalence of 
categories
$${\otimes} \C[\hbar]: D^bF(D_{X,\bullet}) \cong D^b(D_{\hbar,X}).$$ 
\end{proposition}

We also have a functor ${\otimes}_{\C[\hbar]} \C: \MG(D_{\hbar,X}) \to \rM(\cO_{T^*X})$ and a left derived functor $\stackrel{L}{\otimes}_{\C[\hbar]} \C: D^b(D_{\hbar,X}) \to D^b(\cO_{T^*X})$, setting $\hbar  = 0$.
\begin{proposition}\label{prop:functors} We have commutative diagrams
\begin{align*}
\begin{diagram}
\node{\MF(D_{X,\bullet})} \arrow {e,t}{ \otimes \C[\hbar]} \arrow{se,r}{ \gr} \node{\MG(D_{\hbar,X})} \arrow{s,r}{{\otimes}_{\C[\hbar]} \C}
\\ \node{} \node{\rM(\cO_{T^*X})} \end{diagram}
\quad\quad
\begin{diagram}
\node{D^bF(D_{X,\bullet})} \arrow {e,t}{\otimes \C[\hbar]} \arrow{se,r}{\gr} \node{D^b(D_{\hbar,X})} \arrow{s,r}{\stackrel{L}{\otimes}_{\C[\hbar]} \C}
\\ \node{} \node{D^b(\cO_{T^*X})} \end{diagram}
\end{align*}
\end{proposition}

\begin{example}\label{ex:DX-filtered}
Consider $X=\G_m^n\times \G_m$, with coordinates $(x_1,\ldots,x_n,q)$, and equipped with the $\G_m$-action 
\[
\zeta \cdot (x_1,\ldots,x_n,q) = (\zeta x_1,\ldots,\zeta x_n,\zeta^{n+1}q).
\]
The ring $\C[X]$ of Laurent polynomials has a corresponding gradation by homogeneous polynomials. The potential $f=x_1+\cdots +x_n +\frac{q}{x_1\cdots
x_n}$ has degree one.
The ring of differential operators $D_X$ is filtered by the subspaces $D_{X,i}$, which for each $i\in \Z$, are the linear span of the operators 
\[
x_1^{a_1}  \frac{\partial^{b_1}}{\partial x^{b_1}_1} 
x_2^{a_2}  \frac{\partial^{b_2}}{\partial x^{b_2}_2}
\cdots x_n^{a_n}\frac{\partial^{b_n}}{\partial x^{b_n}_n} 
q^{m}
\frac{\partial^\ell}{\partial q^\ell},\quad
a_1 + a_2+ \cdots + a_n + (n+1)m - n\ell \le i.
\]
The $D_X$-module $D_X/D_X(d - df \wedge)$ that we denote $\Exp^f$ is equipped with a natural filtration, and becomes an object of $MF(D_{X,\bullet})$.
The ring $D_{\hbar,X}$ is the graded noncommutative ring generated by functions and differential operators $\xi_{x_k}$, $\xi_q$ with notably the relations
\[
[ \xi_{x_k} ,x_k] = \hbar, \quad [\xi_q, q]=\hbar.
\]
The degrees are given by $\deg(x_k)=1$, $\deg(\xi_{x_k})=0$, $\deg(q)=n+1$, $\deg(\xi_q)=-n$, 
$\deg(\hbar)=1$. 
One can think $\xi_{x_k}$ as representing ``$\hbar \frac{\partial}{\partial x_k}$", and $\xi_q$ as representing ``$\hbar \frac{\partial}{\partial q}$''.
Applying the functor of Proposition~\ref{prop:h0}, we have that $\Exp^f \otimes \C[\hbar]$ becomes the $D_{\hbar,X}$-module $\Exp^{f/\hbar}$ which we can
 describe as follows.
We have that $\Exp^{f/\hbar}$ is isomorphic to the quotient of $D_{\hbar,X}$ by
the left $D_{\hbar,X}$-ideal generated by the operators $\xi_{x_k} - \frac{\partial f}{\partial x_k} $ 
and $\xi_q - \frac{\partial f}{\partial q}$.
The operators are all homogeneous, hence $\Exp^{f/\hbar}$ is an element of $MG(D_{\hbar,X})$, 
and moreover it is $\hbar$-torsion-free, consistently with Proposition~\ref{prop:h0}.
\end{example}

\subsection{Pushforward functors}
In this and the next subsection only we write $\int_\pi$ to denote the pushforward functor for $\aD$-modules, and reserve $\pi_*$ for the pushforward functor of quasi-coherent sheaves.
Let $\pi:X \to Y$ be a $\G_m$-equivariant morphism between complex irreducible smooth varieties $X$ and $Y$ equipped with $\G_m$-actions.  We recall results concerning the pushforward
functors of $D_X$, $D_{\hbar,X}$, and $\cO_{T^*X}$-modules under $\pi$.  Though we shall not need it, the functors of Proposition \ref{prop:functors} are also
compatible with pullbacks under $\pi$.  

Let $\omega_X$ (resp. $\omega_Y$) denote the canonical line bundles of $X$ (resp. $Y$).  The sheaf $\omega_X$ acquires a grading from the $\G_m$-action so that it becomes a filtered right $D_{X,\bullet}$-module.
Define
$$
D_{Y \leftarrow X} := \pi^{-1}(D_Y \otimes_{\cO_Y} \omega_Y^{-1}) \otimes_{\pi^{-1}\cO_Y} \omega_X,
$$
which is a $(\pi^{-1} D_Y, D_X)$-bimodule on $X$.  The module $D_{Y \leftarrow X}$ inherits a filtration from the filtrations of $D_Y$, $\omega_Y$, and $\omega_X$.
We obtain a filtered $(\pi^{-1} D_{Y,\bullet}, D_{X,\bullet})$-bimodule $D_{Y\leftarrow X, \bullet}$ on $X$, satisfying $\pi^{-1}D_{Y,j} \cdot D_{Y \leftarrow X,i} \cdot D_{X,k} \subset D_{Y \leftarrow X,i+j+k}$.  We define the direct image functor by 
$$\int_\pi M^\cdot := R\pi_*(D_{Y \leftarrow X,\bullet} \stackrel{L}{\otimes}_{D_{X,\bullet}} M^\cdot),$$
where $M^\cdot \in D^bF(D_{X,\bullet})$.  Similarly define $\int_\pi: D^b(D_X) \to D^b(D_Y)$ by forgetting filtrations.

\begin{proposition}[{\cite[(5.6.1.1)]{Laumon:filtres}}]
\label{LaumonPushFiltre}
The following diagram commutes:
$$
\begin{diagram}
\node{D^bF(D_{X,\bullet})} \arrow {e,t}{\int_\pi} \arrow{s,l}{} \node{D^bF(D_{Y,\bullet})} \arrow{s}
\\ \node{D^b(D_X)} \arrow{e,t}{\int_\pi} \node{D^b(D_Y),} 
\end{diagram}
$$
where the vertical arrows are the natural forgetful functors.
\end{proposition}

Let $T^*Y \times_Y X$ be the pullback of the cotangent bundle $T^*Y$ to $X$, fitting into the commutative diagram \cite[(5.0.1)]{Laumon:filtres}
$$
\begin{diagram}
\node{T^*X} \node{T^*Y \times_Y X} \arrow{w,t}{\Pi} \arrow{s,l}{\bar \pi} \arrow{e,t}{\ppi_\pi} 
\node{X} \arrow{s,r}{\pi} 
\\ \node{} \node{T^*Y} \arrow{e,b}{\ppi_Y} \node{Y} 
\end{diagram}
$$
We have 
$$
\gr D_{Y \leftarrow X,\bullet} = \pi^* \cO_{T^*Y} \otimes_{\cO_X} \omega_{X/Y}
$$
which has a natural structure of a graded $(\pi^* \cO_{T^*Y},\cO_{T^*X})$-bimodule.  We now define a functor $\int_\pi: D^b(\cO_{T^*X}) \to D^b(\cO_{T^*Y})$ by
\begin{equation}\label{eq:0pushforward}
\int_\pi M^\cdot_0 := (R\bar \pi_* \circ \Pi^![d])(M^\cdot_0),
\end{equation}
where $d= \dim X - \dim Y$ and $\Pi^!: D^b(\cO_{T^*X}) \to D^b(\cO_{T^*Y \times_Y X})$ denotes the upper-shriek functor on derived categories of quasi-coherent sheaves.

We will only use \eqref{eq:0pushforward} when the map $\pi: X \to Y$ is smooth, in which case, we have
\begin{equation}\label{eq:0pushforward2}
\Pi^![d](-) = L\Pi^*(-) \otimes_{\cO_{T^*Y \times_Y X}} \ppi_\pi^*\omega_{X/Y}.
\end{equation}
where $L\Pi^*: D^b(\cO_{T^*X}) \to D^b(\cO_{T^*Y \times_Y X})$ is the left derived functor of the usual pullback functor $\Pi^*$ of quasi-coherent sheaves.

We have the following compatibility result of pushforwards. 
\begin{proposition}[{\cite[(5.6.1.2)]{Laumon:filtres}}] \label{prop:pushLau}
The following diagram commutes:
$$
\begin{diagram}
\node{D^bF(D_{X,\bullet})} \arrow {e,t}{\int_\pi} \arrow{s,l}{\gr} \node{D^bF(D_{Y,\bullet})} \arrow{s,r}{\gr}
\\ \node{D^b(\cO_{T^*X})} \arrow{e,t}{\int_\pi} \node{D^b(\cO_{T^*Y}).}
\end{diagram}
$$
\end{proposition}

Finally, we describe the pushforward functor for $D_{X,\hbar}$-modules.  We define $D_{Y\leftarrow X, \hbar} := D_{Y \leftarrow X,\bullet} \otimes \C[\hbar]$,
which is a graded $(\pi^{-1} D_{Y,\hbar}, D_{X,\hbar})$-bimodule.  We define the direct image functor $\int_\pi: D^b(D_{X,\hbar}) \to D^b(D_{Y,\hbar})$ by 
$$\int_\pi M^\cdot := R\pi_*(D_{Y \leftarrow X,\hbar} \stackrel{L}{\otimes}_{D_{X,\hbar}} M^\cdot).$$
\begin{proposition}\label{prop:pushfilter}
The following diagram commutes:
$$
\begin{diagram}
\node{D^bF(D_{X,\bullet})} \arrow {e,t}{\int_\pi} \arrow{s,l}{{\otimes}\C[\hbar] } \node{D^bF(D_{Y,\bullet})} \arrow{s,r}{{\otimes}\C[\hbar]}
\\ \node{D^b(D_{X,\hbar})} \arrow{e,t}{\int_\pi} \node{D^b(D_{Y,\hbar}).}
\end{diagram}
$$
\end{proposition}
\begin{proof}
A direct comparison shows that $$(D_{Y \leftarrow X,\bullet} \stackrel{L}{\otimes}_{D_{X,\bullet}} M^\cdot) \otimes \C[\hbar] = D_{Y
\leftarrow X,\hbar} \stackrel{L}{\otimes}_{D_{X,\hbar}} (M^\cdot \otimes \C[\hbar]),$$ as graded $\pi^{-1}(D_{Y,\hbar})$-modules.  
Similarly, $\otimes \C[\hbar]$ is an exact functor, so it commutes with $R\pi_*$.
\end{proof}

\begin{proposition}\label{prop:h0commute}
The following diagram commutes:
$$
\begin{diagram}
\node{D^b(D_{X,\hbar})} \arrow {e,t}{\int_\pi} \arrow{s,l}{ \stackrel{L}{\otimes}_{\C[\hbar]} \C} \node{D^b(D_{Y,\hbar})} \arrow{s,r}{ \stackrel{L}{\otimes}_{\C[\hbar]}\C}
\\ \node{D^b(\cO_{T^*X})} \arrow{e,t}{\int_\pi} \node{D^b(\cO_{T^*Y}),}
\end{diagram}
$$
where the vertical arrows are the natural forgetful functors.
\end{proposition}
\begin{proof}
Combine Proposition \ref{prop:functors} with Propositions \ref{prop:pushLau} and \ref{prop:pushfilter}.
\end{proof}

\begin{example}\label{Dhbar-QDE}
Consider $Y=\G_m$, graded by $\deg(q)=n+1$. The ring $D_{\hbar,Y}=\C[q^{\pm 1},\hbar]\langle \xi_q \rangle$ satisfies the relation $[\xi_q,q]=\hbar$, and the gradation
is given by $\deg(\xi_q)=-n$, $\deg(\hbar)=1$.
The quantum differential operator $(q\xi_q)^{n+1}-q$ is homogeneous of degree $n+1$.
It shall follow from the next subsection that it is isomorphic to the pushforward $\int_pi 
\Exp^{f/\hbar}$ where $\Exp^{f/\hbar}$ is as in Example~\ref{ex:DX-filtered} and $\pi:X\to Y$ is the 
projection onto the second factor.
\end{example}

\subsection{Application to the character $D_\hbar$-module}
Let $\pi: X \to Z(L_P)$ denote the geometric crystal and $f: X \to \A^1$ denote the superpotential.  Recall that we defined $\G_m$-actions on $X$ and $Z(L_P)$ in \S \ref{s:homogeneous}.

\begin{proposition}\label{prop:hfree}
The character $D_{\hbar,Z(L_P)} \otimes \Sym(\t)$-module $\WBK^{1/\hbar}_{(G,P)} \in 
D^b(D_{\hbar,Z(L_P)} 
\otimes \Sym(\t))$ is $\hbar$-torsion-free and concentrated in a 
single degree.
\end{proposition}
\begin{proof}
To simplify the notation we will prove the proposition for $\BK^{1/\hbar}_{(G,P)}\in 
D^b(D_{\hbar,Z(L_P)})$ without the 
weight. Thus let 
$$
M^\hbar = D_{\hbar,X}/(\xi - (\xi \cdot f))
$$
denote the cyclic $D_{\hbar,X}$-module generated by a single section $e^{f/\hbar}$.  Here $\xi \in \Gamma(X,\Theta_X)$ denotes a vector field on $X$. 
We shall show that $N^\hbar := \int_\pi M^\hbar \in D^b(D_{\hbar,Z(L_P)})$ is isomorphic to an 
$\hbar$-torsion-free $D_{Z(L_P),\hbar}$-module concentrated in a single degree.  The condition 
that $N^\hbar$ is $\hbar$-torsion-free and concentrated in one cohomological degree is 
equivalent to the condition that the object $N_0 = N^\hbar \stackrel{L}{\otimes}_{\C[\hbar]} \C \in 
D^b(\cO_{T^*Z(L_P)})$ (see \S \ref{ssec:Dh}) is concentrated in a single cohomological degree.

Let $M_0 = M^\hbar \otimes_{\C[\hbar]} \C \in \rM(\cO_{T^*X})$.  Then $M_0$ isomorphic to $\cO_V$, where $V \subset T^*X$ is cut out by the equations $\xi - (\xi \cdot f)$.
By Proposition \ref{prop:h0commute}, we have $N_0 = \int_\pi M_0$.  Denote $T^*Z(L_P) \times_{Z(L_P)} X$ by $Y$.  By \eqref{eq:0pushforward} and \eqref{eq:0pushforward2}, we have
$$
\int_\pi M_0 = R\pi'_*(LF^*(M_0)\otimes_{\cO_{Y} } \tilde \pi^*\omega_{X/Z(L_P)} ),
$$
where $\tilde \pi: Y \to X$ and $\pi':Y \to T^*Z(L_P)$ are the two projections and $F: Y \to T^*X$ is the natural inclusion.  We first show that $LF^*(M_0) \in D^b(\cO_W)$ is concentrated in a single cohomological degree.  
This is equivalent to the condition that $\Tor^i_{\cO_{T^*X}}(\cO_Y,\cO_V) = 0$ for $i > 0$.  It 
is easy to see that both $V$ and $Y$ are smooth subvarieties of $T^*X$, and hence 
Cohen--Macaulay.  

The fiber of $Y \cap V$ under $Y \to T^*Z(L_P) \to Z(L_P)$ over a point $q \in Z(L_P)$ can be identified with the critical point set of $f|_{\pi^{-1}(q)}$.  Rietsch
\cite{Rietsch:mirror-construction-QH-GmodP} showed that this critical point set is 0-dimensional, and it follows that $Y \cap V$ is pure of dimension $1$.  Since $\dim V = \dim X$ and $\dim Y = \dim X + 1$, it follows that the intersection $Y \cap V$ is proper.  

If $\Tor^i_{\cO_{T^*X}}(\cO_Y,\cO_V)$ is nonzero then it is nonzero after localizing to some irreducible component of $C$ of $Y \cap V$.  Applying \cite[V.6,
Cor.]{Serre:local-algebra} we obtain  $\Tor^i_{\cO_{T^*X,C}}(\cO_{Y,C},\cO_{V,C}) = 0$ for all $i > 0$, where $\cO_{T^*X,C}$ (resp. $\cO_{Y,C}$, $\cO_{V,C}$) denotes the localization.  Thus $\Tor^i_{\cO_{T^*X}}(\cO_Y,\cO_V) = 0$ for $i > 0$ and we deduce that $L^iF^*(M_0)= 0$ for $i > 0$.  Since $\pi: X \to Z(L_P)$ is affine, the map $\pi'$ is also affine, so $R\pi'_*(F^*(M_0))$ is concentrated in a single degree.  
\end{proof}

\subsection*{Acknowledgements.} 
Special thanks to Nick Katz for his invaluable help on $G_2$ differential equations, and to
Xinwen Zhu for discussions on~\cite{Zhu:FG-HNY}. It is also a pleasure to thank Sasha Braverman, 
Victor Ginzburg, Dick Gross, Allen Knutson, Andrei 
Okounkov, Konstanze Rietsch, Geordie Williamson and Zhiwei Yun for helpful discussions.  
We thank the referees for many helpful suggestions. TL acknowledges support from the 
Simons Foundation under award number 341949 and from the NSF under agreement No. 
DMS-1464693 and DMS-1953852. NT acknowledges support from the NSF under agreement No. 
DMS-1454893.

\def\cprime{$'$} \def\cprime{$'$} \def\cprime{$'$} \def\cprime{$'$}
  \def\cprime{$'$} \def\cprime{$'$} \def\cprime{$'$} \def\cprime{$'$}
  \def\cprime{$'$} \def\cprime{$'$} \def\cprime{$'$} \def\cprime{$'$}

\end{document}